\documentclass[reqno]{amsart}
\usepackage{amssymb,amsmath,hyperref}
\usepackage{amsrefs, float}
\usepackage[foot]{amsaddr}
\usepackage{bbold,stackrel, multicol}
 
\newcommand{\Z}{{\textsf{\textup{Z}}}}

\newtheorem{thm}{Theorem}
\newtheorem{lem}[thm]{Lemma}
\newtheorem{cor}[thm]{Corollary}
\newtheorem{defi}[thm]{Definition}
\newtheorem{rem}[thm]{Remark}
\newtheorem{nota}[thm]{Notation}

\newtheorem{princ}[thm]{Principle}

\newtheorem{ack}[thm]{Acknowledgement}

\newtheorem{conj}[thm]{Conjecture}
\newtheorem*{tempo*}{Template}

\newtheorem{obs}[thm]{Observation}

\newcommand\be{\begin{equation}}
\newcommand\ee{\end{equation}} 

\usepackage{amsmath,amsfonts} 

\usepackage[applemac]{inputenc}

\def\bdefi{\begin{defi}\rm}
\def\edefi{\end{defi}}
\def\bnota{\begin{nota}\rm}
\def\enota{\end{nota}}
\def\FIVE{\Pi_{1}^{1}\text{-\textup{\textsf{CA}}}_{0}}
\def\SIX{\Pi_{2}^{1}\text{-\textsf{\textup{CA}}}_{0}}

\def\SIXK{\Pi_{k}^{1}\text{-\textsf{\textup{CA}}}_{0}^{\omega}}

\def\ATR{\textup{\textsf{ATR}}}

\def\ZFC{\textup{\textsf{ZFC}}}
\def\ZF{\textup{\textsf{ZF}}}

\def\L{\textsf{\textup{L}}}

 \def\r{\mathbb{r}}

\def\RCA{\textup{\textsf{RCA}}}
\def\({\textup{(}}
\def\){\textup{)}}

\def\RCAo{\textup{\textsf{RCA}}_{0}^{\omega}}
\def\ACAo{\textup{\textsf{ACA}}_{0}^{\omega}}
\def\ATRo{\textup{\textsf{ATR}}_{0}^{\omega}}
\def\WKL{\textup{\textsf{WKL}}}

\def\WWKL{\textup{\textsf{WWKL}}}
\def\bye{\end{document}}
\def\N{{\mathbb  N}}
\def\Q{{\mathbb  Q}}
\def\R{{\mathbb  R}}
\def\A{{\textsf{\textup{A}}}}

\def\SS{\textup{\textsf{S}}}

\def\di{\rightarrow}

\def\CBN{\textup{\textsf{CB}}\N}

\def\asa{\leftrightarrow}

\def\ACA{\textup{\textsf{ACA}}}

\def\QFAC{\textup{\textsf{QF-AC}}}
\def\AC{\textup{\textsf{AC}}}

\def\NCC{\textup{\textsf{NCC}}}

\def\SSEP{\Sigma\textup{\textsf{-SEP}}}
\def\PSEP{\Pi\textup{\textsf{-SEP}}}
\def\pwo{\textup{\textsf{pwo}}}

\def\INDX{\textup{\textsf{IND}}_{1}}
\def\INDY{\textup{\textsf{IND}}_{0}}
\def\cloq{\textup{\textsf{cloq}}}
\def\boco{\textup{\textsf{Boco}}}
\def\PHM{\textup{\textsf{Pohm}}}
\def\Borel{\textup{\textsf{Borel}}}
\def\RUC{\textup{\textsf{RUC}}}
\def\CUC{\textup{\textsf{CUC}}}
\def\BUC{\textup{\textsf{BUC}}}

\def\fun{\textup{\textsf{fun}}}
\def\DCAA{\textup{\textsf{DCA}}}
\def\WSAC{\textup{\textsf{weak-$\Sigma_{1}^{1}$-AC$_{0}$}}}
\def\SAC{\textup{\textsf{$\Sigma_{1}^{1}$-AC$_{0}$}}}

\def\OCO{\textup{\textsf{ccc}}}
\def\ccc{\textup{\textsf{ccc}}}
\def\accu{\textup{\textsf{accu}}}

\def\HAR{\textup{\textsf{Harnack}}}

\def\cocode{\textup{\textsf{cocode}}}

\def\NIN{\textup{\textsf{NIN}}}
\def\NCC{\textup{\textsf{NCC}}}
\def\NBI{\textup{\textsf{NBI}}}

\def\HBC{\textup{\textsf{HBC}}}

\def\closed{\textup{\textsf{closed}}}

\def\CBT{\textup{\textsf{CBT}}}

\def\BOOT{\textup{\textsf{BOOT}}}

\def\IND{\textup{\textsf{IND}}}
\def\NFP{\textup{\textsf{NFP}}}

\def\HBU{\textup{\textsf{HBU}}}
\def\HBT{\textup{\textsf{HBT}}}

\def\range{\textup{\textsf{range}}}

\def\net{\textup{\textsf{net}}}

\def\lex{\textup{\textsf{lex}}}

\def\BW{\textup{\textsf{BW}}}
\def\BWC{\textup{\textsf{BWC}}}

\def\BWL{\textup{\textsf{BWL}}}

\def\seq{\textup{\textsf{seq}}}
\def\fin{\textup{\textsf{fin}}}

\def\LIN{\textup{\textsf{LIN}}}

\def\RM{\textup{\textsf{RM}}}

\def\DCA{\Delta\textup{\textsf{-CA}}}

\def\CWO{\textup{\textsf{CWO}}}

\def\MCT{\textup{\textsf{MCT}}}

\def\PST{\textup{\textsf{PST}}}
\def\eps{\varepsilon}

\def\X{\textup{\textsf{X}}}

\def\ADS{\textup{\textsf{ADS}}}

\def\RT{\textup{\textsf{RT}}}

\def\ECF{\textup{\textsf{ECF}}}

\usepackage{graphicx}
\usepackage{tikz}
\usetikzlibrary{matrix, shapes.misc}
\usepackage{comment,tikz-cd}

\numberwithin{equation}{section}
\numberwithin{thm}{section}

\begin{document}
\title[On robust theorems due to Bolzano, Weierstrass, Jordan, Cantor]{On robust theorems due to Bolzano, Weierstrass, Jordan, and Cantor}
\author{Dag Normann}
\address{Department of Mathematics, The University of Oslo, Norway} 
\email{dnormann@math.uio.no}
\author{Sam Sanders}
\address{Institute for Philosophy II, RUB Bochum, Germany}
\email{sasander@me.com}
\keywords{Countable sets, Bolzano-Weierstrass theorem, Reverse Mathematics, higher-order computability, Kleene S1-S9, bounded variation, regulated functions}
\subjclass[2020]{03B30, 03F35, 03D55, 03D30}
\thanks{\emph{Attributions.} The direct contributions of the first author are mostly in Section \ref{frap}, esp.\ the intricate proofs. The results in Section \ref{bigger} are however partially-but-essentialy based on conceptual ideas due to the first author, the structure functional $\Omega$ pioneered in \cite{dagsamXIII} in particular.  The second author made partial contributions to both sections.}
\begin{abstract}
Reverse Mathematics (RM hereafter) is a program in the foundations of mathematics where the aim is to identify the \emph{minimal} axioms needed to prove a given theorem from ordinary, i.e.\ non-set theoretic, mathematics.
This program has unveiled surprising regularities: the minimal axioms are very often \emph{equivalent} to the theorem over the \emph{base theory}, a weak system of `computable mathematics', while most theorems 
are either provable in this base theory, or equivalent to one of only \emph{four} logical systems.  The latter plus the base theory are called the `Big Five' and the associated equivalences are \emph{robust} following Montalb\'an, i.e.\ stable under small variations of the theorems at hand.  Working in Kohlenbach's \emph{higher-order} RM, we obtain two new and long series of equivalences based on theorems due to Bolzano, Weierstrass, Jordan, and Cantor; these equivalences are extremely robust {and} have no counterpart among the Big Five systems.
Thus, higher-order RM is much richer than its second-order cousin, boasting at least two extra `Big' systems.
\end{abstract}
%

\maketitle
\thispagestyle{empty}



\section{Introduction}
\subsection{Motivation and caveat}
Like Hilbert (\cite{hilbertendlich}), we believe the infinite to be a central object of study in mathematics. 
That the infinite comes in `different sizes' is a relatively new insight, due to Cantor around 1874 (\cite{cantor1}), in the guise of the \emph{uncountability of the real numbers}, also known simply as \emph{Cantor's theorem}. 

\smallskip

With the notion `countable versus uncountable' in place, it is an empirical observation, witnessed by many textbooks, that to show that a set is countable one often 
constructs an injection (or bijection) to $\N$.  
When \emph{given} a countable set, one (additionally) assumes that this set can be \emph{enumerated}, i.e.\ represented by some sequence.  
In this light, implicit in much of mathematical practise is the following most basic principle about countable sets: 
\begin{center}
\emph{a set that can be mapped to $\N$ via an injection \(or bijection\) can be enumerated}.
\end{center}
This principle was studied in \cite{dagsamX, samcount,samNEO2} as part of the study of the uncountability of $\R$. 
In this paper, we continue the study of this principle in \emph{Reverse Mathematics} (RM hereafter) and connect it to well-known `household name' theorems due to Bolzano-Weierstrass, Cantor, Jordan, and Heine-Borel, as discussed in detail in Section~\ref{intro}.  We assume basic familiarity with RM, also sketched in Section \ref{prelim}.   
In particular, working in Kohlenbach's \emph{higher-order} RM, we obtain two new long series of extremely robust equivalences involving the aforementioned theorems.  
In this concrete way, third-order arithmetic is \emph{much} richer than its second-order cousin in that the former boasts (at least) two extra `Big' systems\footnote{A logical system is called `Big' if it boasts many equivalences involving robust principles.} compared to the latter.  

\smallskip

For all the aforementioned reasons, our results provide new answers to one of the driving questions behind RM, formulated as follows by Montalb\'an.  
\begin{quote}
The way I view it, gaining a greater understanding of [the big five] phenomenon is currently one of the driving questions behind reverse mathematics.
To study [this] phenomenon, one distinction that I think is worth making is the one between robust systems and non-robust systems. 
A system is \emph{robust} if it is equivalent to small perturbations of itself. This is not a precise notion yet, but we can still recognize some robust systems. All the big five systems are very robust. For example, most theorems about ordinals, stated in different possible ways, are all equivalent to each other and to $\ATR_{0}$. Apart from those systems, weak weak K\H{o}nig's Lemma $\WWKL_{0}$ is also robust, and we know no more than one or two other systems that may be robust. (\cite{montahue}*{p.\ 432}, emphasis in original)
\end{quote}

Finally, the uncountability of $\R$ deals with arbitrary mappings with domain $\R$ and is therefore best studied in a language that has such objects as first-class citizens. 
Obviousness, much more than beauty, is however in the eye of the beholder.  Lest we be misunderstood, we formulate a blanket caveat: all notions (computation, continuity, function, open set, et cetera) used in this paper are to be interpreted via their higher-order definitions, also listed below, \emph{unless explicitly stated otherwise}.    

\subsection{From Bolzano-Weierstrass to Heine-Borel and Jordan}\label{intro}
In this section, we provide an overview of our results;  in a nutshell, we obtain a large number of robust equivalences involving the \emph{Bolzano-Weierstrass theorem} for countable sets and many theorems concerned with countable sets and related notions.  
We also obtain equivalences for theorems that \emph{do not} involve countable sets in any obvious or direct way at all, namely the \emph{Jordan decomposition theorem} and similar results on functions of bounded variation and related notions.  

\smallskip

First of all, the \emph{Bolzano-Weierstrass theorem} comes in different formulations.  Weierstrass formulates this theorem around 1860 in \cite{weihimself}*{p.~77} as follows, while Bolzano \cite{russke}*{p.~174} states the existence of suprema rather than just limit points.  
\begin{quote}
If a function has a definite property infinitely often within a finite domain, then there is a point such that in any neighbourhood of this point there are infinitely many points with the property.
\end{quote}
We start by studying the Bolzano-Weierstrass theorem for countable sets as in Principle~\ref{bwc}.  
Precise definitions of all notions involved can be found in Section~\ref{somedefi} while motivation for our choice of definitions is provided in Section \ref{crux}. 
\begin{princ}[$\BWC$]\label{bwc}
For a countable set $A\subset 2^{\N}$, the supremum $\sup A$ exists.  
\end{princ}
\noindent
Unless explicitly stated otherwise, the supremum is taken relative to the lexicographic ordering. 
A number of variations $\BWC_{i}^{j}$ of Principle \ref{bwc} are possible, which we shall express via the indicated super- and sub-scripts as follows.
\begin{itemize} 
\item For $i=0$, \emph{countable sets} are defined via \textbf{injections} to $\N$ (Definition \ref{standard}).
\item For $i=1$, we restrict to \emph{strongly countable sets}, which are defined via \textbf{bijections} to $\N$  (Definition \ref{standard}).
\item For $j$ including $\seq$, we additionally have that a \textbf{sequence} $(f_{n})_{n\in \N}$ in $A$ is given with $\lim_{n\di\infty}f_{n}=\sup A$.
\item For $j$ including $\fun$, we additionally have that $\sup_{f\in A}F(f)$ exists for aribitrary \textbf{functionals} $F:2^{\N}\di 2^{\N}$.
\item For $j$ including $\pwo$, the supremum is relative to the pointwise\footnote{The pointwise ordering `$\leq_{1}$' is defined as $f\leq_{1}g\equiv (\forall n\in \N)(f(n)\leq_{\N}g(n))$ for any $f, g\in \N^{\N}$.  The sequence $\sup A$ is the supremum of $A\subset 2^{\N}$ for this ordering if $(\forall f\in A)(f\leq_{1}\sup A)$ and $(\forall k\in \N)((\sup A)(k)=1\di (\exists f\in A)(f(k)=1))$.} ordering.  
\end{itemize}
Since Cantor space with the lexicographic ordering and $[0,1]$ with its usual ordering are intimately connected, we take the former ordering to be fundamental.  
We have shown in \cite{dagsamX} that $\BWC_{0}^{\fun}$ is `explosive' in that it yields the much stronger $\SIX$ when combined with the Suslin functional, i.e.\ higher-order $\FIVE$. 
Previously, metrisation theorems from topology were needed to reach $\SIX$ via $\FIVE$ (\cite{mummyphd,mummy, mummymf}), while Rathjen states in \cite{rathjenICM} that $\SIX$ \emph{dwarfs} $\FIVE$ and Martin-L\"of talks of a \emph{chasm} and \emph{abyss} between these two systems in \cite{loefenlei}. 
Analogous results hold at the level of computability theory, in the sense of Kleene's S1-S9 (\cite{kleeneS1S9}), while we even obtain $\exists^{3}$, and hence full second-order arithmetic, if we assume \textsf{V=L}, by \cite{dagsamX}*{Theorem 4.6}.  Thus, the following natural questions arise.
\begin{enumerate}
\item[(Q0)] Is the `extra information' as in `$\fun$' or `$\seq$' necessary for explosions? 
\item[(Q1)] Is it possible to `split' e.g.\ $\BWC_{0}$ in `less explosive' components?  
\item[(Q2)] Since $\BWC_{0}$ is formulated using injections, is there an equivalent formulation only based on \emph{bijections}?
\item[(Q3)] Is the explosive nature of $\BWC_{0}$ caused by the use of injections or bijections?
\item[(Q4)] Are there equivalences involving $\BWC_{0}$ from ordinary mathematics, especially involving theorems not related to countability in any obvious way? 
\end{enumerate}
Secondly, to answer (Q0), we connect $\BWC_{0}$ to the other variations $\BWC_{i}^{j}$, as part of Kohlenbach's \emph{higher-order Reverse Mathematics}, briefly introduced in Section~\ref{prelim}.  We assume basic familiarity with Reverse Mathematics (RM hereafter), to which \cite{stillebron} provides an introduction.
We establish the series of equivalences in \eqref{hong} in Section \ref{WEC}, where $\IND_{i}$ are fragments of the induction axiom.  
\be\label{hong}\tag{\textsf{EQ}}
\begin{matrix}
\BWC_{0}^{\fun}\asa \BWC_{0}^{\seq}\asa \BWC_{0}^{\pwo}\asa [\BWC_{0}+\IND_{0}]\asa \BWC_{0}^{\fun, \pwo}. \\~\\
\vspace{-1mm}
\cocode_{1}\asa \DCA_{C}^{-}\asa\BW_{1}^{\seq}\asa [\BW_{1}+\INDX]\asa \BW_{1}^{\pwo}.
\end{matrix}
\ee
Here, $\DCA_{C}^{-}$ is a peculiar axiom inspired by $\Delta_{1}^{0}$-comprehension while $\cocode_{1}$ expresses that \emph{strongly countable} sets, i.e.\ boasting bijections to $\N$, can be enumerated.  
We point out that $\BWC_{0}\asa \BWC_{0}^{\seq}$ is interesting as follows:  to obtain the extra sequence in the latter, the only method\footnote{Apply countable choice to $(\forall n\in \N)(\exists f\in A)( d(f, \sup A)<\frac{1}{2^{n}})$ which holds by definition.} seems to use countable choice, while
the equivalence is provable \emph{without} the latter.  Thus, the extra sequence from $\BWC_{0}^{\seq}$, while seemingly a choice function, can be defined explicitly in terms of the other data, i.e.\ without the Axiom of Choice.  
By Remark \ref{hypersect}, the second line of \eqref{hong} is connected to \emph{hyperarithmetical analysis}. 

\smallskip

Thirdly, in answer to (Q3), the principles from \eqref{hong} are formulated using injections and bijections to $\N$, while items \eqref{ka}-\eqref{put} below are basic theorems about the real line $\R$ 
based on \emph{enumerable} sets, i.e.\ listed by (possibly infinite) sequences, which is essentially the notion of countable set used in second-order RM:
\begin{enumerate}
 \renewcommand{\theenumi}{\alph{enumi}}
\item $\accu$: a non-enumerable \textbf{closed} set in $\R$ {has} a limit point,\label{ka}
\item $\accu'$: a non-enumerable set in $\R$ {contains} a limit point,
\item $\ccc$: a collection of disjoint open intervals in $\R$ is enumerable. \label{put}
\item $\cloq$: a countable linear ordering is order-isomorphic to a subset of $\Q$. \label{vag}
\end{enumerate}
Closed sets are defined as in Definition \ref{openset}, which generalises the second-order notion (\cite{simpson2}*{II.5.6}).  
The principles $\ccc_{i}$ and $\accu_{i}$ for $i=0,1$ are defined as for $\BWC_{i}$ above.  We establish the following series of implications in Section \ref{flum}.
\be\label{hong2}\tag{\textsf{EQ2}}
\begin{matrix}
\accu\asa \accu'\asa \ccc\asa \BW_{0}\asa [\CBN+\BW_{1}]\asa [\CWO^{\omega}+\IND_{0}].  \\~\\
\vspace{-1mm}
\ccc_{1}\asa \CBN\asa \accu_{1}, \textup{ and } \cocode_{0}\asa [\cloq +\IND_{0}]\asa [\cloq'+\IND_{0}].
\end{matrix}
\ee
Here, $\CBN$ is the \emph{Cantor-Bernstein theorem} for $\N$ as in Principle~\ref{cbn}, which is \emph{independent} of $\BWC_{1}$ by Theorem~\ref{goeddachs}, thus answering (Q2).
The principle $\CWO^{\omega}$ expresses that countable well-orderings are comparable, while $\cloq'$ is Cantor's theorem characterising the order type $\eta$ of $\Q$.  
The notion of \emph{limit point} goes back to Cantor (\cite{cantor33}*{p.\ 98}) in 1872; he also proved the first instance of the \emph{countable chain condition} $\ccc$ in \cite{cantor33}*{p.\ 161} and introduced order types, including $\eta$, in \cite{cantorbook90,cantorm}.

\smallskip

Fourth, following (Q4), we also study $\BWC_{0}$ and $\BWC_{1}$ in the grand(er) scheme of things, namely how they connect to set theory and ordinary mathematics.   
In Section~\ref{heli2}, we obtain equivalences between $\BWC_{0}$ and $\BWC_{1}$, and fragments of the well-known \emph{countable union theorem} from set theory (see e.g.\ \cite{heerlijkheid}*{\S3.1}).
As to ordinary mathematics, in Section~\ref{heli}, we establish equivalences between $\BWC_{0}$ and versions of the \emph{Lindel\"of lemma} and \emph{Heine-Borel theorem} as studied in \cites{dagsamX, dagsamIII}. 
In Section \ref{BV}, we establish equivalences between $\BWC_{0}$, the \emph{Jordan decomposition theorem}, and related results from \cites{dagsamXII, dagsamXIII}.  The latter theorem and its ilk have no obvious or direct connection to countability \emph{at all}.  

\smallskip

Finally, we discuss how these results provide detailed answers to (Q0)-(Q4) in the below sections.  In light of all the aforementioned equivalences, we believe 
the following quote by Friedman to be apt:
\begin{quote}
When a theorem is proved from the right axioms, the axioms can be proved from the theorem. (\cite{fried})
\end{quote}
Next, Section \ref{krelim} details the definitions used in this paper while a neat motivation for our choice of definitions is provided in Section \ref{crux}, with the gift of hindsight. 

\subsection{Preliminaries and definitions}\label{krelim}
We briefly introduce \emph{Reverse Mathematics} and \emph{higher-order computability theory} in Section \ref{prelim}.
We introduce some essential definitions in Section \ref{somedefi}.  A full introduction may be found in e.g.\ \cite{dagsamX}*{\S2}.
In Section \ref{crux}, we motivate our choice of definitions, Definition \ref{openset} in particular.  
\subsubsection{Reverse Mathematics and higher-order computability theory}\label{prelim}
Reverse Mathematics (RM hereafter) is a program in the foundations of mathematics initiated around 1975 by Friedman (\cites{fried,fried2}) and developed extensively by Simpson (\cite{simpson2}).  
The aim of RM is to identify the minimal axioms needed to prove theorems of ordinary, i.e.\ non-set theoretical, mathematics. 

\smallskip

We refer to \cite{stillebron} for a basic introduction to RM and to \cite{simpson2, simpson1} for an overview of RM.  We expect basic familiarity with RM, in particular Kohlenbach's \emph{higher-order} RM (\cite{kohlenbach2}) essential to this paper, including the base theory $\RCAo$.   An extensive introduction can be found in e.g.\ \cites{dagsamIII, dagsamV, dagsamX}.  
We have chosen to include a brief introduction as a technical appendix, namely Section \ref{RMA}.  
All undefined notions may be found in the latter. 

\smallskip

Next, some of our main results will be proved using techniques from computability theory.
Thus, we first make our notion of `computability' precise as follows.  
\begin{enumerate}
\item[(I)] We adopt $\ZFC$, i.e.\ Zermelo-Fraenkel set theory with the Axiom of Choice, as the official metatheory for all results, unless explicitly stated otherwise.
\item[(II)] We adopt Kleene's notion of \emph{higher-order computation} as given by his nine clauses S1-S9 (see \cite{longmann}*{Ch.\ 5} or \cite{kleeneS1S9}) as our official notion of `computable'.
\end{enumerate}
We refer to \cite{longmann} for a thorough overview of higher-order computability theory.
%
%
%
%
%
%

\subsubsection{Some definitions in higher-order arithmetic}\label{somedefi}
We introduce the standard definitions for countable set and related notions. 

\smallskip

First of all, the main topic of \cite{dagsamX} is the logical and computational properties of \emph{the uncountability of $\R$}, established in 1874 by Cantor in his \emph{first} set theory paper \cite{cantor1}, in the guise of the following natural principles:
\begin{itemize}
\item $\NIN$: \emph{there is no injection from $[0,1]$ to $\N$},
\item  $\NBI$: \emph{there is no bijection from $[0,1]$ to $\N$}.
\end{itemize}
As it happens, $\NIN$ and $\NBI$ are among the weakest principles that require a lot of comprehension for a proof.  
An overview may be found in \cite{dagsamX}*{Figure 1}.

\smallskip

Secondly, we shall make use of the following notion of (open) set, which was studied in detail in \cite{dagsamVII, samcount, dagsamX}.
We motivate this choice in detail in Section \ref{crux}.  
\bdefi[Sets in $\RCAo$]\label{openset}
We let $Y: \R \di \R$ represent subsets of $\R$ as follows: we write `$x \in Y$' for `$Y(x)>_{\R}0$' and call a set $Y\subseteq \R$ Ôopen' if for every $x \in Y$, there is an open ball $B(x, r) \subset Y$ with $r^{0}>0$.  
A set $Y$ is called `closed' if the complement, denoted $Y^{c}=\{x\in \R: x\not \in Y \}$, is open. 
\edefi
Note that for open $Y$ as in the previous definition, the formula `$x\in Y$' has the same complexity (modulo higher types) as in second-order RM (see \cite{simpson2}*{II.5.6}), while given $(\exists^{2})$ from Section \ref{lll} the former becomes a `proper' characteristic function, only taking values `0' and `$1$'.  Hereafter, an `(open) set' refers to Definition~\ref{openset}, while `RM-open set' refers to the second-order definition from RM.  

\smallskip

The attentive reader has of course noted that e.g.\ the unit interval is only a set in the sense of Definition \ref{openset} in case we assume $\ACAo\equiv \RCAo+(\exists^{2})$.  For this reason, we shall adopt the latter as our base theory in our paper.  We discuss how the reader may obtain equivalences over $\RCAo$ in Remark \ref{LEM}.  

\smallskip

Thirdly, the notion of `countable set' can be formalised in various ways, namely via Definitions \ref{eni} and \ref{standard}.
\bdefi[Enumerable sets of reals]\label{eni}
A set $A\subset \R$ is \emph{enumerable} if there exists a sequence $(x_{n})_{n\in \N}$ such that $(\forall x\in \R)(x\in A\di (\exists n\in \N)(x=_{\R}x_{n}))$.  
\edefi
This definition reflects the RM-notion of `countable set' from \cite{simpson2}*{V.4.2}.  
We note that given $\mu^{2}$ from Section \ref{lll}, we may replace the final implication in Definition \ref{eni} by an equivalence. 

\smallskip

The usual definition of `countable set' is as follows in $\RCAo$. 
\bdefi[Countable subset of $\R$]\label{standard}~
A set $A\subset \R$ is \emph{countable} if there exists $Y:\R\di \N$ such that $(\forall x, y\in A)(Y(x)=_{0}Y(y)\di x=_{\R}y)$. 
If $Y:\R\di \N$ is also \emph{surjective}, i.e.\ $(\forall n\in \N)(\exists x\in A)(Y(x)=n)$, we call $A$ \emph{strongly countable}.
\edefi
The first part of Definition \ref{standard} is from Kunen's set theory textbook (\cite{kunen}*{p.~63}) and the second part is taken from Hrbacek-Jech's set theory textbook \cite{hrbacekjech} (where the term `countable' is used instead of `strongly countable').  
For the rest of this paper, `strongly countable' and `countable' shall exclusively refer to Definition \ref{standard}, \emph{except when explicitly stated otherwise}. 

\smallskip

Finally, we shall use the following definition of finite and infinite set.  
\bdefi[Finite and infinite sets]\label{plonk}
A set $A\subset \R$ is called \emph{infinite} if 
\[
(\forall n\in \N)(\exists w^{1^{*}})[ |w|\geq n \wedge  (\forall i, j<|w|)(w(i), w(j)\in A  \wedge (i\ne j \di w(i)\ne w(j)) )], 
\]
i.e.\ there are arbitrary long finite sequences of distinct elements in $A$.  A set $A\subset \R$ is \emph{finite} if it is not infinite.   
\edefi
The exact definition of (in)finite set plays a \emph{minor} role in most of this paper, but a \emph{major} role in the study of the Jordan decomposition theorem and related topics in Section \ref{BV}.
This observation is explained at length in Remark \ref{diunk}.  In a nutshell, the notion of finite set as in Definition \ref{plonk} is suitable for the RM-study of functions of bounded variation, whereas the `usual' definitions of finite set, involving injections or bijections to $\N$, are not.  
\subsubsection{Some axioms of higher-order arithmetic}\label{lll2}
We introduce a number of axioms of higher-order arithmetic, including the `higher-order counterparts' of $\WKL_{0}$ and $\ACA_{0}$.
We motivate the latter term in detail based on Remark \ref{ECF}.

\smallskip

First of all, with Definitions \ref{openset} and \ref{standard} in place, the following principle has interesting properties, as studied in \cites{samcount, dagsamX, samNEO2}.
\begin{princ}[$\cocode_{0}$]
For any non-empty countable set $A\subseteq [0,1]$, there is a sequence $(x_{n})_{n\in \N}$ in $A$ such that $(\forall x\in \R)(x\in A\asa (\exists n\in \N)(x_{n}=_{\R}x))$.
\end{princ}
Indeed, as explored in \cites{dagsamX, samcount}, we have $\cocode_{0}\asa \BWC_{0}^{\fun}$ over $\ACAo$, while another interesting equivalence is based on the `projection' axiom studied in \cite{samph}:
\be\tag{$\BOOT$}
(\forall Y^{2})(\exists X\subset \N)\big(\forall n\in \N)(n\in X\asa (\exists f \in \N^{\N})(Y(f, n)=0)\big).
\ee
We mention that $\BOOT$ is equivalent to e.g.\ the monotone convergence theorem for nets indexed by Baire space (see \cite{samph}*{\S3}), while it is essentially Feferman's \textsf{(Proj1)} from \cite{littlefef} without set parameters. 
The axiom $\BOOT^{-}$ results from restricting $\BOOT$ to functionals $Y$ with the following `at most one' condition: 
\be\label{uneek}
(\forall n\in \N)(\exists \textup{ at most one } f \in \N^{\N})(Y(f, n)=0), 
\ee
where similar constructs appear in the RM of $\ATR_{0}$ by \cite{simpson2}*{V.5.2}.
The weaker $\BOOT^{-}$ appears prominently in the RM-study of open sets given as (third-order) characteristic functions (\cite{dagsamVII}).
In turn, $\BOOT_{C}^{-}$ is $\BOOT^{-}$ with `$\N^{\N}$' replaced by `$2^{\N}$' everywhere; $\BOOT^{-}_{C}$ was introduced in \cite{dagsamX}*{\S3.1} in the study of $\BWC_{0}^{\fun}$, and we have 
$\BOOT_{C}^{-}\asa \BWC_{0}^{\fun}$ over $\RCAo$ by \cite{samcount}*{Theorem 3.12}.  
In light of \cite{simpson2}*{V.5.2}, $\ACAo+\BOOT_{C}^{-}$ proves $\ATR_{0}$. 

\smallskip

Secondly, related to $\BOOT^{-}_{C}$ and $\cocode_{0}$ is the following principle.  
\begin{princ}[$\range_{0}$] For $Y:2^{\N}\di \N$ an injection on $A\subset 2^{\N}$, we have
\[
(\exists X\subset \N)(\forall n\in \N)\big[(\exists f\in A)(Y(f)=n)\asa n\in X \big],
\]
i.e.\ the range of $Y$ restricted to $A$ exists. 
\end{princ}
\noindent
With the gift of hindsight\footnote{A countable $A\subset \R$ yields a linear order via $x\preceq y\equiv Y(x)\leq Y(y)$, where $Y$ is injective on $A$.} from \cites{samcount, samNEO2,dagsamX}, we see that $\cocode_{0}$ is equivalent to:
\be\label{locutus}
\text{\emph{a linear order $(A, \preceq_{A})$ for countable $A\subset \R$ can be enumerated.}}
\ee
In second-order RM, countable linear orders are represented by sequences (see \cite{simpson2}*{V.1.1}), i.e.\ the previous principle
seems essential if one wants to interpret theorems about countable linear orders in higher-order arithmetic or set theory. 
Another useful fragment of $\BOOT$ is $\DCA$, which is central to `lifting' second-order reversals to higher-order arithmetic (see \cite{samFLO2, samrecount}).  
\begin{princ}[$\DCA$]\label{DAAS}
For $i=0, 1$, $Y_{i}^{2}$, and $A_i(n)\equiv (\exists f \in \N^\N)(Y_{i}(f,n)=0)$\textup{:}
\[
(\forall n\in \N)(A_0(n) \asa \neg A_1(n))\di (\exists X\subset \N)(\forall n\in \N)(n\in X\asa A_{0}(n)).
\]
\end{princ}
This principle borrows its name from the fact that the $\ECF$-translation (see Remark \ref{ECF}) converts $\DCA$ into $\Delta_{1}^{0}$-comprehension.  
As will become clear below, $\DCA$ with the `at most one' condition \eqref{uneek} plays an important role in the RM of the Bolzano-Weierstrass theorem. 

\smallskip

Thirdly, the Heine-Borel theorem states the existence of a finite sub-covering for an open covering of certain spaces. 
Now, a functional $\Psi:\R\di \R^{+}$ gives rise to the \emph{canonical covering} $\cup_{x\in I} I_{x}^{\Psi}$ for $I\equiv [0,1]$, where $I_{x}^{\Psi}$ is the open interval $(x-\Psi(x), x+\Psi(x))$.  
Hence, the uncountable covering $\cup_{x\in I} I_{x}^{\Psi}$ has a finite sub-covering by the Heine-Borel theorem, which yields the following principle.
\begin{princ}[$\HBU$]
$(\forall \Psi:\R\di \R^{+})(\exists  y_{0}, \dots, y_{k}\in I){(\forall x\in I)}(\exists i\leq k)(x\in I_{y_{i}}^{\Psi}).$
\end{princ}
Note that $\HBU$ is essentially \emph{Cousin's lemma} (see \cite{cousin1}*{p.\ 22}), i.e.\ the Heine-Borel theorem for canonical coverings.  
By \cite{dagsamIII, dagsamV}, $\Z_{2}^{\Omega}$ proves $\HBU$, but $\Z_{2}^{\omega}+\QFAC^{0,1}$ cannot.
Basic properties of the \emph{gauge integral} (\cite{zwette, mullingitover}) are equivalent to $\HBU$.
By \cite{dagsamIII}*{Theorem 3.3}, $\HBU$ is equivalent to the same compactness property for $2^{\N}$.
\begin{princ}[$\HBU_{\textsf{c}}$]
$(\forall G^{2})(\exists  f_{1}, \dots, f_{k} \in 2^{\N} ){(\forall f\in 2^{\N})}(\exists i\leq k)(f\in [\overline{f_{i}}G(f_{i})]).$
\end{princ}
As studied in \cite{sahotop}*{\S3.1}, canonical coverings as in $\HBU$ are not suitable for the study of basic topological notions like paracompactness and dimension.   
This suggests the need for a more general notion of covering; the solution adopted in \cite{sahotop} it to allow $\psi:I\di \R$, i.e.\ $I_{x}^{\psi}$ is empty in case $\psi(x)\leq 0$. 
In this way, we say that `$\cup_{x\in I}I_{x}^{\psi}$ covers $[0,1]$' if $(\forall x\in [0,1])(\exists y\in [0,1])(x\in I_{y}^{\psi})$.  
Thus, we obtain the Heine-Borel theorem as in $\HBT$, going back to Lebesgue in 1898 (see \cite{lesje}*{p.\ 133}).
\begin{princ}[$\HBT$]
For $ \psi:[0,1]\di \R$, , if $\cup_{x\in I}I_{x}^{\psi}$ covers $[0,1]$, then there are  $y_{1}, \dots, y_{k} \in [0,1]$ such that $\cup_{i\leq k}I_{y_{i}}^{\psi}$ covers $[0,1]$.
\end{princ}
As shown in \cite{sahotop}*{\S3}, we have $\HBU\asa \HBT$ over various natural base theories, some of which we shall discuss and use in Section \ref{link2}.
%

\smallskip

Finally, as discussed in detail in \cite{kohlenbach2}*{\S2}, the base theories $\RCAo$ and $\RCA_{0}$ prove the same $\L_{2}$-sentences `up to language' as the latter is set-based (the $\L_{2}$-language) and the former function-based (the $\L_{\omega}$-language).   
Here, $\L_{2}$ is the language of second-order arithmetic, while $\L_{\omega}$ is the language of all finite types. 
This conservation result is obtained via the so-called $\ECF$-interpretation, discussed next. 
\begin{rem}[The $\ECF$-interpretation]\label{ECF}\rm
The (rather) technical definition of $\ECF$ may be found in \cite{troelstra1}*{p.\ 138, \S2.6}.
Intuitively, the $\ECF$-interpretation $[A]_{\ECF}$ of a formula $A\in \L_{\omega}$ is just $A$ with all variables 
of type two and higher replaced by type one variables ranging over so-called `associates' or `RM-codes' (see \cite{kohlenbach4}*{\S4}); the latter are (countable) representations of continuous functionals.  
The $\ECF$-interpretation connects $\RCAo$ and $\RCA_{0}$ (see \cite{kohlenbach2}*{Prop.\ 3.1}) in that if $\RCAo$ proves $A$, then $\RCA_{0}$ proves $[A]_{\ECF}$, again `up to language', as $\RCA_{0}$ is 
formulated using sets, and $[A]_{\ECF}$ is formulated using types, i.e.\ using type zero and one objects.  

\smallskip

In light of the widespread use of codes in RM and the common practise of identifying codes with the objects being coded, it is no exaggeration to refer to $\ECF$ as the \emph{canonical} embedding of higher-order into second-order arithmetic.  Moreover, $\RCAo+\BOOT$ is called the `higher-order counterpart' of $\ACA_{0}$ as the former is a conservative extension of the latter, and $\ECF$ maps $\BOOT$ to $\ACA_{0}$.  Similarly, $\RCAo+\HBT$ is the `higher-order counterpart' of $\WKL_{0}$. 
\end{rem}
As a neat application of the $\ECF$-interpretation, Remark \ref{blafte} establishes that the Jordan decomposition theorem (see Section \ref{deffer}) does not imply $(\exists^{2})$, although the former theorem applies to discontinuous functions.

\section{Equivalences for the Bolzano-Weierstrass theorem}\label{frap}
\subsection{Introduction}\label{frapintro}
We establish the results sketched in Section \ref{intro} and \eqref{hong}.  

\smallskip

In Section \ref{dix}, we establish the equivalence between $\cocode_{1}$ and the Bolzano-Weierstrass theorem for \emph{strongly} countable sets in Cantor space in various guises, including $\BWC_{1}$.
In Section \ref{diff}, we do the same for $\cocode_{0}$ and $\BWC_{0}$ and variations.  
In Section \ref{CBN}, we study $\CBN$, the Cantor-Berstein theorem for $\N$, and show that it is strictly weaker than $\BWC_{0}$ in that $\Z_{2}^{\omega}+\CBN$ cannot even prove $\NBI$.  
In Section \ref{flum}, we study items \eqref{ka}-\eqref{vag} from Section \ref{intro},   
which are basic theorems about limit points in $\R$ and related concepts, all going back to Cantor somehow.  We establish equivalences between versions of these items on one hand, and $\CBN$ and $\cocode_{0}$ on the other hand; unlike the latter, items \eqref{ka}-\eqref{put} do not mention `injections' or `bijections'.

\smallskip

As to technical machinery, we mention the `excluded middle trick' pioneered in \cite{dagsamV}.  
While we adopt $\ACAo$ as our base theory, the following trick can be used to replace the latter theory by $\RCAo$ \emph{if the reader so desires}.
\begin{rem}[Excluded middle trick]\label{LEM}\rm
The law of excluded middle as in $(\exists^{2})\vee \neg(\exists^{2})$ is quite useful as follows:  suppose we are proving $T\di \cocode_{0}$ over $\RCAo$.  
Now, in case $\neg(\exists^{2})$, all functions on $\R$ are continuous by \cite{kohlenbach2}*{\S3} and $\cocode_{0}$ trivially
holds.  Hence, what remains is to establish $T\di \cocode_{0}$ 
\emph{in case we have} $(\exists^{2})$.  However, the latter axiom e.g.\ implies $\ACA_{0}$ and can uniformly convert reals to their binary representations.  
In this way, finding a proof in $\RCAo+(\exists^{2})$ is `much easier' than finding a proof in $\RCAo$.
In a nutshell, we may wlog assume $(\exists^{2})$ when proving theorems that are trivial (or readily proved) when all functions (on $\R$ or $\N^{\N }$) are continuous, like $\cocode_{0}$.   
\end{rem}
We stress that the previous trick should be used sparingly: the unit interval is not a set in the sense of Definition \ref{openset} in the absence of $(\exists^{2})$.

\smallskip

In addition to the previous remark, we shall need a coding trick based on the well-known lexicographic ordering $<_{\lex}$, as described in Notation \ref{fluk}. 
For brevity, we sometimes abbreviate $\langle n\rangle *w^{0^{*}}*f^1$ as $nwf$ if all types are clear from context.
\begin{nota}[Sequences with information]\label{fluk}\rm
For a finite binary sequence $s^{0^{*}}$, define $w_{s}$ by replacing $0$ in $s$ with the word $1001$ and $1$ in $s$ with $101$. Conversely, if $w^{0^*}$ is a finite conjunction of words $1001$ and $101$, we let $s_w$ be the finite binary  sequence $s$ such that $w_{s_w} =_{0^*}w$. This coding and decoding transfers directly to infinite binary sequences and infinite conjunctions of the words $1001$ and $101$.  
A \emph{sequence with information} is any coded presentation $g = w_s0f$ of a pair $(s,f)$ where $s^{0^{*}}$ is a finite binary sequence and $f\in 2^{\N}$.
\end{nota}
This notation is convenient when trying to define the set $X$ of binary sequences $s^{0^{*}}$ such that $ (\exists f \in 2^{\N}) [Y(s,f)=0]$ for some fixed $Y^{2}$.    
Indeed, one point is that the coding as in Notation \ref{fluk} preserves the lexicographic ordering of the sequences. Another point is that if $s_1$ is a strict subsequence of $s_2$, and $w_{s_1}0f_1$ and $w_{s_2}0f_2$ are two sequences with information, then $w_{s_1}0f_1 < _{\lex}w_{s_2}0f_2$.  In this way, the above versions of the Bolzano-Weierstrass are applied to sets of sequences with information in such a way that the information parts do not show up in the supremum.

\subsection{Bolzano-Weierstrass theorem and (strongly) countable sets}\label{WEC}
In this section, we study the RM of the Bolzano-Weierstrass theorem in the guise of $\BWC_{i}^{j}$ from Section \ref{intro}.
In particular, we provide a positive answer to question (Q0) from the latter by establishing the equivalences in \eqref{hong}.

\subsubsection{Strongly countable sets}\label{dix}
We connect the Bolzano-Weierstrass theorem for strongly countable sets to $\cocode_{1}$, which is $\cocode_{0}$ restricted to strongly countable sets. 
We discuss the connection to hyperarithmetical analysis in Remark~\ref{hypersect}.

\smallskip

First of all, we need a little bit of the induction axiom, formulated as in $\INDX$ in Principle \ref{IX}.
The equivalence between induction and bounded comprehension is well-known in second-order RM (\cite{simpson2}*{X.4.4}).
\begin{princ}[$\IND_{1}$]\label{IX}
Let $Y^{2}$ satisfy $(\forall n \in \N) (\exists !f \in 2^{\N})[Y(n,f)=0]$.  Then $ (\forall n\in \N)(\exists w^{1^{*}})\big[ |w|=n\wedge  (\forall i < n)(Y(i,w(i))=0)\big]$.  
\end{princ}
Note that $\IND_{1}$ is a special case of the axiom of finite choice, and is valid in all models considered in \cites{dagsam, dagsamII, dagsamIII, dagsamV, dagsamVI, dagsamVII, dagsamIX, dagsamX}. 
Moreover, $\IND_{1}$ is trivial in case $\neg( \exists^2)$ since the condition on $Y$ is then false.  
%
\begin{lem} \label{wonk}
The system $\ACAo$ proves $\cocode_{1} \rightarrow \IND_{1}$.
\end{lem}
\begin{proof} 
To show that $\cocode_{1} \rightarrow \IND_{1}$, assume $(\forall n \in \N) (\exists !f \in 2^\N)A_{0}(n,f)$ where $A_{0}$ is quantifier-free. Let $\langle n\rangle* f\in A$ if $A_{0}(n,f)$ and define $F(g) := g(0)$, i.e. $F(\langle n\rangle*f) = n$. Modulo coding, we may view $A$ as a subset of $2^{\N}$.
By assumption, $F$ is a bijection from $A$ to $\N$, and by $\cocode_{1}$, $A$ is enumerable as $\{g_i\}_{i \in \N}$. 
From this enumeration, we can (Turing) compute $n \mapsto f_n$ where $f_n$ is the unique $f$ with $A_{0}(n,f)$ for any $n\in \N$, and in particular an object as claimed to exist by $\IND_{1}$.\end{proof}
Secondly, the following theorem completes most of the results in \eqref{hong} for $\BWC_{1}$.  
\begin{thm}[$\ACAo$]\label{wonkja}
$[\BWC_{1}+ \INDX]\asa\BWC_{1}^{\pwo}\asa\cocode_{1}$. 
\end{thm}
\begin{proof} We have already established that $\cocode_{1} \rightarrow \IND_{1}$ in Lemma \ref{wonk}.
Moreover, it is straightforward to prove both $\BWC_{1}^{\pwo}$ and $\BWC_{1}$ from $\cocode_{1}$.
We first prove that $\BWC_{1}^{\pwo} \rightarrow \cocode_{1}$ in $\RCAo$. To this end, let $F:2^{\N}\di \N$ be a bijection on $A \subseteq 2^\N$. Define the set $B\subset 2^{\N}$ as follows: $g \in B$ if the 
following items are satisfied: 
\begin{itemize}
\item for all $n, m, a, b\in \N$, $g(\langle n,a\rangle) = g(\langle m , b\rangle) = 1\di n = m$,
\item for a unique $n_{0}\in \N$, $g(\langle n_{0} , 0\rangle) = 1$,
\item for this $n_{0}$, the function $\lambda a.g(\langle n_{0} , a+1\rangle) $ is in $ A$ and maps to $n_{0}$ under $F$.
\end{itemize}
Clearly, $B$ is strongly countable and $\BWC_{1}^{\pwo}$ yields a pointwise least upper bound for $B$. 
This is essentially the characteristic function of the disjoint union of the sets (with characteristic functions) in $A$, and we can recover an enumeration of $A$.



\smallskip

Next, we prove that $\BWC_{1} \rightarrow \cocode_{1}$, using $\INDX$. 
Let $F$ be bijective on $A\subset 2^{\N}$. We will construct a strongly countable set $B$ such that $g(\langle i,j\rangle) = F^{-1}(i)(j)$ is coded as the lexicographic supremum of $B$. Let $w0f \in B$ if $f = g_{0}\oplus g_{1}\oplus \dots \oplus g_{k-1}$ where $k$ is the length of $s_w$, where $F(g_i) = i$ for $i < k$, and where $s_w(\langle i,j\rangle) = g_i(j)$ whenever $\langle i,j\rangle < k$. We  let $G(w0f)$ be the length of $s_w$. 
Then $G$ is a bijection on $B$. We need $\IND_{1}$ to establish the unique existence of $g_{0}\oplus g_{1}\oplus \dots \oplus g_{k-1}$ for each $k$ for this otherwise trivial fact. 
The supremum of $B$ in the lexicographic ordering now codes the enumeration of $A$ 
via the inverse of $F$ and the $0 \mapsto 1001$ and $1 \mapsto 101$ translation from Notation \ref{fluk}.
\end{proof}

\smallskip

Thirdly, by the following, $\ACAo+\BWC_{1}^{\pwo}$ and $\ACAo+\BWC_{1}+\IND_{1}$ are connected to \emph{hyperarithmetical analysis}. 
We discuss this connection in Remark \ref{hypersect}
\begin{cor}\label{hyp}
The system $\ACAo+\BWC_{1}^{\pwo}$ proves $\WSAC$; the former yields a conservative extension when added to $\SAC$.
\end{cor}
\begin{proof}
By \cite{samcount}*{Theorem 3.17}, $\QFAC^{0,1}\di \cocode_{1}\di !\QFAC^{0,1}$, where the final principle is the first principle with a uniqueness condition.  
Now, $\ACAo+\QFAC^{0,1}$ is a conservative extension of $\SAC$ by \cite{hunterphd}*{Cor.\ 2.7}, while $\ACAo+!\QFAC^{0,1}$ clearly proves $\WSAC$
\end{proof}
%
We note that the monotone convergence theorem for nets with \emph{strongly countable} index set, called $\MCT_{1}^{\net}$ in \cite{samcount}, is equivalent to $\cocode_{1}$ over $\RCAo$ by \cite{samcount}*{Theorem~3.12}.
Hence, this theorem has the same status as e.g.\ $\BWC_{1}+\INDX$. 

\smallskip

Finally, the previous results suggest a connection between $\cocode_{1}$ and hyperarithmetical analysis.  
A well-known system here is $\Delta_{1}^{1}$-comprehension (see \cite{simpson2}*{Table 4, p.\ 54}) and we now connect the latter to $\cocode_{1}$.  
%
%
To this end, let $\DCA^{-}_{C}$ be $\DCA$ restricted to formulas $A_i(n)\equiv (\exists f \in 2^\N)(Y_{i}(f,n)=0)$ also satisfying $(\forall n \in \N)(\exists \textup{ at most one } f\in 2^{\N})(Y_{i}(f, n)=0)$ for $i=0,1$.
In this way, $\DCA_{C}^{-}$ is similar in role and form to $\BOOT_{C}^{-}$.  
We have the following surprising result. 
\begin{thm}\label{aars}
The system $\ACAo$ proves that the following are equivalent:
\begin{enumerate}
 \renewcommand{\theenumi}{\alph{enumi}}
\item $\cocode_{1}$: any strongly countable set can be enumerated,\label{GF0}
\item For strongly countable $A\subset [0,1]$, any subset of $A$ can be enumerated,\label{GF}
\item $\DCA_{C}^{-}$: the axiom $\DCA$ with an `at most one' condition for $2^{\N}$. \label{GF1}
\end{enumerate}
\end{thm}
\begin{proof}
For the implication $\eqref{GF0}\di \eqref{GF}$, let $A\subset [0,1]$ be strongly countable and use $\cocode_{1}$ to obtain a sequence listing all elements of $A$.  
For $B\subset A$, use $\mu^{2}$ to remove all elements in $A\setminus B$ from this sequence.  
For $\eqref{GF}\di \eqref{GF1}$, fix $Y_{i}^{2}$ for $i=0,1$ as in $\DCA_{C}^{-}$ and define the following subsets of Cantor space:
\[
A:= \{f\in 2^{\N}: (\exists n\in \N)(Y_{0}(f, n)=0)  \} \textup{ and } B:= \{g\in 2^{\N}: (\exists m\in \N)(Y_{1}(g, m)=0)  \}.
\]
Define $Z, W:2^{\N}\di \N$ as $Z(f):= (\mu n)(Y_{0} (f, n)=0)$ and $W(g):= (\mu m)(Y_{1} (g, m)=0)$.
By the assumption on $Y_{0}$ (resp.\ $Y_{1}$), $Z$ (resp.\ $W$) is injective on $A$ (resp.\ $B$).
Now let $A~{\dot{\cup}}~B$ be the disjoint\footnote{The disjoint union $A~{\dot{\cup}}~B$ can be defined as $\{ (\langle n\rangle *f) \in 2^{\N} : (n=0 \wedge f\in A) \vee (n=1 \wedge f\in B)   \}$.} union of $A$ and $B$ and define the following:
\be\label{fool}
V(h):=
\begin{cases}
Z(h(1)*h(2)*\dots ) & h(0)=0 \wedge h(1)*h(2)*\dots \in A \\
W(h(1)*h(2)*\dots) &  h(0)=1 \wedge h(1)*h(2)*\dots \in B \\
0 & \textup{ otherwise}
\end{cases}.
\ee
Now, $V:2^{\N}\di \N$ defined as in \eqref{fool} is \emph{bijective} on $A~\dot{\cup}~B$, which is readily verified via a tedious-but-straightforward case distinction.  Hence, $A~\dot{\cup}~B$ is strongly countable and applying item \eqref{GF} yields 
an enumeration $(f_{n})_{n\in \N}$ of $A$.   By the definition of $A$, we have $(\exists f\in 2^{\N})(Y_{0}(f, n)=0)\asa (\exists m\in\N)(Y_{0}(f_{m}, n)=0)$, for any $n\in \N$.
Now define $X\subset \N$ as follows:  $n\in X\asa (\exists m \in \N)(Y_{0}(f_{m}, n)=0)$.  This set is exactly as needed for $\DCA_{C}^{-}$, and we are done. 

\smallskip

For the implication $\DCA_{C}^{-}\di \cocode_{1}$, let $Y:2^{\N}\di \N$ be bijective on $A\subset 2^{\N}$.
Now consider, for any $n, m\in \N$ and $i=0, 1$, the following:
\[
 (\exists g \in A)(  g(m) = i\wedge Y(g) = n)\asa (\forall f\in A)( f(m)\ne i\di Y(f)\ne n  ), 
\]
which follows by definition and satisfies the required `at most one' conditions. 
Then $\DCA_{C}^{-}$ provides $X\subset \N^{3}$ such that 
\[
(n,m,i)\in X\asa (\exists g \in A)(  g(m) = i\wedge Y(g) = n)
\]
for any $n, m\in \N$ and $i=0, 1$.  The enumeration of $A$ is given by $f_n(m) = i$ for the unique $i$ such that $(n,m,i) \in X$, and we are done. 
\end{proof}
The `at most one' conditions in $\DCA_{C}^{-}$ may seem strange, but similar constructs exist in second-order RM: as discussed in \cite{simpson2}*{p.\ 181}, a version of 
Suslin's classical result that the Borel sets are exactly the $\Delta_{1}^{1}$-sets can be proved in $\ATR_{0}$.   However, Borel sets in second-order RM are in fact given by $\Delta_{1}^{1}$-formulas 
that satisfy an `at most one' condition, in light of \cite{simpson2}*{V.3.3-4}.

\smallskip

We finish this section with a remark on hyperarithmetical analysis.
\begin{rem}\label{hypersect}\rm
The notion of \emph{hyperarithmetical set} (\cite{simpson2}*{VIII.3}) gives rise to the (second-order) definition of \emph{system/statement of hyperarithmetical analyis} (see e.g.\ \cite{monta2} for the exact definition), 
which includes systems like $\Sigma_{1}^{1}$-$\textsf{CA}_{0}$ (see \cite{simpson2}*{VII.6.1}).  Montalb\'an claims in \cite{monta2} that \textsf{INDEC}, a special case of \cite{juleke}*{IV.3.3}, is the first `mathematical' statement of hyperarithmetical analysis.  The latter theorem by Jullien can be found in \cite{aardbei}*{6.3.4.(3)} and \cite{roosje}*{Lemma 10.3}.  

\smallskip

%
The monographs \cites{roosje, aardbei, juleke} are all `rather logical' in nature and $\textsf{INDEC}$ is the \emph{restriction} of a higher-order statement to countable linear orders in the sense of RM (\cite{simpson2}*{V.1.1}), i.e.\ such orders are given by sequences.  
In our opinion, the statements $\MCT_{1}^{\net}$ and $\BWC_{1}$ introduced above are (much) more natural than \textsf{INDEC} as they are obtained from theorems of mainstream mathematics 
by a (similar to the case of $\textsf{INDEC}$) restriction, namely to strongly countable sets.  
Now consider,  $\ACAo+\X$ where $\X$ is either $\MCT_{1}^{[0,1]}$, $\cocode_{1}$, $\DCA_{C}^{-}$, or $\BWC_{1}+\INDX$.
By the above, $\ACAo+\X$ is a rather natural system \emph{in the range of hyperarithmetical analysis}, namely sitting between $\RCAo+\textsf{weak}$-$\Sigma_{1}^{1}$-$\textsf{CA}_{0}$ and $\ACAo+\QFAC^{0,1}\equiv_{\L_{2}}\Sigma_{1}^{1}$-\textsf{CA}$_{0}$. 
\end{rem}

\subsubsection{Countable sets}\label{diff}
We study the Bolzano-Weierstrass for countable sets in its various guises and connect it to $\cocode_{0}$.

\smallskip

Firstly, as perhaps expected in light of the use of $\INDX$ above, we also need a fragment of the induction axiom, as follows. 
\bdefi[$\INDY$]
Let $Y^{2}$ satisfy $(\forall n\in \N)(\exists \textup{ at most one } f\in 2^{\N})(Y(f, n)=0)$.  
For $k\in \N$, there is $w^{1^{*}}$ with $|w|=k$ such that for $m\leq k$, we have:
\[
(w(m)\in 2^{\N}\wedge Y(w(m), m)=0) \asa (\exists f\in 2^{\N})(Y(f, m)=0).
\]
\edefi
Note that $\IND_{0}\di \IND_{1}$ by definition. 
The following theorem is a first approximation of the results in \eqref{hong}.
\begin{thm}\label{reefer} The system $\ACAo$ proves $\BWC_{0}^{\pwo}\asa \cocode_{0}$
\end{thm}
\begin{proof}
The reverse implication is immediate as $\cocode_{0}$ converts $A$ into a sequence.  
Of course, $(\exists^{2})$ implies $\ACA_{0}$ and hence the second-order Bolzano-Weierstrass theorem by \cite{simpson2}*{III.2}.
For the forward implication, the construction in the proof of Theorem~ \ref{wonkja} is readily adapted. 
\end{proof}
Secondly, what remains to establish \eqref{hong} is the following.  
\begin{thm}\label{kantnochwally}
The system $\ACAo$ proves 
\be\label{zentrum}
\cocode_{0}\asa [\BWC_{0}+\IND_{0}]\asa \range_{0}\asa \BWC_{0}^{\pwo}.
\ee
\end{thm}
\begin{proof}
The implication $\BWC_{0}^{\pwo}\di[\BWC_{0}+\IND_{0}]$ follows in the same way as for $\BWC_{1}^{\pwo}\di [\BWC_{1}+\IND_{1}]$ in the proof of Theorem \ref{wonkja}, i.e.\ via $\cocode_{0}$. 
To prove $[\BWC_{0}^{}+\IND_{0}] \rightarrow \range_{0}$, let $F:2^{\N} \rightarrow \N$ be injective on $A\subset 2^{\N}$. Define the set $B$ of sequences with information $w0g$ such that 
$w0g \in B$ if $g$ is of the form $g_0 \oplus \cdots \oplus g_{k-1}$, where $k$ is the length of $s_w$, and such that $F(g_i) = i$ whenever $s_w(i) = 1$. Then $B$ is clearly countable since $A$ is countable. Using $\IND_{0}$ we see that for each $k$ there is a $w_k$ such that $s_{w_k}$ has length $k$ and approximates the characteristic function of the range of $F$. Using $\IND_{0}$ again, there is $g = g_0 \oplus \cdots \oplus g_{k-1}$ such that $w_k0g \in B$. This object is the lexicographicly largest object $w'0g' \in B$ such that the length of $s_{w'} \leq k$. It follows that $\sup B$ will approximate a coded representation of the characteristic function of the range of $F$, and $\range_{0}$ follows. 

\smallskip

To prove $\range_{0} \rightarrow \BWC_{0}^{\pwo}$, let $F:2^{\N}\di \N$ be injective on $A \subset 2^{\N}$. Let $nf \in B$ if $f \in A$ and $f(n) = 1$ and let $G(nf) = \langle n,F(f)\rangle$. Then $G$ is injective on $B$, so let $X$ be the range of $B$ under $G$. The pointwise least upper bound $f$ of $A$ is then definable from $X$ and $\exists^2$ by $f(n) = 1 \asa (\exists k)(\langle n,k\rangle \in X)$.
\end{proof}
Thirdly, we also obtain some nice equivalences for $\IND_{0}$, which can be proved as well for $\IND_{1}$ and \emph{strongly} countable sets. 
Note that the third item uses the `set theoretic' definition of finite sets of reals, also discussed in Section \ref{crux}.  
\begin{thm}[$\ACAo+\QFAC^{0,1}$]\label{xruc}
The following are equivalent. 
\begin{itemize}
\item $\IND_{0}$.
\item A countable and finite set can be enumerated \(by a finite sequence\).
\item A set $A\subset[0,1]$ with $Y:[0,1]\di \N$ injective and bounded on $A$, can be enumerated \(by a finite sequence\).
\end{itemize}
We only need $\QFAC^{0,1}$ to obtain the second item.  
\end{thm}
\begin{proof}
The second item readily implies the third one.  We now prove the second item from the first one.  
To this end, assume $A, Y$ are as in the second item and suppose $(\forall n\in \N)(\exists x\in A)(Y(x)>n)$.  
Apply $\QFAC^{0,1}$ and let $(x_{n})_{n\in \N}$ be the resulting sequence.  Define $g:\N\di \N$ as follows:
\[
g(0):= 0 \textup{ and } g(n+1):=Y(x_{g(n)}).
\]
for which we use the primitive recursion scheme in $\RCAo$.  Now note that $(x_{g(n)})_{n\in \N}$ is a sequence of distinct reals in $A$, contradicting the assumption that it is finite (as in Definition \ref{plonk}).
The previous contradiction implies that there is $N\in \N$ such that $(\forall x\in A)(Y(x)\leq N)$.  
Since $Y$ is injective on $A$, we also have $(\forall n\leq N)(\exists \textup{ at most one } x\in A)(Y(x)=n  )$.  
Now apply $\IND_{0}$ to obtain the desired enumeration of $A$.  To prove $\IND_{0}$ from the third item, let $Y$ be as in the former and fix $k\in \N$.  
Define the set $A:=\{ f\in 2^{\N}: (\exists n\leq k)(Y(f, n)=0) \}$ and define $Z(f):= (\mu n\leq k)(Y(f, n)=0)$, if such there is, and $0$ otherwise.  
Clearly, $Z$ is injective and bounded (by $k$) on $A$.  Applying the third item, we can enumerate $A$, yielding $w^{1^{*}}$ as required by $\IND_{0}$. 
\end{proof}
The previous theorem suggests that $\IND_{0}$ (and even $\cocode_{0}$) cannot prove that a finite set is enumerable, due to the absence of
an injection.  However, finite sets (that come without any obvious injection) do occur `in the wild', namely in the study of functions of bounded variation, as discussed in detail in Section \ref{truerm}.  

\smallskip

Finally, we discuss equivalences for $\cocode_{0}$ from other parts of mathematics.  
\begin{rem}[Lifting results]\rm
Firstly, consider the following algebra statement:
\begin{center}
\emph{any countable sub-field of $\R$ is isomorphic to a sub-field of an algebraically closed countable field}.
\end{center}
The second-order version of the latter is equivalent to $\ACA_{0}$ by \cite{simpson2}*{III.3.2}.  
Now, the centred statement with `countable' removed everywhere is (equivalent to) \textsf{ALCL} from \cite{samrecount}*{\S3.6.1}; it is shown in \cite{samrecount}*{Theorem 3.31} that
\be\label{fluff}
\ACAo+ \DCA \textup{ proves } \textup{\textsf{ALCL}} \di \BOOT.
\ee
\emph{without any essential modification} to the proof of \cite{simpson2}*{III.3.2}, i.e.\ the proof of the latter is `lifted' to the proof in \eqref{fluff} by `bumping up' all the relevant types by one.    
Now let \textsf{ALCL}$_{0}$ be the above centred statement in italics with `countable' interpreted as in Definition~\ref{standard}.  One readily modifies the proof from \eqref{fluff} to yield:
\be\label{fluff2}
\ACAo+ \DCA^{-}_{C} \textup{ proves } \textup{\textsf{ALCL}}_{0} \di \BOOT^{-}_{C}, 
\ee
therewith yielding $[\textup{\textsf{ALCL}}_{0}+\cocode_{1}] \asa \cocode_{0}$ over $\RCAo$ by Theorem \ref{aars}.  One can obtain similar results for the other proofs in \cites{samrecount, samFLO2}, and most likely for any second-order reversal involving countable algebra (and beyond).  
\end{rem}
\subsection{The Cantor-Bernstein theorem}\label{CBN}
We connect $\BWC_{0}^{\pwo}$ to the \emph{Cantor-Bernstein theorem} for $\N$ as studied in \cite{samcount} and defined as in Principle \ref{cbn}. 
As it happens, this theorem is studied in second-order RM as \cite{dougremmelt}*{Problem 1} and was studied by Cantor already in 1878 in \cite{cantor2}.  Our results provide an answer to (Q1).  
\begin{princ}[$\CBN$]\label{cbn}
A countable set $A\subset \R$ is strongly countable if there exists a sequence $(x_{n})_{n\in \N}$ of pairwise distinct reals such that $(\forall n\in \N)(x_{n}\in A)$.
\end{princ} 
First of all, the equivalence $[\CBN+\cocode_{1}]\asa \cocode_{0}$ is proved\footnote{The forward implication is trivial, assuming $\exists^2$. For the reverse implication: if $A$ is countable, consider a set $B$ isomorphic to $\N \oplus A$. Apply $\CBN$ to this set to show that it is strongly countable, and then $\cocode_1$ to show that it is enumerable. Thus, $A$ is enumerable.}
 in \cite{samcount}*{Theorem~3.12}.  
We have the following corollary to Theorems \ref{wonkja} and \ref{reefer}, which provides an answer to (Q1), as $\BW_{0}^{\pwo}$ can be split further. 
\begin{cor}\label{murni}
The system $\ACAo$ proves $ \BWC_{0}^{\pwo}\asa [\CBN+\BWC_{1}^{\pwo}]$.
\end{cor} 
\begin{proof} 
Immediate form $[\CBN+\cocode_{1}]\asa \cocode_{0}$ and Theorems \ref{wonkja} and \ref{reefer}.
%
%
\end{proof}
Secondly, we show that $\CBN$ does not imply $\cocode_{1}$, based on the proof of \cite{dagsamX}*{Theorem 3.26}. 
This establishes that the statements inside the same square brackets in \eqref{honkeg} are independent, even relative to $\Z_{2}^{\omega}$:
\be\label{honkeg}
\cocode_{0}\asa [\BWC_{1}^{\pwo}+\CBN]\asa [\cocode_{1}+\CBN]\asa [\DCA_{C}^{-}+\CBN].
\ee
Note that \eqref{honkeg} follows from Corollary \ref{murni}, while trivially $\cocode_{1}\di \NBI$.
\begin{thm}\label{goeddachs}
The system $\Z_{2}^{\omega}+\CBN+\IND_{0}$ cannot prove $\NBI$. 
\end{thm}
\begin{proof}
The proof of \cite{dagsamX}*{Theorem 3.28} discusses a model $\textbf{Q}^{*}$ of $\Z_{2}^{\omega}+\neg \NBI$, which implies that $\Z_{2}^{\omega}$ cannot prove $\NBI$ (or $\cocode_{1}$).  
This model is defined in \cite{dagsamX}*{Definition 2.28} and its properties are based on \cite{dagsamX}*{Lemma 2.16 and Theorem~2.17}. Here, we will explain the properties of $\textbf{Q}^{*}$ essential for the proof of our theorem, namely that this model satisfies $\CBN+\IND_{0}$.   
For the proofs of (most of) these properties, we refer to \cite{dagsamX}. 

\smallskip

First of all, the construction of the model is based on Kleene-computability relative to the functionals $\SS^2_k$, where $\SS^2_k$ is the characteristic function of some complete $\Pi^1_k$-subset of $\N^\N$. Using the L\"owenheim-Skolem theorem, we let $A \subseteq \N^\N$ be a countable set such that all $\Pi^1_k$-formulas are absolute for $A$ for all $k$. We let $(g_k)_{k \in \N}$ be an enumeration of $A$ and we let $A_k$ be the set of functions computable in $\SS^2_k$ and $\{g_0 , \ldots ,g_{k-1}\}$. The key properties are that $A_k \subsetneq A_{k+1} \subsetneq A$ and that $A_{k+1}$ contains an enumeration of $A_k$ for each $k$. 

\smallskip

Secondly, we let $\textbf{Q}[1] = A$ be the elements of the model of pure type 1. The definition of $\textbf{Q}[2]$ is as follows: If $F:A \rightarrow \N$ we let $F \in \textbf{Q}[2]$ if there is a $k_0$ such that for all $k \geq k_0$, the restriction of $F$ to $A_k$ is partially computable in $\SS^2_k$ and $\{g_0 , \ldots , g_{k-1}\}$. No uniformity is required. 
On top of this, we close $\textbf{Q}[1]$ and $\textbf{Q}[2]$ under Kleene computability hereditarily for each pure type. As proved in \cite{dagsamX}, this will not add new elements of type 1 or type 2 to the structure.   Finally, we use a  canonical extension to interpretations of all finite types. The resulting type structure, named $\textbf{Q}^*$ in \cite{dagsamX}, is a model of $\Z_2^\omega$ and satisfies our weak induction axioms $\IND_{0}$ and $\IND_{1}$.
Indeed, the models are constructed as computational closures, implying that for any sequence $ f_0 , \ldots , f_n$ of elements in the model, the coded sequence $(f_0 , \ldots , f_n)$ is also in the model, and the two induction axioms $\IND_{0}$ and $\IND_{1}$ readily follow.
 
\smallskip

Thirdly, having witnessed the construction of the model $\textbf{Q}^*$, we now show that it satisfies that all infinite subsets of $\N^\N$ are strongly countable.  In particular, we have that $\CBN$ holds in $\textbf{Q}^{*}$.
To this end, fix some arbitrary $B \in \textbf{Q}[2]$ that is (the characteristic function of) an infinite subset  of $A = \textbf{Q}[1]$.
We have established in the proof of \cite{dagsamX}*{Theorem~3.26} that $\textbf{Q}[2]$ contains a bijection $\phi:\textbf{Q}[1] \di \N$ with the extra property that $\phi_k$, the restriction of $\phi$ to $A_k$, is partially computable in $\SS^2_k$ and $g_0 , \ldots , g_{k-1}$. We do not need the explicit construction of $\phi$: it suffices to split the argument for finding a bijection from $B$ to $\N$ in two cases, as follows.
\begin{itemize}
\item If $B \subseteq A_k$ for some $k$, then $B$ is enumerable in $A_{k'}$ for some $k' > k$ (property of the model $\textbf{Q}^*$), and the inverse can be found directly.
\item In the `otherwise' case, we construct an increasing sequence of functionals $ \psi_k: (B \cap A_k)\di \N$ as being equal to the restriction of $\phi$ to $B \cap A_k$ except at finitely many points; we use the finite set of exceptions to make $\psi:=\lim_{k\di \infty}\psi_{k}$ surjective.  Now, for infinitely many $k$ we have that $B \cap (A_{k+1} \setminus A_k) \neq \emptyset$.   At each stage where this is the case, and where the range of $A_k \cap B$ under $\psi_k$ is a proper subset of the range of $A_k$ under $\phi$, we define $\psi_{k+1}$ as follows.
\begin{itemize}
\item Choose one element $f$ in $B \cap (A_{k+1} \setminus A_k)$. Let $n$ be the least element in the 
range of $A_k$ under $\phi$ that is not in the range of $B \cap A_k$ under $\psi_k$, and define $\psi_{k+1}(f) := n$.
\item We let $\psi_{k+1}$ be equal to  $\phi$ on the rest of $B \cap (A_{k+1} \setminus A_k)$, noticing that the injectivity of $\phi$ ensures that the value $n$ used above will not be in the range of $A_{k+1} \setminus A_k$ under $\phi$, so injectivity is preserved. 
\end{itemize}
Since $B$ and $\phi$ are elements in $\textbf{Q}[2]$ and $\psi$ differs from the restriction of $\phi$ to $B$ at only finitely many points in each $A_k$, it follows that $\psi \in \textbf{Q}[2]$.
\end{itemize}
The previous case distinction finishes the proof. 
\end{proof}
We now list some other interesting properties of the model $\textbf{Q}^*$ constructed above. 
If $A$ and $A_k$ are as in the construction, and $B\subset A$ is such that $B \cap A_k$ is finite for each $k$, then automatically $B \in \textbf{Q}[2]$. Since each set $A_{k+1} \setminus A_k$ is dense in $\N^\N$, this opens up numerous possibilities for counter-intuitive properties consistent with $\Z_2^\omega$. A few examples are as follows.
\begin{itemize}
\item There is a strongly countable set such that all enumerable subsets are finite.
\item There is an infinite subset of $[0,1]$ with no cluster-point.
\item There is an infinite subset of $[0,1]$ with one cluster-point $0$, but with no sequence from the set converging to $0$.
\end{itemize}
By the above, $\CBN$ is weaker than $\BWC_{0}$ and we also
conjecture that the former is `less explosive' than the latter as follows.
\begin{conj}
The system $\FIVE^{\omega}+\CBN$ cannot prove $\SIX$.
\end{conj}
Proving the previous conjecture may be difficult, as Theorem \ref{nine} suggests that $\CBN$ is `very close' to $\BWC_{0}$ in explosive power.

\smallskip

Now, the Cantor-Bernstein theorem is a standard exercise in axiomatic set theory (see e.g.\ \cite{hrbacekjech}*{p.\ 69}).   Experience bears out that when the students are asked to construct a bijection $H : A \di B$ from given injections $F : A\di B$ and $G : B \di A$, the successful solutions will all have the property that for each $a \in A$, either $H(a) = F(a)$ or $a = G(H(a))$. 
Let $H$ with this property be called a \emph{canonical witness} to the Cantor-Bernstein theorem.
Let $\CBN^{+}$ be $\CBN$ augmented with the existence of such a canonical witness. Such witnesses are assumed in \cite{dougremmelt}*{Problem~1} as part of the study of the Cantor-Bernstein theorem in second-order RM.   
\begin{thm}\label{nine}
The system $\FIVE^{\omega}+\CBN^{+}$ proves $\SIX$.
\end{thm}
\begin{proof}
We prove $\CBN^{+}\di\BOOT_{C}^{-}$ and note that \cite{dagsamX}*{Theorem 3.23} yield $\SIX$ via $\FIVE^{\omega}$.
Let $Y^{2}$ be such that $(\forall n\in \N)(\exists \textup{ at most one }f\in 2^{\N})(Y(f, n)=0)$.
Let $f_{0}$ be the constant zero function and define $A\subset \N\times 2^{\N}$ as follows:  
\[
(m, f)\in A\asa (\exists n\in \N)[ (m=2n+1\wedge Y(n, f)=0) \vee (m=2n\wedge f=f_{0})   ].
\]
Modulo coding, we can view $A$ as a subset of $2^{\N}$.
Define $F: A\di \N$ and $G:\N\di A$ as follows: $F\big((k,f)\big) := k$ and let $G(n): = (2n,f_{0})$. Both functions are injective, so let $H : A \di \N$ be a canonical witness as in $\CBN^{+}$. 
Now consider the following: 
\be\label{moii}
(\exists f\in 2^{\N})(Y(f, n)=0)\asa [H\big(2(2n+1),f_{0}\big) =2( 2n + 1)].
\ee
To prove \eqref{moii}, assume the left-hand side of \eqref{moii} for fixed $n\in \N$.  
Then there is $f\in 2^{\N}$ such that $(2n+1,f) \in A$. Since this $(2n+1,f)$ is not in the range of $G$, we must have that $H(2n+1,f) = F(2n+ 1,f) = 2n + 1$ and $G(2n + 1) = (2(2n + 1),f_{0})$.  
Since the case that $H(2(2n+1),f_{0}) = 2n+1$ (the inverse of $G$) violates that $H$ is injective, we must have that $H(2(2n + 1), f_{0}) = 2(2n + 1)$.
Now assume the left-hand side of \eqref{moii} is false for fixed $n\in \N$. Then there is no $f\in 2^{\N}$ such that $F(f)=2n+1$. Since there is an $m\in \N$ and a $g\in 2^{\N}$ such that $(m,g) \in  A$ and $H(m,g) = 2n+1$, we must have used the $G^{-1}$-part of $H$ and have that $m = 2(2n+1)$ and $g = f_0$. This contradicts that $H (2(2n + 1), f_0 ) = 2(2n + 1)$.
\end{proof}
\begin{cor}
The system $\ACAo$ proves $\BOOT^{-}_{C}\asa \CBN^{+}$.
\end{cor}
\begin{proof}
Immediate from the proof of the theorem, the above results, and the fact that $\BOOT_{C}^{-}\asa \BWC_{0}^{\fun}$ over $\RCAo$ by \cite{samcount}*{Theorem 3.12}. 
\end{proof}

\subsection{Theorems going back to Cantor}\label{flum}
In this section, we establish \eqref{hong2} from Section~\ref{intro}.  In particular, we extend Theorem~\ref{kantnochwally} via a number of equivalences involving basic theorems about the real line or limit points, all going back to Cantor one way or the other.  While interesting in their own right, our results also provide (positive) answers to questions (Q2)-(Q3) from Section \ref{intro}.
On a conceptual note, the order type $\eta$ of $\Q$ appears throughout the second-order RM, but Cantor's characterisation of $\eta$ as in $\cloq'$ below is \emph{quite} explosive by Corollary \ref{fexpo}.
%

\smallskip

First of all, the \emph{perfect set theorem} or the \emph{Cantor-Bendixson theorem} (see \cite{simpson2}*{V and VI} for the RM-study) imply that a nonempty \emph{uncountable} and \emph{closed} set has a perfect subset, and therefore \emph{the original set has at least one limit point}.  
We shall study the latter for {closed sets} as in Definition \ref{openset}. 
We note that the modern notion of limit/accumulation point was first articulated by Cantor in \cite{cantor33}*{p.\ 98}.
\begin{princ}\label{akku}
A non-enumerable and closed set in $\R$ has a limit point.  
\end{princ}
Theorem \ref{fungi} shows that $\BW_{0}^{\fun}$ is equivalent to a version of Principle \ref{akku}, which is interesting as the latter does not mention bijections or injections.  
In particular, Principle \ref{akku} is a sentence of second-order arithmetic\footnote{Let $\accu_{\RM}$ be Principle \ref{akku} formulated with RM-closed sets.  
Since $\ATR_{0}$ implies the perfect set theorem (\cite{simpson2}*{I.11.5}), we have the first implication in $\ATR_{0}\di \accu_{\RM}\di \ACA_{0}$, while the second one follows 
via the proof of \cite{simpson2}*{III.2.2}.   
We believe that the final implication reverses.} \emph{with one single modification}, namely the use of Definition \ref{openset} rather than RM-closed sets. 

\smallskip

Secondly, Cantor shows in \cite{cantor33}*{p.\ 161, Hilfsatz II} that a collection of disjoint open intervals in $\R$ is countable; this is the first instance of the well-known \emph{countable chain condition}.   
The following principle $\OCO$ expresses the former property \emph{without} mentioning the words `injection' or `bijection'.
\begin{princ}[$\OCO$]
Let $A\subset \R^{2}$ be such that for any non-identical intervals $(a, b)$ and $(c, d)$ in $A$, the intersection is empty.  
Then $A$ can be enumerated.  
\end{princ}
Let $\OCO_{0}$ be $\OCO$ with the conclusion `$A$ is countable'.  
As will become clear in the proof of Theorem \ref{fungi}, $\OCO_{0}$ is provable in $\RCAo$, akin to how Cantor's theorem is provable in $\RCA_{0}$ by \cite{simpson2}*{II.4.7}.  

\smallskip

Thirdly, the countable chain condition is found in the original version of \emph{Suslin's hypothesis}, first formulated by Suslin in \cite{sousvide}.  
In this context, Cantor contributed the following theorem (for any countable set), as discussed in \cite{riot}*{p.\ 122-123}. 
\begin{princ}[$\cloq$]
A countable linear ordering $(X, \preceq_{X})$ for $X\subset \R$ is order-isomorphic to a subset of $\Q$. 
\end{princ}
Moreover, Cantor introduces the notion of \emph{order type} in \cite{cantorbook90} and characterises the order type $\eta$ of $\Q$ in \cite{cantorm} based on the following (for any countable set).
\begin{princ}[$\cloq'$]
A countable and dense linear ordering without endpoints $(X, \preceq_{X})$ for $X\subset \R$ is order-isomorphic to $\Q$. 
\end{princ}
We use the usual\footnote{Namely that the relation $\preceq_{X}$ is transitive, anti-symmetric, and connex, just like in \cite{simpson2}*{V.1.1}.} definition of \emph{linear ordering} where `$\preceq_{X}$' is given by a characteristic function $F_{X}:\R^{2}\di \N$, i.e. $x\preceq_{X} y\equiv F_{X}(x, y)=_{0}1$, while $(X, \preceq_{X})$ is called \emph{countable} if $X\subset \R$ is.  
Similarly, an order-isomorphism from $(X, \preceq_{X})$ to $(Y, \preceq_{Y})$ is a surjective\footnote{Note that \eqref{crange} implies that $(\forall x,x' \in X)( Y(x)=_{Y}Y(x') \di x=_{X}x' )$, i.e.\ $Y$ is injective relative to the equalities `$=_{X}$' and `$=_{Y}$', i.e.\ `surjective' may be replaced by `bijective'.} $Y: X\di Y$ that respects the order relation (see \cite{simpson2}*{Def.\ V.2.7}), i.e.\
\be\label{crange}
(\forall x,x' \in X)(x\preceq_{X} x' \asa Y(x)\preceq_{Y}Y(x')),
\ee
while a \emph{well-founded linear order} (well-order) has no strictly descending sequences.  The reader should verify that using a stronger definition of order-isomorphism does not change the below equivalences. 

\smallskip

As to well-orders, Simpson calls the comparability of countable well-orders `indispensable' for a decent theory of ordinals, pioneered by Cantor in \cite{cantor1883}.  
We agree that it would be very indecent indeed to have incomparable countable well-orders, suggesting the following principle, which is just the second-order $\CWO$ from \cite{simpson2}*{V.6} formulated for linear orders over $\R$ that are countable.  
\begin{princ}[$\CWO^{\omega}$]
For countable well-orders {$(X, \preceq_{X}) $ and $ (Y, \preceq_{Y})$} where $X, Y\subset \R$, the former order is order-isomorphic to the latter order or an initial segment of the latter order, or vice versa.
\end{princ}
%
Thirdly, we present a preliminary result that got everything started.
\begin{thm}[$\ACAo$]\label{klankk}
Principle \ref{akku} implies the uncountability of $\R$ as in $ \NIN$.
\end{thm}
\begin{proof}
Let $Y:[0,1]\di \N$ be an injection and  
use $\exists^{2}$ to define $A\subset \R$ as follows:
\be\label{fli}
x\in A \asa (\exists n\in \N)(n\leq x< n+1\wedge Y(x-n)=n).
\ee
Intuitively, $A$ is the set $\{z+Y(z): z\in [0,1]\}$, although the latter need not exist (as a set) in $\ACAo$. 
Each $[m, m+1)\cap A$ has at most one element as $x, y \in \big( [m, m+1)\cap A\big)$ implies $Y(x-m)=Y(y-m)$ by \eqref{fli} and hence $x=_{\R}y$ by the injectivity of $Y$.  
In this light, $A$ does not have a limit point, while this set is trivially closed.  

\smallskip

Towards a contradiction, we now show that $A$ is non-enumerable.  Suppose $(x_{n})_{n\in \N}$ lists all elements of $A$, i.e.\ $(\forall x\in A)(\exists n\in \N)(x=_{\R}x_{n})$. 
Since we have $x\in A\asa  (Y(x- \lfloor x\rfloor)=\lfloor x\rfloor )$ for non-negative $x\in \R$, the sequence $(x_{n}-\lfloor x_{n}\rfloor)_{n\in \N}$ lists all elements of $[0,1]$.
Indeed, for $y_{0}\in [0,1]$, $(y_{0}+Y(y_{0}))\in A$ by definition and suppose $x_{n_{0}}=_{\R}y_{0}+Y(y_{0})$.  Hence, $\lfloor x_{n_{0}}\rfloor=Y(y_{0})$, and hence $y_{0}=x_{n_{0}}-\lfloor x_{n_{0}}\rfloor$.  
A sequence listing the reals in $[0,1]$ yields a contradiction by \cite{simpson2}*{II.4.7}.
\end{proof}
As is often the case (see e.g.\ \cite{dagsamX, samrecount, samNEO2}), the previous proof can be generalised to yield $\cocode_{0}$.  
As noted above, there is however a fundamental difference between $\NIN$ and $\cocode_{0}$: the latter combined with $\FIVE^{\omega}$ proves $\SIX$, while the former does not (seem to go beyond $\FIVE$).  
\begin{cor}[$\ACAo$]\label{fara}
Principle \ref{akku} implies $\cocode_{0}$.
\end{cor}
\begin{proof}
Note that $\cocode_{0}$ is trivial in case $\neg(\exists^{2})$, as all functions on $\R$ are then continuous by \cite{kohlenbach2}*{\S3}.
Hence, for the rest of the proof, we may assume $(\exists^{2})$.  

\smallskip

Let $B\subset [0,1]$ be a countable set, i.e.\ there exists $Y:[0,1]\di \N$ such that $Y$ is injective on $B$.  
Similar to \eqref{fli}, we define the following set using $\exists^{2}$:
\be\label{fli2}
x\in A \asa (\exists n\in \N)(n\leq x< n+1\wedge Y(x-n)=n\wedge (x-n) \in B).
\ee
Since $Y$ is an injection on $B$, $(m, m+1]\cap A$ has at most one element. 
Thus, $A$ is a closed subset with no limit point.  By the contraposition of Principle \ref{akku}, there is a sequence $(x_{n})_{n\in \N}$ such that $x\in A\asa (\exists n\in \N)(x=x_{n})$.  
Clearly, $(x_{n}-\lfloor x_{n}\rfloor))_{n\in \N}$ similarly enumerates $B$, and we are done.
\end{proof}
The previous corollary is interesting as follows: let $\PST$ and $\CBT$ be the \emph{perfect set theorem} and the \emph{Cantor-Bendixson theorem} formulated as in \cite{dagsamVII}, i.e.\ for closed sets as in Definition \ref{openset} that are not enumerable. 
Note that $\FIVE$ proves these theorems formulated for RM-closed sets (and in $\L_{2}$) by \cite{simpson2}*{V and VI}.
\begin{cor}
The system $\FIVE^{\omega}$ cannot prove $\PST$ or $\CBT$.
\end{cor}
\begin{proof}
By Theorem \ref{fara}, both $\PST$ and $\CBT$ imply Principle \ref{akku}.  If $\FIVE^{\omega}$ could prove e.g.\ $\PST$, we would obtain $\SIX$ by \cite{dagsamX}*{Theorem 4.22}.
However, $\FIVE^{\omega}$ is $\Pi_{3}^{1}$-conservative over $\FIVE$ by \cite{yamayamaharehare}*{Theorem 2.2}.
\end{proof}
By the previous proof, $\FIVE^{\omega}+\PST$ proves $\SIX$ (and the same for $\CBT$), i.e.\ Definition \ref{openset} makes these theorems quite explosive.   

\smallskip

Unfortunately, we could not find a way to obtain the reversal of Corollary~\ref{fara}. 
On the other hand, \emph{assuming} Principle \ref{akku}, in case $A\subset \R$ is closed and has no limit points, one readily defines $G:\R\di \N$ (using $\exists^{2}$)
such that 
\be\label{dist2}\textstyle
(\forall x\in A)(\exists n\leq G(x))\big(B(x, \frac{1}{2^{n}})\cap A=\{x\} \big), 
\ee
where \eqref{dist2} expresses that $G$ is a witnessing functional for `$A$ has no limit points'.  In other words, Principle \ref{akku} `enriches itself' with 
a witnessing functional $G$, while the set $A$ from \eqref{fli2} has an almost trivial such witnessing functional (again using $\exists^{2}$).  
All this suggests the latter witnessing construct merits further study.

\smallskip

Fourth, to obtain the equivalences in Theorem \ref{fungi}, we seem to need the following slight constructive enrichment of Principle \ref{akku}, as provided by \eqref{dist2}.  
\begin{princ}[$\accu$]
For any closed $A\subseteq \R$ and $G:\R\di \N$ such that \eqref{dist2}, there is $(x_{n})_{n\in \N}$ in $A$ with $(\forall x\in \R)\big(x\in A\asa (\exists n\in \N)(x=x_{n})\big)$. 
\end{princ}
%

We also study the following, apparently stronger, variation in Theorem \ref{fungi}.
\begin{princ}[$\accu'$]
For any $A\subseteq \R$ and $G:\R\di \N$ such that \eqref{dist2}, there is $(x_{n})_{n\in \N}$ in $A$ with $(\forall x\in \R)\big(x\in A\asa (\exists n\in \N)(x=x_{n})\big)$. 
\end{princ}
We note that Theorem \ref{fungi} provides a positive answer to (Q3) from Section~\ref{intro} as $\accu$ and $\ccc$ do not 
involve the notions `injection' or `bijection'.  Moreover, $\accu'\asa \accu$ is a nice robustness result, showing that quantifying over all sub-sets of $\R$ (rather than just the closed ones) need not be problematic.  
%
%
%
%
%
\begin{thm}\label{fungi}
The following are equivalent over $\ACAo$:
\begin{multicols}{2}
\begin{enumerate}
 \renewcommand{\theenumi}{\alph{enumi}}
\item $\cocode_{0}$,\label{won}
\item $\BWC_{0}^{\pwo}$ \(Bolzano-Weierstrass\),\label{toe}
\item $\accu'$,
\item $\accu$,
\item $\OCO$.\label{fina}
\end{enumerate}
\end{multicols}
\end{thm}
\begin{proof}
The equivalence $\eqref{won}\asa \eqref{toe}$ can be found in Theorem \ref{kantnochwally}.
We note that $\frac{1}{1+e^{x}}$ defines an injection from $\R$ to $(0,1)$.  
Hence, using $\exists^{2}$, one readily extends $\cocode_{0}$ to subsets $A\subset \R$.  

\smallskip

The implication $\accu\di \cocode_{0}$ is (essentially) proved in Corollary \ref{fara}, as the functional $G$ as in \eqref{dist2} is readily defined in this case.  
%
%
To prove $\OCO\di \accu$, let $A, G$ be as in the latter, i.e.\ satisfying \eqref{dist2}.  
By the latter, we have that for $x, y\in A$, the intersection $B(x, \frac{1}{2^{G(x)+1}})\cap B(y, \frac{1}{2^{G(y)+1}})$ is empty in case $x\ne y$.  
We shall now define the set consisting of $ B(x, \frac{1}{2^{G(x)+1}})$ for $x\in A$.  
To this end, define $B\subset \R^{2}$ as:
\be\label{kieper}\textstyle
(a, b)\in B\asa \big[\frac{a+b}{2}\in A\wedge a=_{\R}\frac{a+b}{2}- \frac{1}{2^{G(\frac{a+b}{2})+1}} \wedge b=_{\R}\frac{a+b}{2}+ \frac{1}{2^{G(\frac{a+b}{2})+1}} \big].
\ee
Applying $\OCO$ to the set $B$ to yield sequences $(a_{n})_{n\in \N}$ and $(b_{n})_{n\in \N}$ with
\[
(\forall a, b\in \R)\big( (a, b)\in B\asa (\exists n\in \N)(a_{n}=a\wedge b_{n}=b)    )
\]
Then the sequence $(\frac{a_{n}+b_{n}}{2})_{n\in \N}$ enumerates $A$, and this implication is done.

\smallskip

For $\cocode_{0}\di \OCO$, we first prove $\OCO_{0}$ in $\ACAo$.  To the latter end, let $A\subset \R^{2}$ be as in $\OCO_{0}$ and fix some enumeration $(q_{n})_{n\in \N}$ of $\Q$.
Define $Y((a, b))$ as the least $n\in \N$ such that $q_{n}\in (a, b)$ if such there is, and $0$ otherwise.  Clearly, $Y$ is injective on $A$ and the latter is countable, i.e.\ $\OCO_{0}$ follows inside $\ACAo$.  
Clearly, the combination $\cocode_{0}+\OCO_{0}$ implies $\OCO$.  Thus, $\cocode_{0}\di \OCO$ over $\ACAo$ follows. 

\smallskip

Finally, the reverse implication in $\accu\asa \accu'$ is trivial.  
Now fix $A\subset \R$ and $G:\R\di \N$ satisfying \eqref{dist2}.  Then \eqref{kieper} yields a collection of 
open disjoint intervals in $\R$.  
Since $\accu\di \OCO$, this collection can be enumerated, yielding $\accu'$.  
\end{proof}
We shall obtain an equivalence between Principle \ref{akku} and $\cocode_{0}$ over an elegant base theory in Section \ref{truerm}

\smallskip

Fifth, the theorem has some interesting corollaries as follows.  
Let $\accu_{0}$ be $\accu$ with the consequent weakened to stating that $A$ is countable.  
In contrast to $\BWC_{0}^{j}$, the former principles for countable sets are weak, as follows. 
\begin{cor}\label{labbe}
The system $\ACAo$ proves $\accu_{0}$ and $\ccc_{0}$.
\end{cor}
\begin{proof}
Note that $\ccc_{0}$ was proved in $\ACAo$ in the proof of the theorem. 
To prove $\accu_{0}$, note that \eqref{dist2} yields an injection to $\Q$.
\end{proof}
Let $\accu_{1}$ be the restriction of $\accu$ to infinite sets $A\subset \R$ and with conclusion weakened to: there is a bijection from $A$ to $\N$.
Similarly, let $\OCO_{1}$ be $\OCO$ with the weaker conclusion `$A$ is strongly countable' for infinite $A\subset \R$.  
\begin{cor}\label{labbedva}
The system $\ACAo$ proves $\cocode_{0}\asa [  \accu_{1}+ \cocode_{1}]$ and $\ccc_{1}\asa \CBN\asa \accu_{1}$.
\end{cor}
\begin{proof}
For the second part, $\RCAo$ proves $\ccc_{0}$ by the proof of the theorem.  Applying $\CBN$ to the conclusion of the former for infinite sets, one obtains $\ccc_{1}$.   The proof of $\ccc_{1}\di \accu_{1}$ follows from the proof of $\ccc\di \accu$.  Finally, the proof of Corollary~\ref{fara} is readily adapted to $\accu_{1}\di \CBN$.
For the first part, we note that $\accu_{1}$ only deals with infinite sets.  To obtain the same results for finite sets $A\subset \R$, consider the infinite set $B:=A\cup \Q$ and note that $\mu^{2}$ can enumerate all elements in $A\cap \Q$.
Given an enumeration of $B$, one similarly obtains an enumeration of $A$.
\end{proof}
The first equivalence is interesting as the left-hand side (only) deals with injections, while the right-hand side (only) deals with bijections. 
Similarly, we have $\BWC_{0}\asa [\ccc_{1}+\BW_{1} ]$ by Corollary \ref{murni}.   
Thus, we have provided an answer to question (Q2) from Section \ref{intro}.  Next, we consider $\CWO^{\omega}$ as follows. 
\begin{thm}
The system $\ACAo$ proves $\cocode_{0}\asa [\CWO^{\omega}+\IND_{0}]$. 
\end{thm}
\begin{proof}
%
%
For $\cocode_{0}\di \CWO^{\omega}$, use the proof that $\ATR_{0}\di \CWO$ over $\RCA_{0}$ from \cite{simpson2}*{V.6.8}.  
Note that $\ACAo+\BOOT_{C}^{-}$ proves $\ATR_{0}$ by \cite{simpson2}*{V.5.2}.  
Recall that $\cocode_{0}\di \IND_{0}$ is proved in Theorem~\ref{aars}.

\smallskip

For $[\CWO^{\omega}+\IND_{0}]\di \cocode_{0}$, let $Y:\R\di \N$ be injective on $A\subset[0,1]$.  In case $(\exists m\in\N)(\forall x\in A)(Y(x)\leq m)$, $\IND_{0}$ provides an enumeration of $A$ as we have 
\[
(\forall n\in \N)(\exists \textup{ at most one } x\in [0,1])(x\in A\wedge Y(x)=n).
\]    
Hence, we may assume $(\forall m\in \N)(\exists x\in A)(Y(x)\geq m)$.  
Now define the linear order $(A, \preceq_{A})$ via the following formula:
\[
x\preceq_{A} y \equiv \big[Y(y)=n_{0} \vee [  Y(x)\ne n_{0}\wedge Y(x)\leq_{\N}Y(y)]  \big],
\]
where $n_{0}\in \N$ is the least $n\in \N$ such that $(\exists x\in A)(Y(x)=n)$; this number is readily defined using $\IND_{0}$. 
Let $y_{0}\in A$ be such that $Y(y_{0})=n_{0}$.
Intuitively, $(A,\preceq_{A})$ has order type $\omega+1$, i.e.\ the order of $\N$ followed by one element.  
Hence, of the four different possibilities provided by the consequent of $\CWO^{\omega}$, three lead to contradiction.  Indeed, a finite initial segment of either $(\N, \leq_{\N})$ or $(A, \preceq_{A})$ has only got finitely many elements (since $Y$ is an injection), while $\N$ is infinite and $A$ satisfies $(\forall m\in \N)(\exists x\in A)(Y(x)\geq m)$.  Similarly, an order-isomorphism $W:A\di \N$ leads to contradiction as follows: since there is $y_{0}\in A$ such that $Y(y_{0})=n_{0}$, there cannot be a injection from $A\setminus \{y_{0}\}$ to $\{0, 1,\dots, W(y_{0})\}$, as the latter set is finite, while the former is not. 
Similarly, an order-isomorphism $Z:\N\di A$ yields a contradiction as any $n\geq n_{0}$ is mapped below $Z(n_{0})\in A$ (relative to $\preceq_{A}$), which is not possible as $Y$ is an injection. 
The only remaining possibility is that $\CWO^{\omega}$ provides an order-isomorphism $Z: \N\di A\setminus \{y_{0}\}$, where $A\setminus \{y_{0}\}=\{ y\in A: y \prec y_{0} \}$ is an initial segment of $A$.    
The morphism $Z$ is then a sequence satisfying $(\forall x\in A\setminus\{y_{0}\})(\exists n\in \N)(Z(n)=_{\R}x)$, i.e.\ we obtain an enumeration of $A$. 
\end{proof}
\begin{thm}
The system $\ACAo$ proves $\cocode_{0}\asa [\cloq+\IND_{0}]$. 
\end{thm}
\begin{proof}
To prove $\cocode_{0}\di \cloq$, use the well-known `back-and-forth' proof based on the enumeration of $A$ (see \cite{riot}*{p.\ 123}). 
By Theorem \ref{kantnochwally}, we only need to prove $\cloq\di \range_{0}$ in $\ACAo$.  To this end, fix $A\subset [0,1]$ and let $Y:[0,1]\di \N$ be countable in $A$.  
Wlog we may assume that $0, 1\not \in A$.  
Now define the set $R\subset \R$ as follows:  $y\in R$ if and only we have either $(\exists n\in \N)(y=_{\R}n)$, or the following holds
\[
(\exists q\in \Q)(|y-q|\in A ) \wedge  (\forall m\in \N)(m<|y|<m+1\di Y(|y-q|)=m ).
\]
Clearly, the set $R$ is countable and $(R, \leq_{\R})$ is a linear order.  Apply $\cloq$ to obtain $Q\subset \Q$ and $Z:R\di \Q$ such that $Z$ is an order-isomorphism from $(R, \leq_{\R})$ to $(Q, \leq_{\Q})$.  
Now consider the following formula where $n\in \N$:
\begin{align}
(\exists x\in A)(Y(x)=n)
&\asa (\exists y\in  (n, n+1))(y\in R)\notag\\
&\asa (\exists q\in Q)(Z(n)<_{\Q}q <_{\Q}Z(n+1)).\label{perfide}
\end{align}
The first equivalence holds by the definition of $R$, while the second equivalence follows from the fact that $Z$ is an order-isomorphism. 
Since \eqref{perfide} is decidable given $(\exists^{2})$, $\range_{0}$ is now immediate. 
\end{proof}
Inspired by the previous proof, a version of Hausdorf's decomposition theorem for countable linear orders (see \cite{clotekussen}*{Theorem 12} for the second-order RM version) should imply $\cocode_{0}$. 
In turn, the previous proof inspires the following corollary.  
\begin{cor}\label{fexpo}
The system $\ACAo$ proves $\cocode_{0}\asa [\cloq'+\IND_{0}]$. 
\end{cor}
\begin{proof}
To prove $\cocode_{0}\di \cloq'$, use the well-known `back-and-forth' proof based on the enumeration of $A$ (see \cite{riot}*{p.\ 123}). 
To prove $\cloq'\di \range_{0}$, fix $A\subset [0,1]$ and let $Y:[0,1]\di \N$ be countable in $A$.  
Wlog we may assume that $ A\cap \Q=\emptyset$ as Feferman's $\mu^{2}$ allows us to list the rationals in $A$.  
Now define the set $R'\subset \R$ as follows:  $y\in R'$ if and only we have either $(\exists q\in \Q)(y=_{\R}q)$, or the following holds
\[
(\exists q\in \Q)(|y-q|\in A ) \wedge  (\forall m\in \N)(m<_{\R}|y|<_{\R}m+1\di Y(|y-q|)=m ).
\]
Clearly, the set $R'$ is countable and $(R', \leq_{\R})$ is a dense linear order without end points.  Apply $\cloq'$ to obtain an order-isomorphism $Z$ from $(R', \leq_{\R})$ to $(\Q, \leq_{\Q})$.  
Now consider the following formula where $n\in \N$:
\begin{align}
(\exists x\in A)(Y(x)=n)
&\asa (\exists y\in  (n, n+1))(y\in R' \wedge \textup{$y$ is irrational})\label{co1}\\
&\asa (\exists q\in \Q\cap (Z(n), Z(n+1))  )(\forall r\in \Q\cap (n, n+1))(Z(r)\ne_{\Q} q ).\notag
\end{align}
The first equivalence holds by definition while the second equivalence follows from the fact that $Z$ is an order-isomorphism. 
As for the theorem, $\range_{0}$ follows. 
\end{proof}
Restricting $\cloq'$ to strongly countable sets, one readily obtains an equivalence to $\cocode_{1}+\IND_{1}$ by introducing an extra condition `$x>p$' in \eqref{co1} with $p\in\Q$.  

\smallskip

Finally, as to related research, Mal'tsev's theorem on countable ordered groups (\cite{mallebabbe}) is studied in second-order RM (\cite{solvent}), and seems to imply $\cocode_{0}$.  
%
%
%

\section{The bigger picture}\label{bigger}
Section \ref{frap} yields many (robust) equivalences for the Bolzano-Weierstrass theorem as in $\BW_{0}$ and $\BWC_{1}$.  
With these in place, it is time to connect the latter to the bigger picture, namely ordinary mathematics and set theory, as follows.
\begin{itemize}
\item In Section \ref{heli}, we connect the Bolzano-Weierstrass theorem as in $\BWC_{0}$ to the Heine-Borel theorem and the Lindel\"of lemma as studied in \cites{dagsamIII, dagsamVI}. 
\item We connect $\BWC_{0}$ to the \emph{countable union theorem} from set theory (Section~\ref{heli2}); a natural restriction of the latter is equivalent to the former. 
\item In Section \ref{BV}, we show that $\BWC_{0}$ is equivalent to the \emph{Jordan decomposition theorem} and similar results on functions of \emph{bounded variation}.  We also consider theorems on \emph{regulated} functions.  
\item In Section \ref{unorder}, we show that $\BWC_{0}$ is equivalent to basic properties of \emph{unordered sums}, which are a device for bestowing meaning upon sums over uncountable index sets. 
\end{itemize}
Regarding the final item, the Jordan decomposition theorem and its ilk have no obvious or direct connection to countability at all, and have been studied in second-order RM (\cites{kreupel, nieyo}).  

\subsection{Heine-Borel and Lindel\"of}\label{heli}
\subsubsection{Introduction}\label{kintro}
In this section, we connect the Bolzano-Weierstrass theorem as in $\BWC_{0}$ to the Heine-Borel theorem and the Lindel\"of lemma. 
An overview of our results is as follows.

\smallskip

In Section~\ref{frik}, we identify weak/countable versions of the Heine-Borel theorem and Lindel\"of lemma that are equivalent to $\BW_{0}$.
In Section \ref{link}, we show that $\LIN$, a most general version of the Lindel\"of lemma for $\N^{\N}$, we have $\BOOT+\QFAC^{0,1}\di\LIN\di \BWC_{0}$, working over $\ACAo$. 
In Section \ref{link2}, assuming a fragment of the induction axiom, we similarly establish:
\be\label{darfl}
\BOOT\di \HBT\di \BWC_{0}\di \BWC_{1}.
\ee
Recall that $\BOOT$ and $\HBT$ are the higher-order counterparts of $\ACA_{0}$ and $\WKL_{0}$ (see Remark \ref{ECF}).  
 In this light, higher-order RM yields a much richer picture than its second-order counterpart, in that there are at least two extra `Big' systems.  

\smallskip

Next, the following series of implications is also established in Section \ref{link2}, \emph{without} the use of extra induction:
\be\label{darflK}
\BOOT\di \SSEP\di \BWC_{0}\di \BWC_{1},
\ee
where $\SSEP$ is the higher-order counterpart of $\Sigma_{1}^{0}$-separation.  The latter is equivalent to $\WKL_{0}$ by \cite{simpson2}*{IV.4.4} and $\ECF$ maps $\SSEP$ to $\Sigma_{1}^{0}$-separation; we believe that $\HBU$ `speaks more to the imagination' than $\SSEP$.  
Moreover, $\HBU\asa \HBT\asa \SSEP$ is established in Section \ref{link2}, assuming extra axioms discussed next.  

\smallskip

Finally, we should say a few words on the \emph{neighbourhood function principle} $\NFP$ from \cite{troeleke1}*{p.\ 215}.
Restricted to the $\L_{2}$-langauge, $\NFP$ is equivalent to the usual comprehension principle of $\Z_{2}$.  
Now, the higher-order generalisation of comprehension, in the form of the functionals $\SS_{k}^{2}$, does not provide a satisfactory classification of e.g.\ $\HBU$. 
Indeed, we know that $\Z_{2}^{\Omega}$ proves $\BOOT, \HBU, \BWC_{0}, \BWC_{1}$ and $\Z_{2}^{\omega}$ does not, while of course $\Z_{2}\equiv_{\L_{2}}\Z_{2}^{\omega}\equiv_{\L_{2}}\Z_{2}^{\Omega}$. 
As explored in \cites{dagsamX, samph, sahotop}, the higher-order generalisation of $\NFP$ provides a more satisfactory classification of these principles: there are natural fragments of $\NFP$ equivalent to $\BOOT$, $\HBU$, and the Lindel\"of lemma, \emph{assuming} a fragment of $\NFP$ called $\A_{0}$, discussed in Section \ref{link2}.  
By Theorem \ref{timtam} and Corollary \ref{drapa}, we have $\HBU\asa \SSEP\asa \HBT$ working over $\ACAo+\A_{0}$.  

\smallskip

We should not have to point out that second-order RM assumes/needs $\Delta_{1}^{0}$-comprehension in the base theory.  Thus, it stands to reason that the development 
of RM based on $\NFP$ requires a fragment of the latter, like the $\A_{0}$ axiom, in the base theory.  This argument is explored at length in \cites{sahotop, samph}. 
Moreover, by Corollary~\ref{dikker}, there is even a fragment of $\NFP$, similar to $\A_{0}$, that is equivalent to $\BWC_{0}$.

\subsubsection{Countable coverings}\label{frik}
We connect $\BWC_{0}$ to versions of the Heine-Borel theorem and Lindel\"of lemma for coverings that are countable as in Definition~\ref{standard}.

\smallskip

First of all, the following version of the `countable' Heine-Borel theorem implies $\NIN$ by \cite{dagsamX}*{Cor.\ 3.20}, but no reversal is known. 
\begin{princ}[$\HBC_{0}$]
For countable $A\subset \R^{2}$ with $(\forall x\in [0,1])(\exists (a, b)\in A)(x\in (a, b))$, there are $(a_{0}, b_{0}), \dots (a_{k}, b_{k})\in A$ with $(\forall x\in [0,1])(\exists i\leq k)(x\in (a_{i},b_{i} ))$.
\end{princ}
The Heine-Borel theorem for different representations of open coverings is studied in RM (\cite{paultoo}), i.e.\ the motivation for $\HBC_{0}$ is already present in second-order RM.
Moreover, Borel in \cite{opborrelen2}*{p.\ 42} uses `countable infinity of intervals' and \textbf{not} `sequence of intervals' in his formulation\footnote{In fact, Borel's explicitly mentions a version of $\cocode_{1}$ in \cite{opborrelen2}*{p.\ 6} while the proof of the Heine-Borel theorem in \cite{opborrelen2}*{p.\ 42} starts with an application of $\cocode_{1}$ and then proceeds with the usual `interval-halving' proof, similar to Cousin's proof in \cite{cousin1}.} of the Heine-Borel theorem.  He also mentions in \cite{opborrelen2}*{p.\ 42, Footnote (1)} a `theoretical method' for `effectively determining' the finite sub-covering at hand. 
In this light, we may assume that the finite sub-covering in $\HBC_{0}$ is given by a finite sequence of reals without fear of adding `extra data'. 

\smallskip

We shall study a `sequential' version of $\HBC_{0}$ involving sequences of (sub-)coverings.  Such sequential theorems are well-studied in RM, starting with \cite{simpson2}*{IV.2.12}, and also in \cites{fuji1,fuji2,hirstseq,dork2,dork3, yokoyamaphd, polahirst}.    
\begin{princ}[$\HBC_{0}^{\seq}$]
Let $(A_{n})_{n\in \N}$ be a sequence of sets in $\R^{2}$ with countable union.  Then there is $(b_{n})_{n\in \N}$ such that for $n\in \N$, $b_{n}$ is a finite sequence of elements of $A_{n}$ and if 
the intervals in $A_{n}$ cover $[0,1]$, then so do the intervals in $b_{n}$.
\end{princ}
On a related note, let $\LIN_{0}$ be the Lindel\"of lemma for countable sets $A\subset \R$, i.e.\ for $\Psi:\R\di \R^{+}$, there is $(x_{n})_{n\in \N}$ in $A$ such that  $\cup_{n\in \N}B(x_{n}, \Psi(x_{n}))$ covers $A$. 
We have the following theorem connecting the aforementioned principles.
\begin{thm}\label{jeffreys}
The system $\ACAo$ proves $\LIN_{0}\asa \cocode_{0}\asa  \HBC_{0}^{\seq} $.
\end{thm}
\begin{proof}
The implication $\LIN_{0}\leftarrow \cocode_{0}$ is trivial, while the reversal follows 
from $\LIN_{0}\di \accu$ which in turn follows from applying $\LIN_{0}$ to the covering provided by $\frac{1}{2^{G(x)+1}}$ as in \eqref{dist2}. 
The implication $ \cocode_{0}\di  \HBC_{0}^{\seq}$ is also straightforward as $\cocode_{0}$ allows us to convert
countable sets into sequences.  The usual second-order proof from \cite{simpson2}*{IV.1} in $\WKL$ now yields $\HBC_{0}$, while \cite{yokoyamaphd}*{Theorem 2.7} yields the sequential version, also working in $\WKL_{0}$.  

\smallskip

Finally, to obtain $\HBC_{0}^{\seq}\di \cocode_{0}$, fix a set $A\subset [0,1]$ with $Y:[0,1]\di \N$ injective on $A$.  Now define the sequence $(A_{n})_{n\in \N}$ in $\R^{2}$ as follows:
\[\textstyle
(a, b)\in A_{n}\asa \big[ \frac{a+b}{2}\in A\wedge Y(\frac{a+b}{2})=n \wedge b-a= 4 \max\big( |1-\frac{a+b}{2}|, \frac{a+b}{2}  \big) \big].
\]
By definition, each $A_{n}$ has at most one element and the union is countable as $\cup_{n\in \N}A_{n}$ is a variation of $A$. 
Let $(b_{n})_{n\in \N}$ be as provided by $\HBC_{0}^{\seq}$ and note that 
\[\textstyle
(\exists x\in A)(Y(x)=n)\asa (\exists (a, b)\in b_{n})(Y(\frac{a+b}{2})=n), 
\]
which immediately yields $\cocode_{0}$, and we are done. 
\end{proof}
Note that the formulation of $\HBC_{0}^{\seq}$ avoids the countable union theorem, which happens
to be the topic of Section~\ref{heli2}.  Theorem \ref{jeffreys} also has a certain robustness: the second equivalence still goes through if we let $(b_{n})_{n\in \N}$ be a sequence of non-empty finite sets, while assuming $\cocode_{1}$.  Moreover, we believe that many sequential versions of theorems are equivalent to $\cocode_{0}$, like e.g.\ $\ADS$ and $\RT_{2}^{2}$ from the RM zoo (see \cite{dsliceke}).
An exception is $\cloq'$, as shown in Section \ref{heli2}. 

\smallskip

Finally, by Theorem \ref{jeffreys}, the general Lindel\"of lemma for \emph{any} set $A\subset \R$ is quite explosive, yielding $\SIX$ when combined with $\FIVE^{\omega}$. 
Nonetheless, we show in the next section that this general version is still provable from $\BOOT+\QFAC^{0,1}$. 
%
%
%
%

\subsubsection{A general Lindel\"of lemma}\label{link}
We show that a most general formulation of the Lindel\"of lemma still follows from $\BOOT$.  
We have established a similar result for the Heine-Borel theorem for uncountable coverings of closed sets in \cite{dagsamVII}*{Theorem~4.5}.
We note that Lindel\"of proves his eponymous lemma for \emph{any} set in $\R^{n}$ in \cite{blindeloef}.  
\begin{princ}[$\LIN$]\label{LIN}
For any $G^{2}$ and $D\subseteq \N^{\N}$, there is $(f_{n})_{n\in \N}$ in $D$ such that $\cup_{n\in \N} [\overline{f_{n}}G(f_{n})] $ covers $D$. 
\end{princ}
The following theorem is the main result of this section. 
\begin{thm}
The system $\ACAo+\QFAC^{0,1}+\BOOT$ proves $\LIN$.
\end{thm}
\begin{proof}

\smallskip

Fix a non-empty set $D\subset \N^{\N}$ and $G^{2}$ and let $(\sigma_{n})_{n\in \N}$ be a list of all finite sequences.  
Use $\BOOT$ to define $X\subset\N$ such that
\be\label{fril}
n \in X\asa (\exists f\in D )\big( f \in [\sigma_{n}] \wedge \sigma_{n} =_{0^{*}}\overline{f}G(f)\big).
\ee
Define $\tau_{0}$ as $\sigma_{n_{0}}$ where $n_{0}:=(\mu n)(n \in X)$, and define $\tau_{n+1}$ as $\sigma_{n+1}$ if $n+1\in X$, and $\tau_{n}$ otherwise.  
Then $\cup_{n\in \N}[\tau_{n}]$ also covers $D$, but we still need to `identify' the associated $f\in D$ from \eqref{fril}. To this end, apply $\QFAC^{0,1}$ to 
\[
(\forall n\in X)(\exists f\in D)( f \in [\sigma_{n}] \wedge \sigma_{n} =_{0^{*}} \overline{f}G(f)).
\]
The resulting sequence provides the countable sub-covering as required by the conclusion of Principle \ref{LIN}.
\end{proof}
As shown in \cites{dagsamIII, dagsamV}, the Lindel\"of lemma for the full Baire space yields $\FIVE$ when combined with $(\exists^{2})$. 
Moreover, by Theorem \ref{jeffreys}, $\LIN\di \cocode_{0}$ is immediate, implying that that $\FIVE^{\omega}+\LIN$ proves $\SIX$. We also have
\[
[\BOOT+\QFAC^{0,1}]\di \LIN\di \BWC_{0}\di \BWC_{1}.
\]
Finally, since Baire space is not $\sigma$-compact, we believe the use of countable choice in the previous proof to be essential. 

\subsubsection{Uncountable coverings}\label{link2}
In this section, we connect $\HBT$ and related principles to $\BWC_{0}$ as sketched in Section \ref{kintro}

\smallskip

First of all, in more detail, our main result is $\HBU\di \BWC_{0}$ assuming an extra axiom $\A_{0}$ introduced in \cite{samph, sahotop} and discussed below. 
This implication is established using the intermediate principle $\SSEP$ as in Principle \ref{septisch}.  The latter is the third-order counterpart of the $\Sigma_{1}^{0}$-separation principle, which is equivalent to $\WKL_{0}$ by \cite{simpson2}*{IV.4.4}.   
Since $\HBU$ is the higher-order counterpart of $\WKL_{0}$, one expects $\HBU\asa \SSEP$, which is indeed proved in Theorem \ref{timtam}, also assuming $\A_{0}$.  
Regarding \eqref{darfl}, weakening $\A_{0}$ is possible as in \eqref{kru}.  
We note that $\ECF$ maps both $\HBU$ and $\SSEP$ to $\WKL_{0}$, while $\A_{0}, \BWC_{0},\BWC_{1}$ are trivial under $\ECF$.
Moreover, a version of $\A_{0}$ turns out to be equivalent to $\cocode_{0}$ by Corollary \ref{dikker}.

\smallskip

Secondly, we have previously considered a separation principle in connection to $\HBU$ in \cite{samph}, namely as follows. 
\begin{princ}[$\SSEP$]\label{septisch} For $i=0,1$, $Y_{i}^{2}$, and $\varphi_{i}(n)\equiv (\exists f_{i}\in \N^{\N})(Y_{i}(f_{i}, n)=0)$, 
\[
(\forall n\in \N)(\neg\varphi_{0}(n)\vee \neg \varphi_{1}(n))\di (\exists Z\subset \N)(\forall n\in \N)\big[\varphi_{0}(n)\di n\in Z \wedge \varphi_{1}(n)\di n\not\in Z  \big].
\]
\end{princ}
The following theorem implies that $\FIVE^{\omega}+\SSEP$ proves $\SIX$, which also follows immediately from \cite{simpson2}*{VII.6.14}.
\begin{thm}\label{brilli}
The system $\ACAo$ proves $\SSEP\di \cocode_{0}$.
\end{thm}
\begin{proof}
Let $Y:\R\di \N$ be injective on the non-empty set $A\subset [0,1]$.  
Define the formula $\varphi_{i}(n, q)$ as follows where $n\in \N$ and $q\in \Q\cap (0,1)$:
\be\label{sd}
\varphi_{0}(n, q)\equiv (\exists x\in A)( Y(x)=n  \wedge x>_{\R}q)
\ee
\be\label{card}
\varphi_{1}(n, q)\equiv (\exists x\in A)( Y(x)=n  \wedge x\leq_{\R} q).
\ee
Since $Y$ is injective on $A$, we have $(\forall n\in \N, q\in \Q\cap (0,1))(\neg\varphi_{0}(n, q)\vee \neg \varphi_{1}(n, q))$.
Let $Z\subset \N\times \Q$ be as in $\SSEP$ and note that for $n\in \N, q\in \Q\cap (0,1)$, we have
\be\label{lap}
(n, q)\in Z\di (\forall x\in A)( Y(x)=n  \di x>_{\R} q),
\ee
 \be\label{land}
 (n, q)\not\in Z\di (\forall x\in A)( Y(x)=n  \di x\leq_{\R} q).
\ee
Based on \eqref{lap} and \eqref{land}, define a sequence $(x_{n})_{n\in \N}$ of reals in $[0,1]$ as follows:  $[x_{n}](0)$ is $\frac{1}{2}$ if $(n, \frac{1}{2})\in Z$, and $0$ otherwise; $[x_{n}](k+1)$ is $[x_{n}](k)+\frac{1}{2^{k+1}}$ if $(n, [x_{n}](k)+\frac{1}{2^{k+1}})\in Z$, and $[x_{n}](k)$ otherwise.   Using Feferman's $\mu^{2}$, define $(y_{n})_{n\in \N}$ as a sub-sequence (possibly with repetitions) of $(x_{n})_{n\in \N}$ such that $(\forall n\in \N)(y_{n}\in A)$.
Then $(y_{n})_{n\in \N}$ is an enumeration of $A$ such that for all $k\in \N$:
\be\label{xh}
(\exists x\in A)(Y(x)=k)\asa (\exists m\in\N)(Y(y_{m})=k).  
\ee
Indeed, the reverse implication in \eqref{xh} is immediate by the definition of $(y_{n})_{n\in \N}$.  
For the forward implication if $(\exists x\in A)(Y(x)=k) $ for fixed $k\in \N$, then $Y(x_{k})=k$ and $x_{k}\in A$, by the definition of $(x_{n})_{n\in \N}$.  
Hence, the right-hand side of \eqref{xh} follows, and we observe that $(y_{n})_{n\in \N}$ enumerates $A$.
\end{proof}
\noindent
We can obtain an equivalence via the following `at most one' condition:
\be\label{pirf}
(\forall i\in \{0, 1\})(\forall n\in \N)(\exists\textup{ at most one } f\in 2^{\N})(Y_{i}(f, n)=0). 
\ee
Let $\SSEP^{-}_{C}$ be $\SSEP$ with all type $1$ quantifiers restricted to $2^{\N}$ and \eqref{pirf}.
\begin{cor}\label{somecor}
The system $\ACAo$ proves $\SSEP^{-}_{C}\asa \cocode_{0}$.
\end{cor}
\begin{proof}
The forward implication is immediate from the proof of the theorem as \eqref{sd} and \eqref{card} satisfy the required `at most one' conditions. 
For the reverse implication, let $Y_{i}^{2}$ be as in $\SSEP^{-}_{C}$ and define $A_{i}:=\{ f\in 2^{\N}: (\exists n\in \N)(Y_{i}(f,n )=0)  \}$.  Clearly, 
this set is countable as $Z_{i}(f):=(\mu n)(Y_{i}(f,n )=0) $ yields an injection on $A_{i}$.  Hence, $\cocode_{0}$ provides an enumeration $(f_{m})_{m\in \N}$ of $A_{0}$, implying
\[
\varphi_{0}(n)\asa (\exists f\in 2^{\N})(Y_{0}(f, n)=0)\asa (\exists m\in \N)(Y(f_{m}, n)=0), 
\]
i.e.\ $\varphi_{0}(n)$ is decidable modulo $\exists^{2}$.  The same holds for $\varphi_{1}(n)$ and we are done.
\end{proof}
Next, as shown in \cite{samph}*{\S5} and \cite{sahotop}, $\HBU$, $\BOOT$, and the Lindel\"of lemma are equivalent to elegant fragments of the \emph{neighbourhood function principle} $\NFP$ from \cite{troeleke1}.  
In the same way as $\Delta_{1}^{0}$-comprehension is included in $\RCA_{0}$, the RM of $\NFP$ warrants a base theory that includes the following fragment of $\NFP$, as discussed \emph{at length and in minute detail} in \cite{samph}*{\S5} and \cite{sahotop}*{\S3.5}.
\bdefi[$\textsf{A}_{0}$]\label{froem}
For $Y^{2}$ and $A(\sigma^{0^{*}})\equiv (\exists f\in 2^{\N})(Y(f, \sigma)=0)$, we have
\[
(\forall f\in \N^{\N})(\exists n\in \N)A(\overline{f}n)\di (\exists \Phi^{2})(\forall f\in \N^{\N})A(\overline{f}\Phi(f)).
\]
\edefi
Recall the equivalence from \cite{simpson2}*{X.4.4} between $\Sigma_{1}^{0}$-induction and bounded $\Sigma_{1}^{0}$-comprehension.  
As noted above, $\IND_{0}$ occupies the same category as the latter axiom, while an equivalence between $\HBU$ and $\SSEP$ needs \emph{bounded separation}, as follows.  
The axiom `\textsf{bounded}-$\SSEP$' is $\SSEP$ weakened such that for any $k\in \N$:
\[
(\forall n\leq k)(\neg\varphi_{0}(n)\vee \neg \varphi_{1}(n))\di (\exists Z\subset \N)(\forall n\leq k)\big[\varphi_{0}(n)\di n\in Z \wedge \varphi_{1}(n)\di n\not\in Z  \big].
\]
Clearly, \textsf{bounded}-$\SSEP$ only provides a finite/bounded fragment of the separating set from $\SSEP$, and the former follows from the induction axiom.  
We now have the following theorem which establishes \eqref{darfl}.   
\begin{thm}\label{timtam}
The system $\ACAo+\A_{0}$ proves $[\HBU+\textup{ \textsf{bounded}-$\SSEP$}]\asa \SSEP $; the reverse implication holds over $\ACAo$.
\end{thm}
\begin{proof}
%
%
Assume $\HBU$ and suppose $\neg\SSEP$.  Fix $Y_{0}, Y_{1}$ as in the latter and let $A(\overline{Z}n)$ be the following, i.e.\ the formula in square brackets in $\SSEP$:
\be\label{ds}
(\varphi_{0}(n)\di n\in Z) \wedge (\varphi_{1}(n)\di n\not\in Z)  ,
\ee
where the notation `$\overline{Z}n$' in $A(\overline{Z}n)$ is justified by noting that the set $Z$ is only invoked in \eqref{ds} in the form `$n\in Z$'.
By assumption, we have $(\forall Z\subset \N)(\exists n\in \N)\neg A(\overline{Z}n)$, which has the right form to apply $\A_{0}$.  
Hence, there is $G:2^{\N}\di \N$ such that $(\forall Z\subset \N)\neg A(\overline{Z}G(Z))$.   Apply $\HBU$ to obtain $f_{1}, \dots, f_{k}\in 2^{\N}$, a finite sub-covering of the canonical covering $\cup_{f\in 2^{\N}}[\overline{f}G(f)]$.
Define $n_{0}:=\max_{i\leq k}G(f_{i})$ and note that $(\forall Z\subset \N)(\exists n\leq n_{0})\neg A(\overline{Z}n)$.
However, \textsf{bounded}-$\SSEP$ provides a set $Z_{0}\subset \N$ such that for $m\leq n_{0}+1$, we have $A(\overline{Z_{0}}m)$, a contradiction, and we are done.   
%

\smallskip

For the reverse implication, $\SSEP$ implies \textsf{bounded}-$\SSEP$.  Now assume $\SSEP$ and suppose $\HBU$ fails for $\Psi_{0}:[0,1]\di \R^{+}$.  
Consider the following for $q\in \Q\cap (0,1)$:
\[
\varphi_{0}(q)\equiv (\exists w^{1^{*}})\big[ (\forall i<|w|)(w(i)\in [0,1])\wedge [0, q] \subset \cup_{i<|w|}I_{w(i)}^{\Psi_{0}} \big],
\]
\[
\varphi_{1}(q)\equiv (\exists v^{1^{*}})\big[ (\forall j<|v|)(v(j)\in [0,1])\wedge [q,1] \subset \cup_{j<|v|}I_{v(j)}^{\Psi_{0}} \big],
\]
where $(\forall q\in \Q\cap (0,1))(\neg\varphi_{0}(q)\vee \neg\varphi_{1}(q))$ by assumption.  Let $Z_{0}\subset \N$ be as provided by $\SSEP$ and define a real $x_{0}\in [0,1]$ as follows.  
Define $[x_{0}](0)$ as $\frac{1}{2}$ if $\frac{1}{2}\in Z_{0}$, and $0$ otherwise; define $[x_{0}](k+1)$ as $[x_{0}](k)+\frac{1}{2^{k+1}}$ if $[x_{0}](k)+\frac{1}{2^{k+1}}\in Z$, and $[x_{0}](k)$ otherwise. 
By definition, the real $x_{0}$ satisfies the following:
\be\label{wrong}\textstyle
(\forall w^{1^{*}})\big[(\forall i<|w|)(w(i)\in [0,1])\di \big([x_{0}](k), [x_{0}](k)+\frac{1}{2^{k+1}}\big)\not\subset \cup_{i<|w|}I_{w(i)}^{\Psi_{0}}  \big], 
\ee
which immediately yields a contradiction as $\big([x_{0}](k), [x_{0}](k)+\frac{1}{2^{k+1}}\big)\subset I_{x_{0}}^{\Psi_{0}}$ for $k$ large enough, and we are done.
\end{proof}
\begin{cor}\label{drapa}
The system $\ACAo$ proves $[\HBT+\textup{\textsf{bounded}-$\SSEP$}]\asa \SSEP$.
\end{cor}
\begin{proof}
The reverse implication readily follows from the second part of the proof of the theorem. 
For the forward implication, consider $(\forall Z\subset \N)(\exists n\in \N)\neg A(\overline{Z}n)$ as in the proof of the theorem.
As noted above, we may use $\exists^{2}$ to code $\N\di \N$ sequences as binary sequences.  
Let $Y$ be the characteristic function of the formula obtained by omitting the leading existential quantifiers (over $2^{\N}$) of $\neg A(\sigma)$.  
Define the function $\psi:[0,1]\di \R$ as follows: $\psi(x):=0$ if there is no initial segment $\sigma^{0^{*}}$ of the binary expansion $\sigma*f$ of $x$ such that $Y(f, \sigma)=0$; otherwise $\psi(x):=\frac{1}{2^{k}}$
where $k$ is the length of the shortest such initial segment.  
Then $\psi$ yields a covering of $[0,1]$ to which $\HBT$ applies.   In the same was as in the proof of the theorem, one obtains a contradiction using \textup{\textsf{bounded}-$\SSEP$}.
\end{proof}
It is straightforward to show that $\HBT$ implies the fragment of $\A_{0}$ needed to prove $\HBU\di \HBT$.  
Another interesting exercise is to consider $\A_{0}^{-}$ which is $\A_{0}$ with the extra condition $(\forall \sigma^{0^{*}}\leq_{0^{*}}1)(\exists \textup{ at most one } f\in 2^{\N})(Y(f, \sigma)=0)$.  
Using the above results, one readily shows that over $\ACAo$:
\be\label{kru}
\BOOT\di [\HBU+\A_{0}^{-}]\di \cocode_{0}\di \A_{0}^{-},
\ee
\be\label{sutha}
[\HBU+\textup{ \textsf{bounded}-$\SSEP$}+\A_{0}^{-}]\asa \SSEP.
\ee
What is more important is the following corollary to Theorem \ref{timtam} related to $\A_{0}^{-}$.   Let $\Sigma$-$\NFP_{C}^{-}$ be $\A_{0}^{-}$ with the conclusion strengthened as in $\NFP$, i.e.\ $(\exists \gamma\in K_{0})(\forall f\in 2^{\N})A(\overline{f}\gamma({f}))$. 
Note that `$\gamma\in K_{0}$' is the notation used in $\NFP$ from \cite{troeleke1} for $\gamma^{1}$ being a total RM-code/associate.  
Let $\textup{ \textsf{bounded}-$\SSEP_{C}^{-}$}$ be $\textup{ \textsf{bounded}-$\SSEP$}$ with the same restrictions as $\SSEP^{-}_{C}$. 
\begin{cor}\label{dikker}
The system $\ACAo$ proves $\cocode_{0}\asa [\Sigma\textup{-}\NFP_{C}^{-}+ \textup{ \textsf{bounded}-$\SSEP^{-}_{C}$}]$.
\end{cor}
\begin{proof}
The forward implication is straightforward: $\BOOT_{C}^{-}$ makes $A(\sigma)$ from $\Sigma$-$\NFP_{C}^{-}$ decidable, i.e.\ there is $X$, up to coding a subset of $\N$,  such that
\[
(\forall \sigma^{0^{*}}\leq 1)\big[ \sigma \in X\asa (\exists f\in 2^{\N})(Y(f, \sigma)=0) \big].  
\]
Using $\QFAC^{1,0}$ (and induction), we obtain $G^{2}$ such that $(\forall f\in 2^{\N})A(\overline{f}G(f))$, where $G(f)$ is the least such number.  
Clearly, $G^{2}$ has an RM-code, and $\NFP_{C}^{-}$ follows.  

\smallskip

For the reverse implication, we prove $[\Sigma$-$\NFP_{C}^{-}+  \textup{ \textsf{bounded}-$\SSEP^{-}_{C}$}]\di \SSEP_{C}^{-}$ and Corollary \ref{somecor} finishes the proof. 
To obtain $\SSEP_{C}^{-}$, consider $A(\sigma)$ as in \eqref{ds}.   Note that $(\forall Z\subset \N)(\exists n\in \N)\neg A(\overline{Z}n)$ has the right form to apply $\Sigma$-$\NFP_{C}^{-}$.  
The resulting function $\gamma\in K_{0}$ has an upper bound given $\WKL$ by \cite{simpson2}*{IV.2.2}.  Now use $\textup{ \textsf{bounded}-$\SSEP_{C}^{-}$}$ to obtain a contradiction in the same way as in the proof of Theorem \ref{timtam}.
Note that Corollary \ref{somecor} yields $\cocode_{0}$.
\end{proof}
Finally, $\A_{1}$ is $\A_{0}$ but for formulas $A(\sigma^{0^{*}})\equiv (\forall f\in 2^{\N})(Y(f, \sigma)=0)$ and
proves the equivalence between $\accu$ and Principle \ref{akku}.   
The axiom $\A_{1}$ implies that any continuous function on $\N^{\N}$ has an associate/RM-code, as explored in \cite{samph}*{\S5}.

%

\subsubsection{More on separation}
In this section, we show that $\PSEP$, a separation principle much weaker than $\SSEP$, implies $\cocode_{1}$. 
We also obtain an equivalence based on a weakening of $\PSEP$. 

\smallskip

First of all, note that the following principle is readily proved by applying $\QFAC^{0,1}$ to the antecedent (see also \cite{simpson2}*{V.5.7}). 
Theorem \ref{waho} is reminiscent of  the fact that $\Pi_{1}^{1}$-separation implies $\Delta_{1}^{1}$-comprehension. 

\begin{princ}[$\PSEP$] For $i=0,1$, $Y_{i}^{2}$, and $\varphi_{i}(n)\equiv (\forall f_{i}\in \N^{\N})(Y_{i}(f_{i}, n)=0)$, 
\[
(\forall n\in \N)(\neg\varphi_{0}(n)\vee \neg \varphi_{1}(n))\di (\exists Z\subset \N)(\forall n\in \N)\big[\varphi_{0}(n)\di n\in Z \wedge \varphi_{1}(n)\di n\not\in Z  \big].
\]
\end{princ}
\begin{thm}\label{waho}
The system $\ACAo$ proves $\PSEP\di \cocode_{1}$.
\end{thm}
\begin{proof}
Let $Y:\R\di \N$ be bijective on the non-empty set $A\subset [0,1]$.  
Define the formula $\varphi_{i}(n, q)$ as follows where $n\in \N$ and $q\in \Q\cap (0,1)$:
\be\label{sd2}
\varphi_{0}(n, q)\equiv (\forall x\in A)( Y(x)=n  \di x>_{\R}q)
\ee
\be\label{card2}
\varphi_{1}(n, q)\equiv (\forall x\in A)( Y(x)=n  \di x\leq_{\R} q).
\ee
Since $Y$ is bijective on $A$, we have $(\forall n\in \N, q\in \Q\cap (0,1))(\neg\varphi_{0}(n, q)\vee \neg \varphi_{1}(n, q))$.
Let $Z\subset \N\times \Q$ be as in $\PSEP$ and note that for $n\in \N, q\in \Q\cap (0,1)$, we have
\be\label{lap2}
(n, q)\in Z\di (\exists x\in A)( Y(x)=n  \wedge x>_{\R} q),
\ee
 \be\label{land2}
 (n, q)\not\in Z\di (\exists x\in A)( Y(x)=n  \wedge x\leq_{\R} q).
\ee
Now proceed as in the proof of Theorem \ref{brilli} to define an enumeration of $A$. 
\end{proof}
\noindent
Finally, let $\PSEP!$ be $\PSEP$ restricted to $Y_{i}^{2}$ such that 
\be\label{ukeen}
(\forall n\in \N)(\exists! f\in 2^{\N})\big[ Y_{0}(f, n)\ne 0 \vee Y_{1}(f, n )\ne 0\big], 
\ee
and all type $1$ quantifiers restricted to $2^{\N}$.
We have the following corollary.
\begin{cor}
The system $\ACAo$ proves ${\PSEP!}\asa \cocode_{1}$.
\end{cor}
\begin{proof}
The forward direction is immediate from the proof of the theorem as \eqref{ukeen} is satisfied by the formulas \eqref{sd2} and \eqref{card2}. 
For the reverse implication, the set $\{f\in 2^{\N}: Y_{0}(f, n)\ne 0 \vee Y_{1}(f, n )\ne 0\}$ is strongly countable.  The enumeration provided by $\cocode_{1}$ readily provides the set $Z$ from $\PSEP!$ and we are done.
\end{proof}
%
%
%

\subsection{Countable unions and the Axiom of Choice}\label{heli2}
In this section, we study the connection between the Bolzano-Weierstrass theorem, the \emph{countable union theorem} for $\R$, and the existence of sets not in the class ${\bf F}_\sigma$. 
By Corollary \ref{PRM}, there are natural versions of the countable union theorem equivalent to $\BWC_{i}$ for $i=0,1$.  

\smallskip

First of all, the Axiom of Choice ($\AC$ for short) is perhaps the most (in)famous axiom of the usual foundations of mathematics, i.e.\ $\ZFC$ set theory.  
It is known that very weak fragments of $\AC$ are independent of $\ZF$, like the \emph{countable union theorem} which expresses that a countable union of countable (or even $2$-element) sets is again countable.  We refer to \cite{heerlijkheid} for an overview of this kind of results on $\AC$, while we note that Cantor already considered the countable union theorem in 1878, namely in \cite{cantor2}*{p.\ 243}.  The countable union theorem involving enumerations and (codes of) analytic sets may be found in second-order RM as \cite{simpson2}*{V.4.10}, i.e.\ the following principle is a quite natural object of study in higher-order RM. 
We discuss the naturalness and generality of $\CUC$ in Remark \ref{flaw}.
%
\begin{princ}[$\CUC$] 
Let $(A_{n})_{n\in \N}$ be a sequence of sets in $\R$ such that for all $n\in \N$, there is an enumeration of $A_{n}$.  
Then there is an enumeration of $\cup_{n\in \N}A_{n}$.  
\end{princ}
Note that we need $(\exists^{2})$ to guarantee that the union in $\CUC$ exists.
As noted above, the countable union theorem \emph{for $2$-element sets} is still unprovable in $\ZF$.  In this light, define $\CUC(2)$ as $\CUC$ where each $A_{n}$ has exactly two elements, i.e.\ 
\be\label{fomp}
(\forall x, y, z\in A_{n})(x=_{\R}y\vee x=_{\R}z)\wedge (\exists w, v\in A_{n})(w\ne_{\R} v).
\ee
The following principle is (possibly) weaker than the countable union theorem according to \cite{heerlijkheid}*{Diagram 3.4, p.\ 23}: 
\emph{$\R$ is not a countable union of countable sets.}
We distill the following principle from the latter.
\begin{princ}[$\RUC$] 
Let $(A_{n})_{n\in \N}$ be a sequence of sets in $\R$ such that for all $n\in \N$, there exists an enumeration of $A_{n}$.  
Then there is $y\in \R$ not in $\cup_{n\in \N}A_{n}$.  
\end{princ}
Note that $\RUC$ fails in the model $\textbf{Q}^{*}$ constructed in the proof of Theorem~\ref{goeddachs}, i.e.\ $\neg \RUC$ is consistent with $\Z_2^\omega$.
By \cite{simpson2}*{II.4.7}, Cantor's theorem (that the reals cannot be enumerated) is provable in $\RCA_{0}$, and hence $\CUC\di \RUC$ over $\ACAo$.
The connection between $\RUC$ and the following principle is however more interesting.  
\begin{princ}[${\bf NF_\sigma}$]
There exists a subset of $\R$ that is not ${\bf F}_\sigma$.
\end{princ}
To be precise, we let ${\bf F}_\sigma$ be the class of sets obtained by closing the class of closed sets under unions of countable subclasses, always assuming that the unions exist.
The following theorem connects $\CUC$ and $\RUC$ to $\BWC_{0} $ and $ \BWC_{1}$.  
\begin{thm}\label{immea}
The system $\ACAo$ proves $\CUC\di \cocode_{0}\di \CUC(2)\di \cocode_{1}$ and ${\bf NF}_\sigma \di \RUC\di \NIN$.  
\end{thm}
\begin{proof}
For the first part, fix non-empty $A\subseteq[0,1]$ and $Y:[0,1]\di \N$ such that the latter is injective on the former. 
Let $x_{0}\in A$ be some element in $A$ and define the sequence of sets $(A_{n})_{n\in \N}$ as follows: 
\be\label{moar}
x\in A_{n}\equiv\big[ [x\in A\wedge Y(x)=n]\vee x=_{\R}x_{0}\big].
\ee
Clearly, for each $n\in\N$, there \emph{exists} an enumeration of $A_{n}$, namely either the sequence $x_{0}, x_{0}, \dots$ or the sequence $x_{0}, y, x_{0}, y, \dots$ where $y\in [0,1]$ satisfies $Y(y)=n$, if such there is.   By $\CUC$, there is an enumeration of $A=\cup_{n\in \N}A_{n}$, yielding $\cocode_{0}$.  Now assume the latter and fix a sequence $(A_{n})_{n\in \N}$ satisfying \eqref{fomp}.  
By the latter, we have  the following 
\be\label{heencent}
(\forall n\in \N)( \exists! x \in [0,1])(\exists! y\in [0,1]) (x, y\in A_{n} \wedge x<_{\R}y ),
\ee
as $A_{n}$ has exactly two elements.  
Recall that $\cocode_{1}\asa !\QFAC^{0,1}$ by \cite{samcount}*{Theorem~3.17}.
Modulo some coding $!\QFAC^{0,1}$ applies to \eqref{heencent}, and let $(x_{n})_{n\in \N}$ and $(y_{n})_{n\in \N}$ be the resulting sequences.  
Use $\exists^{2}$ to remove any reals from $(y_{n})_{n\in \N}$ already in $(x_{n})_{n\in \N}$.
Then $Y:\R\di \N$ is injective on $\cup_{n\in \N}A_{n}$:
\be\label{polp}
Y(x):=
\begin{cases}
0 & x\not \in \cup_{n\in \N}A_{n}\\
2 P(x) & (\exists n\in \N)(x=_{\R}x_{n})\\
2P(x)+1 & (\exists n\in \N)(x=_{\R}y_{n})\\
\end{cases}, 
\ee
where $P(x):=(\mu n)(x\in A_{n})$.  Then $\cocode_{0}$ yields $\CUC(2)$, as required. 
For the implication $\CUC(2)\di \cocode_{1}$, fix $A\subset [0,1]$ such that $Y:[0,1]\di \N$ is bijective on $A$.
Define the set $A_{n}:=\{x\in A: Y(x)=n \vee Y(x)=n+1   \}$ and note that \eqref{fomp} is satisfied.  Applying $\CUC(2)$ yields an enumeration of $A=\cup_{n\in \N}A_{n}$, as required.  

\smallskip

For the second part, suppose $\R=\cup_{n}A_{n}$, where for each $n\in \N$ there \emph{exists} an enumeration of $A_{n}$.  Then all subsets of $\R$ are ${\bf F}_\sigma$ as follows:
for $E\subset \R$, one defines an enumeration of $E\cap A_{n}$ by checking each element in the enumeration of $A_{n}$ for elementhood in $E$.  
Hence, $E=\cup_{n\in \N} [A_{n}\cap E] $ is a countable union of enumerable sets, and therefore ${\bf F}_\sigma$. 
For $\RUC\di \NIN$, suppose $Y:\R\di \N$ is an injection.  
Define a sequence $(A_{n})_{n\in \N}$ as follows $x\in A_{n}\equiv\big[ Y(x)=_{0}n\vee x=_{\R}0\big]$.
Clearly, for each $n\in\N$, there \emph{exists} an enumeration of $A_{n}$.   By $\RUC$, there is $y\in \R$ not in $\cup_{n\in \N}A_{n}$.  However, $\R=\cup_{n\in \N}A_{n}$ by definition, yielding $\RUC\di \NIN$.  
\end{proof}
Assuming $\ACAo + \neg \RUC$, the previous proof implies that all subsets of $\R$ are \textbf{F}$_\sigma$, and considering complements implies that all subsets are also $\textbf{G}_\delta$. 
In stronger systems, the class $\textbf{F}_\sigma \cap \textbf{G}_\delta$ corresponds to $\Delta^0_2$-formulas with function parameters. 

\smallskip

Let $\CUC_{0}(2)$ be $\CUC(2)$ without the second conjunct of \eqref{fomp} and let $\CUC_{1}(2)$ be $\CUC(2)$ where we additionally assume the sets  $A_{n}$ to be pairwise disjoint.
\begin{cor}[$\ACAo$]\label{PRM}
We have $\cocode_{0}\asa \CUC_{0}(2)$ and $\cocode_{1}\asa \CUC_{1}(2)$.
\end{cor}
\begin{proof}
The proof of $\CUC(2)\di \cocode_{1}$ from the theorem yields $\CUC_{1}(2)\di\cocode_{1}$ as $A_{n}:=\{x\in A: Y(x)=2n \vee Y(x)=2n+1   \}$ are indeed pairwise disjoint.  
The proof of $\cocode_{0}\di \CUC(2)$ yields $\cocode_{1}\di \CUC_{1}(2)$ 
as the extra `pairwise disjoint' condition in $\CUC_{1}(2)$ guarantees that $Y$ defined in \eqref{polp} is bijective on $\cup_{n\in \N}A_{n}$.
The proof of $\CUC\di \cocode_{0}$ from the theorem yields a proof of $\CUC_{0}(2)\di\cocode_{0}$ as the sets from \eqref{moar} have at most two elements. 
The proof of $\cocode_{0}\di \CUC(2)$ from the theorem can be adapted as follows: consider the following formula, where the boldface text is different from \eqref{heencent}:
\be\label{heencent4}
(\forall n\in \N)(\exists \textup{ \textbf{at most one} } (x, y)\in \R^{2})(x, y\in A_{n} \wedge x<_{\R} y),
\ee
to which $\BOOT_{C}^{-}$ applies modulo coding.   For the resulting set $X\subset \N$ we have 
\[
(\forall n\in X)( \exists! x \in [0,1])(\exists! y\in [0,1]) (x, y\in A_{n} \wedge x<_{\R}y ).
\]
One now readily modifies \eqref{polp} to the case at hand, which yields an enumeration of all $A_{n}$ that have exactly two elements. 
To enumerate the $A_{n}$ that are singletons, consider the following:
\be\label{heencent9}
(\forall n\in \N\setminus X)(\exists \textup{ at most one } x\in \R)(x\in A_{n}),
\ee
to which $\BOOT_{C}^{-}$ applies modulo coding.   For the resulting set $Z\subset \N$ we have 
\[
(\forall n\in Z)( \exists! x \in [0,1]) (x\in A_{n}  ),
\]
which readily yields the required enumeration. 
\end{proof}
By the previous, one can view $\CUC$ as the sequential version of $\cocode_{0}$.  However, the sequential version of e.g.\ $\BWC_{0}$ is readily proved in $\Z_{2}^{\Omega}$ (and hence $\ZF$).   
By contrast, the sequential version of $\cloq'$ is equivalent to $\CUC$ by Corollary \ref{clong},
\begin{princ}[$\cloq'_{\seq}$]
Let $(X_{n}, \preceq_{n})_{n\in \N}$ be a sequence of dense linear orderings without endpoints, with each $X_{n}\subset \R$ countable.  
Then there is a sequence $(Z_{n})_{n\in \N}$ with $Z_{n}:\R\di \Q$ an order-isomorphism from $(X_{n}, \preceq_{n})$ to $\Q$ for each $n\in \N$. 
\end{princ}
\begin{cor}\label{clong}
The system $\ACAo$ proves $[\cloq'_{\seq}+\IND_{0}]\asa \CUC$. 
\end{cor}
\begin{proof}
For the reverse implication, $\CUC$ yields $\cocode_{0}$ by Theorem \ref{immea}.  
Hence, if $(X_{n}, \preceq_{n})_{n\in \N}$ is as in the antecedent of $\cloq'_{\seq}$, $\cocode_{0}$ implies that for each $X_{n}$, there is an enumeration.  
By $\CUC$, there is a `master' enumeration of $\cup_{n\in \N}X_{n}$.  
Use the well-known `back-and-forth' proof (see \cite{riot}*{p.\ 123}) for each $(X_{n}, \preceq_{n})$, uniformly in $\N$ and based on the master enumeration, to yield a sequence as in the consequence of $\cloq'_{\seq}$. 

\smallskip

For the forward implication, we have access to $\cocode_{0}$ by Corollary \ref{fexpo}.  
Let $(A_{n})_{n\in \N}$ be a sequence as in $\CUC$ and define $A:=\cup_{n\in \N}A_{n}$.  Note that $(\exists^{2})$ shows that each $A_{n}$ is countable via an obvious injection.  
Without loss of generality, we may assume that $\Q\cap A$ is $\emptyset$, since Feferman's $\mu^{2}$ can list all the rationals in a given set of reals.   Now define $X_{n}:= \Q\cup A_{n}$ and $\preceq_{{n}}$ the usual ordering of the reals.  
Let $(Z_{n})_{n\in \N}$ be as provided by $\cloq'_{\seq}$, let $(p_{n})_{n\in \N}$ be the usual list of primes, and let $G:\Q\di (\N\setminus \{0\})$ be an injection.  Define $H(x)$ as $(\mu n)(x\in A_{n})$ and define $Y:\R\di \N$ as $Y(x):=p_{H(x)}^{G(Z_{H(x)}(x))}$.  By definition, $Y$ is an injection on $A$; the latter is therefore countable, and enumerable by Corollary \ref{fexpo}.
\end{proof}
We note in passing that the weak choice principle \textsf{WCC} from \cite{kreep} is intermediate between $\cocode_{0}$ and $\cocode_{1}$ by the previous.  
We also have the following corollary. 
\begin{cor}\label{frsnl}
$\Z_{2}^{\Omega}+\QFAC^{0,1}$ proves $\CUC$; $\Z_{2}^{\omega}+\QFAC^{0,1}$ cannot prove $\RUC$.
\end{cor}
\begin{proof}
For the negative result, $\NIN$ is not provable in $\Z_{2}^{\omega}+\QFAC^{0,1}$ by \cite{dagsamX}*{Theorem~3.2}, while $\RUC\di \NIN$ over $\ACAo$ by Theorem \ref{immea}.
For the positive result, the antecedent of $\CUC$ expresses the following:
\[
(\forall n\in \N)(\exists (x_{m})_{m\in \N})(\forall y\in \R)\big[y\in A_{n}\asa (\exists k\in \N)(x_{k}=_{\R}y)    \big].
\]
Using $\exists^{3}$ and $\QFAC^{0,1}$, there is a `master' sequence, yielding $\CUC$.
\end{proof}
We finish this section with a remark on the naturalness and generality of $\CUC$.
\begin{rem}[$\CUC$, old and new]\label{flaw}\rm
First of all, an $\L_{2}$-version of $\CUC$ for sets represented by analytic codes is proved in \cite{simpson2}*{V.4.10}, inside $\ATR_{0}$.  
Note that enumerable sets are automatically Borel, and therefore analytic.  
Similarly, (codes for) Borel sets are closed under countable unions in second-order RM by \cite{simpson2}*{V.3.3}, also working in $\ATR_{0}$.
Modulo coding, there is thus antecedent for the study of $\CUC$ in second-order RM. 

\smallskip

Secondly, in contrast to the second-order principles from the previous paragraph, $\CUC$ does (seem to) quantify over all enumerable subsets of $\R$.  
This apparent generality of $\CUC$ should not be overstated: an enumerated set is of course measurable (provably having measure zero in $\ACAo$), and the class of (codes for) measurable sets is closed under countable unions in second-order RM, as mentioned in \cite{simpson2}*{X.1.17}.   Similarly, enumerated sets are clearly Borel sets (of low level) in $\ACAo$.   
Hence, $\CUC$ is of a level of generality comparable to what one studies in RM, but formulated with third-order characteristic functions rather than second-order codes.    

\smallskip

Thirdly, in Section \ref{BV}, we connect $\cocode_{0}$ to theorems pertaining to bounded variation (and related concepts), like the \emph{Jordan decomposition theorem} as in Theorem \ref{drd}.
On one hand, this theorem readily implies $\cocode_{0}$, while the reversal \emph{should} go through, seeing as though functions of bounded variation only have countably many points of discontinuity.  
Indeed, an enumeration of the latter set even guarantees that Jordan's original proof (\cite{jordel}) of the Jordan decomposition theorem goes through.  Try as we might, the aforementioned reversal only goes through assuming the following (seemingly trivial) fragment of the countable union theorem, which however does not\footnote{Note that $\neg \NIN$ implies $\CUC_{\fin}$, while $\Z_{2}^{\omega}+\QFAC^{0,1}$ does not prove $\NIN$ by \cite{dagsamX}*{\S3}.} even imply $\NIN$ over $\Z_{2}^{\omega}+\QFAC^{0,1}$.
\begin{princ}[$\CUC_{\fin}$]
Let $(X_{n})_{n\in \N}$ be subsets of $\R$ such that $\cup_{n\in \N}X_{n}$ is not countable.  Then $X_{m}$ is not finite for some $m\in \N$.
\end{princ}
Recall our notion of `finite set' from Definition \ref{plonk}, to be discussed in detail in Section \ref{truerm}.
In the below, we even  obtain equivalences involving $\CUC_{\fin}$, i.e.\ the countable union theorem is a natural/useful object of study in this context.   
\end{rem}

\subsection{Bounded variation and related concepts}\label{BV}
In this section, we establish an equivalence between $\BWC_{0}$ and the well-known \emph{Jordan decomposition theorem} as in Theorem \ref{drd}.
We also obtain other equivalences involving theorems about \emph{bounded variation} and \emph{regulated} functions.  
We introduce definitions for the previous italicised notions in Section~\ref{deffer}, while our main results are in Section~\ref{truerm}.
The latter results provide some non-trivial motivation for our choice of definition of closed and finite set, as discussed in Section \ref{crux}.

\subsubsection{Definitions: bounded variation and related notions}\label{deffer}
We formulate the definitions of bounded variation and regulated functions, as well as some background. 

\smallskip

Firstly, the notion of \emph{bounded variation} (often abbreviated $BV$ below) was first explicitly\footnote{Lakatos in \cite{laktose}*{p.\ 148} claims that Jordan did not invent or introduce the notion of bounded variation in \cite{jordel}, but rather discovered it in Dirichlet's 1829 paper \cite{didi3}.} introduced by Jordan around 1881 (\cite{jordel}) yielding a generalisation of Dirichlet's convergence theorems for Fourier series.  
Indeed, Dirichlet's convergence results are restricted to functions that are continuous except at a finite number of points, while $BV$-functions can have infinitely many points of discontinuity, as already studied by Jordan, namely in \cite{jordel}*{p.\ 230}.
Nowadays, the \emph{total variation} of a function $f:[a, b]\di \R$ is defined as follows:
\be\label{tomb}\textstyle
V_{a}^{b}(f):=\sup_{a\leq x_{0}< \dots< x_{n}\leq b}\sum_{i=0}^{n} |f(x_{i})-f(x_{i+1})|.
\ee
If this quantity exists and is finite, one says that $f$ has bounded variation on $[a,b]$.
Now, the notion of bounded variation is defined in \cite{nieyo} \emph{without} mentioning the supremum in \eqref{tomb}; this approach can also be found in \cites{kreupel, briva, brima}.  
Hence, we shall distinguish between the two notions in Definition \ref{varvar}.  As it happens, Jordan seems to use item \eqref{donp} of Definition \ref{varvar} in \cite{jordel}*{p.\ 228-229}.
This definition suggests a two-fold variation for any result on functions of bounded variation, namely depending on whether the supremum \eqref{tomb} is given, or only an upper bound on the latter.  
\bdefi[Variations on variation]\label{varvar}
\begin{enumerate}  
\renewcommand{\theenumi}{\alph{enumi}}
\item The function $f:[a,b]\di \R$ \emph{has bounded variation} on $[a,b]$ if there is $k_{0}\in \N$ such that $k_{0}\geq \sum_{i=0}^{n} |f(x_{i})-f(x_{i+1})|$ 
for any partition $x_{0}=a <x_{1}< \dots< x_{n-1}<x_{n}=b  $.\label{donp}
\item The function $f:[a,b]\di \R$ \emph{has {a} variation} on $[a,b]$ if the supremum in \eqref{tomb} exists and is finite.\label{donp2}
\end{enumerate}
\edefi
Secondly, the fundamental theorem about $BV$-functions is formulated as follows.
\begin{thm}[Jordan decomposition theorem, \cite{jordel}*{p.\ 229}]\label{drd}
A $BV$-function $f : [0, 1] \di \R$ is the difference of  two non-decreasing functions $g, h:[0,1]\di \R$.
\end{thm}
Theorem \ref{drd} has been studied via second-order representations in \cites{groeneberg, kreupel, nieyo, verzengend}.
The same holds for constructive analysis by \cites{briva, varijo,brima, baathetniet}, involving different (but related) constructive enrichments.  
Now, $\ACA_{0}$ suffices to derive Theorem \ref{drd} for various kinds of second-order \emph{representations} of $BV$-functions in \cite{kreupel, nieyo}.  
By contrast, our results imply that $\Z_{2}^{\omega}+\QFAC^{0,1}$ cannot prove the third-order version of Theorem \ref{drd}, as the latter is equivalent to $\BWC_{0}$ over a suitable base theory (see Theorem \ref{bruhathm}). 
Nonetheless, the third-order Jordan decomposition theorem does not imply much comprehension, by the following remark.  
\begin{rem}[Comprehension and Jordan decompositions]\label{blafte}\rm
The third-order version of the Jordan decomposition theorem (Theorem \ref{drd}) implies neither $(\exists^{2})$ nor any theorem of $\Z_{2}$ not provable in $\ACA_{0}$, working over $\RCAo$.  
Indeed, the $\ECF$-translation (Remark~\ref{ECF}) of the former is implied by $\textsf{Jordan}_{\textsf{cont}}$, the second-order version of Theorem \ref{drd} from \cite{nieyo} and provable in $\ACA_{0}$. 
By contrast, $\ECF$ translates $(\exists^{2})$ to `$0=1$' while second-order sentences are translated to themselves.  
\end{rem}
Thirdly, Jordan proves in \cite{jordel3}*{\S105} that $BV$-functions are exactly those for which the notion of `length of the graph of the function' makes sense.  In particular, $f\in BV$ if and only if the `length of the graph of $f$', defined as follows:
\be\label{puhe}\textstyle
L(f, [0,1]):=\sup_{0=t_{0}<t_{1}<\dots <t_{m}=1} \sum_{i=0}^{m-1} \sqrt{(t_{i}-t_{i+1})^{2}+(f(t_{i})-f(t_{i+1}))^{2}  }
\ee
exists and is finite by \cite{voordedorst}*{Thm.\ 3.28.(c)}.  In case the supremum in \eqref{puhe} exists (and is finite), $f$ is also called \emph{rectifiable}.  
Rectifiable curves predate $BV$-functions: in \cite{scheeffer}*{\S1-2}, it is claimed that \eqref{puhe} is essentially equivalent to Duhamel's 1866 approach from \cite{duhamel}*{Ch.\ VI}.  Around 1833, Dirksen, the PhD supervisor of Jacobi and Heine, already provides a definition of arc length that is (very) similar to \eqref{puhe} (see \cite{dirksen}*{\S2, p.\ 128}), but with some conceptual problems as discussed in \cite{coolitman}*{\S3}.

\smallskip

Fourth, a function is \emph{regulated} (called `regular' in \cite{voordedorst}) if for every $x_{0}$ in the domain, the `left' and `right' limit $f(x_{0}-)=\lim_{x\di x_{0}-}f(x)$ and $f(x_{0}+)=\lim_{x\di x_{0}+}f(x)$ exist.  
Scheeffer studies discontinuous regulated functions in \cite{scheeffer} (without using the term `regulated'), while Bourbaki develops Riemann integration based on regulated functions in \cite{boerbakies}.  
Now, $BV$-functions are regulated (see Theorem \ref{flima}), while Weierstrass' `monster' function is a natural example of a regulated function not in $BV$.  
An interesting observation about regular functions and continuity is as follows.
\begin{rem}[Continuity and the Axiom of Choice]\label{atleast}\rm
As discussed in \cite{kohlenbach2}*{\S3}, the \emph{local} equivalence for functions on Baire space between sequential and `epsilon-delta' continuity can be proved in $\RCAo+\QFAC^{0,1}$, but not in $\ZF$.  
By the final item in Theorem \ref{flima}, this equivalence for \emph{regulated} functions is provable in $\ACAo$.
\end{rem}
%
Finally, the {Jordan decomposition theorem} as in Theorem \ref{drd} shows that a $BV$-function can be `decomposed' as the difference of monotone functions. 
This is however not the only result of its kind: Sierpi\'{n}ski e.g.\ establishes in \cite{voordesier} that for regulated $f:[0,1]\di \R$, there are $g, h$ such that $f=g\circ h$ with $g$ continuous and $h$ strictly increasing on their respective domains. 
 
 \subsubsection{Bounded variation and Reverse Mathematics}\label{truerm}
In this section, we develop the RM of the Jordan decomposition theorem and related results on bounded variation and regulated functions.  
As will become clear, the principle $\CUC_{\fin}$ from Remark~\ref{flaw} is central to this enterprise.  

\smallskip

First of all, we recall our particular notion of `finite set' to be used in $\CUC_{\fin}$ and provide some motivation in Remark \ref{diunk} right below.   On a historical note, the study of various definitions of finite set (in set theory) was the topic of Mostowski's dissertation, as suggested by Tarski (\cite{moserover}*{p.\ 18-19}). 
\begin{defi}[Finite]\label{deadd}\rm
Any $X\subset \R$ is \emph{finite} if there is $N\in \N$ such that for any finite sequence $(x_{0}, \dots, x_{N})$ of distinct reals, there is $i\leq N$ such that $x_{i}\not \in X$.
\edefi
The number $N\in \N$ from the previous definition is called an \emph{upper bound} on the size of the finite set $X\subset \R$, and we use `$|X|\leq N$' as purely symbolic notation for this.
Note that Definition \ref{deadd} is not circular as `finite sequences of reals' are just objects of type $1$, modulo coding using $\exists^{2}$.
We now motivate Definition \ref{deadd}.
\begin{rem}[Finite sets by any other name]\label{diunk}\rm
First of all, working in set theory, the various definitions\footnote{In $\ZF$, a set $A$ is `finite' if there is a bijection to $\{0, 1, \dots, n\}$ for some $n\in \N$; a set $A$ is `Dedekind finite' if any injective mapping from $A$ to $ A$ is also surjective.\label{krukk}} of `finite set' are not equivalent over $\ZF$, while countable choice suffices to establish the equivalence (\cite{jechp}).  Hence, it should not be a surprise that studying finite sets in weak systems requires one to choose a specific definition.  


\smallskip

Secondly, consider the following set where $f$ is a function of bounded variation:
\be\label{lagel2}\textstyle
A_{n}:=\big\{x\in (0,1): |f(x+)- f(x)|>\frac1{2^{n}} \vee |f(x-)- f(x)|>\frac1{2^{n}}\big\} 
\ee
This set is finite as each element of $A_{n}$ contributes at least $\frac{1}{2^{n}}$ to the total variation.  
Finite as $A_{n}$ may be, we are unable to exhibit an injection from $A_{n}$ to $\{0,1, \dots, k\}$ for some $k\in \N$, say computable in some $\SS_{m}^{k}$ (see Remark \ref{dichtbij} for details).
By contrast, $A_{n}$ is trivially finite in the sense of Definition \ref{deadd} in $\ACAo$.  

\smallskip

%
In conclusion, \emph{if} one wants to work in a weak logical system, \emph{then} (certain) finite sets that `appear in the wild' are best studied via Definition \ref{deadd}, and not the definition from Footnote~\ref{krukk} involving bijections or injections.  Moreover, Theorem~\ref{xruc} suggests that $\IND_{0}$ (and $\cocode_{0}$) does not suffice to study finite sets as in Definition \ref{deadd}; as noted in Remark \ref{flaw}, we indeed seem to need $\CUC_{\fin}$.
\end{rem}
Secondly, we need Theorem \ref{flima} to establish basic properties of $BV$ and regulated functions.
We shall make (seemingly essential) use of the following fragment of the induction axiom, which also follows from $\QFAC^{0,1}$.
\bdefi[$\IND_{2}$]
Let $Y^{2}, k^{0}$ satisfy $(\forall n\leq k)(\exists f\in 2^{\N})(Y(f, n)=0)$.  
There is $w^{1^{*}}$ such that $(\forall n\leq k)(\exists i<|w|)(Y(w(i), n)=0)$.
\edefi
Note that we use the `standard' definition of left and right limits, i.e.\ as in \eqref{lopoo}. 
\begin{thm}[$\ACAo$]\label{flima}~
\begin{itemize}
\item Assuming $\IND_{2}$, any $BV$-function $f:[0,1]\di \R$ is regulated.
\item Any monotone function $f:[0,1]\di \R$ has bounded variation.   
\item For any monotone function $f:\R\di \R$, there is a sequence $(x_{n})_{n\in \N}$ that enumerates all $x\in [0,1]$ such that $f$ is discontinuous at $x$.  
\item For regulated $f:[0,1]\di \R$ and $x\in [0,1]$, $f$ is sequentially continuous at $x$ if and only if $f$ is epsilon-delta continuous at $x$.
\item For finite $X\subset [0,1]$, the function $\mathbb{1}_{X}$ has bounded variation. 
\end{itemize}
\end{thm}
\begin{proof}
For the first item, assume $f(c-)$ does not exist for $c\in (0,1]$.  
We obtain a contradiction using $\QFAC^{0,1}$ and then using $\IND_{2}$.  
Hence, there is $\eps>0$ with 
\be\label{tokkie}\textstyle
(\forall k \in \N)(\exists x, y\in (c-\frac{1}{2^{k}}, c)  )(x<y \wedge |f(x)-f(y)|>\eps  ).  
\ee
Apply $\QFAC^{0,1}$ to \eqref{tokkie}; modify the resulting sequence $(x_{n}, y_{n})_{n\in \N}$ to guarantee 
\[\textstyle
x_{m}<y_{m}<c-\frac{1}{2^{m+1}}<x_{m+1}<y_{m+1}<c-\frac{1}{2^{m+1}}
\]
for large enough $m\in \N$.  By definition, $|f(x_{k})-f(y_{k})|>\eps$ for large enough $k\in \N$, i.e.\ collecting enough such points in a partition, the associated variation is arbitrary large.  
We now observe how the previous proof is readily modified: apply $\IND_{2}$ to \eqref{tokkie} after choosing large enough (relative to the variation of $f$) upper bound on $k$ in \eqref{tokkie}.
Hence, $f(c-)$ must exist and the other cases follow in the same way.

\smallskip

For the second part, assume $f:[0,1]\di \R$ is monotone.  Then the usual telescoping sum trick implies that the total variation of $f$ as in \eqref{tomb} exists and equals $|f(0)-f(1)|$.
The third part is follows from \cite{dagsamXII}*{Lemma 7}, which applies to $[0,1]$ but trivially generalises to $\R$.  

\smallskip

For the fourth item, let $f:[0,1]\di \R$ be regulated and fix $x_{0}\in [0,1]$.  We only need to prove the forward implication, i.e.\ assume $f$ is sequentially continuous at $x_{0}$.  
To show that $f(x_{0}-)=f(x_{0})$, consider $y_{n}:= x_{0}-\frac{1}{2^{n+1}}$ and note that $(y_{n})_{n\in \N}$ converges to $x_{0}$, implying that $(f(y_{n}))_{n\in \N}$ converges to $f(x_{0})$.
Now consider the definition of `the left limit $f(x_{0}-)$ exists' as follows:
\be\label{lopoo}\textstyle
(\exists y\in \R)(\forall k\in \N)(\exists N\in \N)( \forall z \in ( x_{0}-\frac{1}{2^{N}}, x_{0}))(|f(z)-y|<\frac{1}{2^{k}})
\ee
Since $(y_{n})_{n\in \N}$ converges to $x_{0}$ and $(f(y_{n}))_{n\in \N}$ converges to $f(x_{0})$, we have $y=f(x_{0})$ in \eqref{lopoo}.
In the same way, one shows that $f(x_{0}+)=f(x_{0})$.  Then \eqref{lopoo} and the associated `right limit' version imply that $f$ is epsilon-delta continuous at $x_{0}$.

\smallskip

For the fifth item, fix finite $X\subset [0,1]$ with $N\in \N$ as in Definition \ref{deadd}.  Now suppose $f(x):=\mathbb{1}_{X}(x)$ does not have bounded variation, i.e.\ for any $n\in \N$, there is a partition $ x_{0}=0,x_{1}, \dots,x_{k}, x_{k+1}=1$ of $[0,1]$ such that $n+5\leq \sum_{i=0}^{k}|f(x_{i+1})-f(x_{i})|$.
By the definition of $f$, the latter inequality implies that there are $i_{0}, \dots, i_{n}\leq k$ such that $x_{i_{j}}\in X$ for $j\leq n$.
Taking $n=N+1$, we obtain a contradiction.
\end{proof}
Thirdly, we can now connect the Jordan decomposition theorem and $\cocode_{0}$.
Note that `bounded variation' refers to item \eqref{donp} in Definition \ref{varvar}.
\begin{thm}[$\ACAo+\IND_{2}+\CUC_{\fin}$]\label{bruhathm} 
The following are equivalent.
\begin{enumerate}
\renewcommand{\theenumi}{\roman{enumi}}
\item The principle $\cocode_{0}$.\label{easy1}
\item The Jordan decomposition theorem \(Theorem \ref{drd}\). \label{easy2}
\item  For a $BV$-function $f:[0,1]\di \R$, there is a sequence enumerating all points where $f$ is discontinuous.\label{easy3}
\end{enumerate}
The previous upward implications are provable over $\ACAo$.
Assuming $\QFAC^{0,1}$, the above are equivalent to the following. 
\begin{enumerate}
\renewcommand{\theenumi}{\roman{enumi}}
\setcounter{enumi}{3}
\item For regulated $f:[0,1]\di \R$, there is a sequence enumerating all points where $f$ is discontinuous.\label{easy4}
\item \(Sierpi\'{n}ski\) For regulated $f:[0,1]\di \R$, there are $g, h$ such that $f=g\circ h$ with $g$ continuous and $h$ strictly increasing on their interval domains. \label{easy5}
\end{enumerate}
The previous upward implications are provable over $\ACAo+\IND_{2}$.
\end{thm}
\begin{proof}
The equivalence $\eqref{easy2}\asa \eqref{easy3}$ follows from Theorem \ref{flima} and the usual proof of the Jordan decomposition theorem.
Indeed, we can `imitate' the supremum in \eqref{tomb} as follows: use $\mu^{2}$ to define, for any $x\in [0,1]$, the following:
\be\label{VX}\textstyle
V(x):= \sup_{0\leq y_{0}< \dots< y_{n}\leq x}\sum_{i=0}^{n} |f(y_{i})-f(y_{i+1})|,
\ee
where $(y_{i})_{i\in \N}$ is the sequence consisting of $\Q\cap [0,1]$ together with the sequence provided by item \eqref{easy3}.
Trivially, $g(x):=\lambda x.V(x)$ is increasing on $[0,1]$ and the same holds for $h(x):= V(x)-f(x)$.  Indeed, for $0\leq y< z\leq 1$, we have
\[
h(z)-h(y)=V(z)-f(z)-V(y) + f(y)= (V(z)-V(y))-(f(z)-f(y))\geq 0,
\]
where the final inequality follows from the definition of $V$.  
We now have $f(x)-g(x)=h(x)$ for all $x\in [0,1]$, yielding the Jordan decomposition theorem.  

\smallskip

For the implication $\eqref{easy3}\di \eqref{easy1}$, fix $A\subset [0,1]$ and $Y:[0,1]\di \N$ injective on $A$.  
Define $f(x)$ as $\frac{1}{2^{Y(x)+1}}$ in case $x\in A$, and $0$ otherwise. 
Clearly, $f\in BV$ as any sum $\sum_{i=0}^{n} |f(x_{i})-f(x_{i+1})|$ is at most $\sum_{i=0}^{n+5}\frac{1}{2^{i+1}}$, which is bounded by $1$ for any $n\in \N$.
The points of discontinuity for $f$ are exactly the points of $A$, and $\cocode_{0}$ follows. 

\smallskip

For the implication $\eqref{easy1}\di \eqref{easy3}$, fix a $BV$-function $f:[0,1]\di \R$ and $n\in \N$.  
We may assume that the upper bound as in item \eqref{donp} in Def.\ \ref{varvar} is $1$.  
The first item of Theorem \ref{flima} guarantees that $f$ is regular.  Now define the following set
\be\label{lagel}\textstyle
A_{n}:=\big\{x\in (0,1): |f(x+)- f(x)|>\frac1{2^{n}} \vee |f(x-)- f(x)|>\frac1{2^{n}}\big\} 
\ee
which is finite (in the sense of Definition \ref{deadd}).  Indeed, assuming $A_{n}$ were not finite, there are arbitrary long finite sequences of elements of $A_{n}$.
However, each element of $A_{n}$ contributes at least $\frac{1}{2^{n}}$ to the variation of $f$, a contradiction.
Hence, $A_{n}$ is finite (and has at most $2^{n}$ elements).  
Using the contraposition of $\CUC_{\fin}$, the union $A:=\cup_{n\in \N}A_{n}$ is countable. 
This union can now be enumerated thanks to $\cocode_{0}$, yielding a sequence listing all points of discontinuity of $f$.  

\smallskip

The implications \eqref{easy5}$\di$\eqref{easy4}$\di$\eqref{easy3} are immediate by Theorem \ref{flima}.
For $\eqref{easy4}\di \eqref{easy5}$, fix regulated $f:[0,1]\di \R$ and consider the proof of \cite{voordedorst}*{Theorem 0.36, p.\ 28}, going back to \cite{voordesier}. 
This proof establishes the existence of $g, h$ such that $f=g\circ h$ with $g$ continuous and $h$ strictly increasing. 
Moreover, one finds an \emph{explicit construction} (modulo $\exists^{2}$) of the function $h$ required, \emph{assuming} a sequence listing all points of discontinuity of $f$ on $[0,1]$.  
The function $g$ is then defined as $\lambda y.f(h^{-1}(y))$ where $h^{-1}$ is the inverse of $h$, definable using $\exists^{2}$. 

\smallskip

Finally, we shall make use of $\QFAC^{0,1}$ to prove \eqref{easy1}$\di$\eqref{easy4}; fix regulated $f:[0,1]\di \R$ and $n\in \N$ and note that $A_{n}$ as in \eqref{lagel} is again finite.  
Indeed, assuming $A_{n}$ were not finite, $\QFAC^{0,1}$ provides a sequence $(x_{j})_{j\in \N}$ of elements of $A_{n}$.
By the Bolzano-Weierstrass theorem, this sequence has a convergent sub-sequence, say with limit $c\in [0,1]$.  However, $f(c+)$ and $f(c-)$ do not exist by the definition of $A_{n}$ (via the usual epsilon-delta argument), a contradiction. 
In conclusion, the union $A:=\cup_{n\in \N}A_{n}$ can now be enumerated, thanks to item \eqref{easy1} and $\CUC_{\fin}$.
\end{proof}
The use of $\QFAC^{0,1}$ in the theorem can be avoided in various ways, one of which is the principle $\NCC$ from \cite{dagsamX}.  We will explore this in a follow-up paper. 

\smallskip

Fourth, we establish a (more) elegant result as in Theorem \ref{bruha}.  In the latter, the \emph{uniform finite union theorem} expresses the existence of $h:\N\di \N$ such that $|X_{n}|\leq h(n)$ for a sequence of finite sets $(X_{n})_{n\in \N}$ in $[0,1]$.  
The \emph{finite union theorem} expresses (only) that for such a sequence, each $\cup_{n\leq k}X_{n}$ is finite for $k\in \N$.  
Regarding item \eqref{sucker}, Principle \ref{akku} was studied in Corollary~\ref{fara} and we can now obtain an equivalence involving the former and $\cocode_{0}$. 
\begin{thm}[$\ACAo+\QFAC^{0,1}$]\label{bruha}
The following are equivalent.
\begin{enumerate}
\renewcommand{\theenumi}{\alph{enumi}}
\item The combination $\CUC_{\fin}+\cocode_{0}$.\label{ha}
\item For regulated $f:\R\di \R$, there is a sequence enumerating the points of discontinuity.  \label{konk3}
\item For regulated $f:[0,1]\di \R$, there is a sequence enumerating the points of discontinuity.  \label{konk}
\item The \emph{uniform finite union theorem} plus the Jordan decomposition theorem. \label{konk4}
\item The \emph{uniform finite union theorem} plus: for $f:[0,1]\di \R$ in $BV$, there is a sequence enumerating the points of discontinuity.  \label{konk2}
\item The \emph{finite union theorem} plus the Jordan decomposition theorem on the half-line: for $f:\R\di \R$ with bounded variation on $[0,y]$ for any $y\in \R^{+}$, there are monotone $g, h$ such that $f(x)=g(x)-h(x)$ for any $x\geq 0$.  \label{konk6}
\item The \emph{finite union theorem} plus: for $f:\R\di \R$ with bounded variation on $[0,y]$ for any $y\in \R^{+}$, there is a sequence enumerating the points of discontinuity of $f$ on $[0, +\infty)$.  \label{konk7}
\item A non-enumerable and closed set in $\R$ has a limit point \(Principle \ref{akku}\).  \label{sucker}
\end{enumerate}
\end{thm}
\begin{proof}
First of all, we derive the following basic properties concerning finite sets, working in our base theory $\ACAo+\QFAC^{0,1}$.  
\begin{enumerate}
\item[(x1)] Any item \eqref{ha}-\eqref{sucker} implies that \emph{a finite set of reals can be enumerated}.  
\item[(x2)] Item \eqref{konk} implies the \emph{finite union theorem}.  
\item[(x3)] Item \eqref{konk3} implies the \emph{uniform finite union theorem} and $\CUC_{\fin}$.  
\end{enumerate}
For item (x1), a finite set has characteristic function that is in $BV$ and regulated by Theorem \ref{flima}, assuming $\IND_{2}$ which follows from $\QFAC^{0,1}$.    
Hence, over our base theory, items \eqref{ha}-\eqref{konk7} imply that a finite set can be enumerated (as a finite sequence, using $\mu^{2}$), where we note the third item of Theorem \ref{flima}. 
Since finite sets do not have limit points, item (x1) also holds for item \eqref{sucker}.  

\smallskip

For item (x2), let $(X_{n})_{n\in \N}$ be a sequence of finite sets.
We may assume $0,1\not \in \cup_{n\in \N}X_{n}$.  
Now consider $f_{k}:[0,1]\di \N$ for $k\geq 2$ defined as follows: 
define $Y_{i}:= \{ y\in (\frac{i}{k}, \frac{i+1}{k}):  k(y-\frac{i}{k})\in X_{i} \}$ and $f_{k}(x):=\sum_{i=0}^{k}\mathbb{1}_{Y_{i}}(x)$.
By definition, $Y_{i}$ is the set $X_{i}$ for $i\leq k$, but shrunk by a factor $\frac{1}{k}$ and moved to $(\frac{i}{k}, \frac{i+1}{k})$.  
Hence, $Y_{i}$ is finite for $i\leq k$ and since $f_{k}(x)$ equals $\mathbb{1}_{Y_{i}}(x)$ for $x\in [\frac{i}{k}, \frac{i+1}{k}]$, the function $f_{k}$ is regulated by Theorem \ref{flima}.  
Thus, item~\eqref{konk} implies that the points of discontinuity of $f_{k}$ can be enumerated, which means $\cup_{i\leq k}Y_{i}$ can be enumerated. 
Using $\mu^{2}$ and the latter enumeration, one finds an upper bound $N_{i}\in \N$ for each $Y_{i}$.  Taking the sum, $\cup_{i\leq k}Y_{i}$ (and hence $\cup_{i\leq k}X_{i}$) is finite.
One obtains item (x3) in the same way: let $Z_{i}$ be the set $X_{i}$ moved to $(i+1, i+2)$ without shrinking for $i\in \N$.  Then the function $\mathbb{1}_{\cup_{n\in \N}Z_{n}}$ is regular on $\R$ and item \eqref{konk3} provides an enumeration of $\cup_{n\in \N}Z_{n}$, which readily yields $\CUC_{\fin}$.  Using this enumeration and $\mu^{2}$, one obtains the function $h:\N\di \N$ as in the \emph{uniform} finite union theorem. 

\smallskip

Secondly, we establish $\eqref{ha} \di \eqref{konk3} \di \eqref{konk}\di \eqref{ha}$.
Now, \eqref{ha} $\di$ \eqref{konk3} follows from the proof of  \eqref{easy1}$\di$\eqref{easy4} in Theorem \ref{bruhathm} by replacing $[0,1]$ by $\R$. 
In turn, \eqref{konk3} $\di$ \eqref{konk} is trivial while \eqref{konk} $\di$ \eqref{ha} is proved as follows: let $(X_{n})_{n\in \N}$ be a sequence of finite sets in $[0,1]$ and define the following function:
\be\label{varivari}
g(x):=
\begin{cases}
\frac{1}{2^{n}} & x \in X_{n} \textup{ and $n$ is the least such number}\\
0 & \textup{otherwise}
\end{cases}.
\ee
To show that $g:[0,1]\di \R$ is regulated, fix $x\in [0,1]$ and $k\in \N$.  Then $\cup_{i\leq k}X_{i}$ is finite by the finite union theorem, which is available due to item (x2) from the first paragraph of this proof.  
Then $(\exists m^{0})(\forall y\in B(x, \frac{1}{2^{m}})\setminus \{x\})(y\not \in \cup_{i\leq k}X_{i})$ readily\footnote{Suppose $(\forall m^{0})(\exists  y\in B(x, \frac{1}{2^{m}})\setminus \{x\})(y\in \cup_{i\leq k}X_{i})$ and apply $\QFAC^{0,1}$ to obtain a sequence in $\cup_{i\leq k}X_{i}$ converging to $x$.  Using $\mu^{2}$, one modifies this sequence to guarantee it consists of pairwise disjoint reals.  This however contradicts the finiteness of $\cup_{i\leq k}X_{i}$.} follows by contradiction.   By definition, $g(x)\leq \frac{1}{2^{k+1}}$ on this punctured disc, i.e.\ $g$ becomes arbitrarily small near $x$, implying $g(x+)=0=g(x-)$.  
Item \eqref{konk} now provides a list $(x_{n})_{n\in \N}$ with all points where $g$ is discontinuous; this sequence also enumerates $\cup_{n\in \N}X_{n}$. 
Indeed, $g(x_{m}-)=0=g(x_{m}+)$ implies that $g(x_{m})>0$ as $g$ must be discontinuous at $x_{m}$; by \eqref{varivari}, $x_{m}$ is in $\cup_{n\in \N}X_{n}$. 
Similarly, if $y$ is in the latter union, we have $g(y)>0$ by \eqref{varivari}; hence $g$ is discontinuous at $y$, implying there is $m\in \N$ with $y=x_{m}$.
Hence, $\cup_{n\in \N}X_{n}$ can be enumerated, which immediately implies $\cocode_{0}$ and $\CUC_{\fin}$.  
We have established \eqref{ha} $\asa$ \eqref{konk3} $\asa$ \eqref{konk}. 

\smallskip

Thirdly, we show that $\eqref{konk3}\di \eqref{konk4}\di \eqref{konk2}\di \eqref{ha}$.    
The implication $\eqref{konk3}\di \eqref{konk4}$ follows from item (x3) and Theorem \ref{bruhathm}.  
The implication $\eqref{konk4}\di \eqref{konk2}$ follows by the third item of Theorem \ref{flima}. 
For the implication \eqref{konk2}$\di$\eqref{ha}, modify \eqref{varivari} as follows:
\be\label{varivari2}
g(x):=
\begin{cases}
\frac{1}{2^{n}}\frac{1}{h(n)+1} & x \in X_{n} \textup{ and $n$ is the least such number}\\
0 & \textup{otherwise}
\end{cases},
\ee
where $h$ is as provided by the uniform finite union theorem.  Since $|X_{n}|\leq h(n)$ for all $n\in \N$, $g$ as in \eqref{varivari2} is in $BV$, with variation bounded by $1$.  
Applying item~\eqref{konk2}, one obtains an enumeration of $\cup_{n\in \N}X_{n}$, as required for item \eqref{ha}.
By the previous paragraph, we obtain $\eqref{ha}\asa \eqref{konk3}\asa \eqref{konk} \asa \eqref{konk4}\asa \eqref{konk2}$.    

\smallskip

Fourth, we show that $\eqref{konk3}\di \eqref{konk6}\di \eqref{konk7}\di \eqref{ha}$.    
The implication $\eqref{konk3}\di \eqref{konk6}$ follows from item (x3) and the generalisation of \eqref{VX} to arbitrary intervals $[0, y]$ for $y>0$; 
the second part is essentially the same as the proof of $\eqref{easy3}\di \eqref{easy2}$ in Theorem~\ref{bruhathm}. 
The implication $\eqref{konk6}\di \eqref{konk7}$ follows by the third item of Theorem~\ref{flima}. 
To prove that item \eqref{konk7} implies item \eqref{ha}, let $(X_{n})_{n\in \N}$ be a sequence of finite sets in $(0,1)$.  
Let $Z_{i}$ be the set $X_{i}$ moved to $(i+1, i+2)$ without shrinking for $i \in\N$. As above, the function $\mathbb{1}_{\cup_{n\in \N}Z_{n}}$ satisfies the conditions of item \eqref{konk7}.  
Indeed, on the interval $[0, y]$ with $0<y\leq m\in \N$, the function $\mathbb{1}_{\cup_{n\in \N}Z_{n}}$ reduces to $\mathbb{1}_{\cup_{n\leq m}Z_{n}}$, and the latter has bounded variation 
by the finite union theorem and the final item in Theorem~\ref{flima}.  An enumeration of the points of discontinuity of $\mathbb{1}_{\cup_{n\in \N}Z_{n}}$ readily yields an enumeration of $\cup_{n\in \N}X_{n}$, as required for item \eqref{ha}.
By the previous paragraph, we obtain $\eqref{ha}\asa \eqref{konk3}\asa \dots \asa \eqref{konk7}$, i.e.\ all that remains is item \eqref{sucker},

\smallskip

Finally, we prove $\eqref{ha}\di \eqref{sucker}\di \eqref{konk2}$, finishing the theorem.
Hence, assume \eqref{sucker} and fix $f:[0,1]\di \R$ in $BV$ and consider $A_{n}$ a in \eqref{lagel}, which is well-defined thanks to Theorem \ref{flima}. 
The set $A_{n}$ is also finite as in the proof of Theorem \ref{bruhathm} and we may assume $0,1 \not \in A_{n}$ for $n\in \N$.  Now let $B_{n}$ be a copy of $A_{n}$ translated from $[0,1]$ to $[n+1, n+2]$ for $n\in \N$.  
Then $B:=\cup_{n\in \N}B_{n}$ has no limit points, which one proves (by contradiction) using $\QFAC^{0,1}$ and the Bolzano-Weierstrass theorem.  
Hence, item~\eqref{sucker} yields an enumeration of $B$ and hence a sequence listing all points where $f$ is discontinuous.  Similarly, item \eqref{sucker} implies the uniform finite union theorem.  
For the implication \eqref{ha}$\di$\eqref{sucker}, fix closed $A\subset\R$ with no limit points.  
Then $A_{n}:=A\cap [-n, n ]$ is finite for any $n\in \N$, which one proves (by contradiction) using $\QFAC^{0,1}$ and the Bolzano-Weierstrass theorem.  
By $\CUC_{\fin}$, $\cup_{n\in \N}A_{n}$ is countable, and can be enumerated using $\cocode_{0}$, and we are done. 
\end{proof}
By the proof of Theorem \ref{bruha}, item \eqref{konk} implies the finite union theorem, while the same does not seem to hold for the Jordan decomposition theorem.  
We believe this is due to fact that `regulated' is a local property while `bounded variation' is a global property (of the domain).  Moreover, there are \emph{many and very different} intermediate spaces (see \cite{voordedorst} or \cite{dagsamXII}*{Remark 4.13}) between the space of regulated and of $BV$-functions; each of these intermediate spaces yields an equivalent generalisation of e.g.\ item \eqref{konk2} in Theorem \ref{bruha}, also showcasing a certain robustness.

\smallskip

Next, by the following theorem, we may replace the finite union theorem in Theorem \ref{bruha} by `more mathematical' principles.  
\begin{thm}[$\ACAo+\QFAC^{0,1}$]\label{frlo}
The higher items imply the lower items. 
\begin{itemize}
 \item The combination $\CUC_{\fin}+\cocode_{0}$.\label{hak}
\item For $f_{1}, \dots, f_{k}:[0,1]\di \R$ in $BV$, the sum $\sum_{i=1}^{k}f_{i}$ is in $BV$.
\item The finite union theorem. 
\end{itemize}
\end{thm}
\begin{proof}
For the first downward implication, the following set 
\[\textstyle
A_{n, i}:=\{x\in [0,1]:  |f_{i}(x+)-f_{i}(x)|>\frac{1}{2^{n}}\vee  |f_{i}(x-)-f_{i}(x)|>\frac{1}{2^{n}}    \} 
\]
is finite for all $n\in \N$ and $i\leq k$ if $f_{1}, \dots, f_{k}\in BV$.  
Clearly, the set $B_{n}:=\cup_{i\leq k}A_{n, i}$ is finite.  
Hence, there is an enumeration of $\cup_{n\in \N} B_{n}$, yielding a sequence $(x_{n})_{n\in \N}$ that lists all points of discontinuity of the functions $f_{i}$ for $i\leq k$.  
Using $(\mu^{2})$,  we can compute $V_{0}^{1}(f_{i})$ for $i\leq q$ as we can replace the usual supremum by one over $\N$ (and $\Q$).  The proof that $V_{0}^{1}(f+g)\leq V_{0}^{1}(f)+V_{0}^{1}(g)$ in \cite{voordedorst}*{p.\ 57} essentially amounts to the triangle inequality over $\R$, i.e.\ that $\sum_{i=1}^{k}f_{i}$ is in $BV$ now follows.  

\smallskip

The second downward implication is straightforward as a characteristic function $\mathbb{1}_{X}$ is in $BV$ if $X\subset [0,1]$ is finite by Theorem \ref{flima}.
\end{proof}
\noindent
We could replace the second item in Theorem \ref{frlo} by the following statement:
\begin{center}
\emph{for $f$ in $BV$ and $0=x_{0}<x_{1}<\dots<x_{k}<x_{k+1}=1$, $V_{0}^{1}(f)=\sum_{i=0}^{k}V_{x_{i}}^{x_{i+1}}(f)$}, 
\end{center}
but this would entail a number of technical details.  The same division property for the arc length of rectifiable functions would of course be rather natural.

\smallskip

Next, Jordan's original motivation for introducing $BV$-functions in \cite{jordel} was the convergence of Fourier series. Now, the latter always converges to $\frac{f(x+)+f(x-)}{2}$ for $f\in BV$.  
In this light, item \eqref{konk2} from Theorem is equivalent to the following.
\begin{center} 
The \emph{uniform finite union theorem} plus: for $f\in BV$, there is a sequence enumerating all points where the Fourier series does \emph{not} equal the function value.  
\end{center}
To derive the centred statement, it a somewhat tedious verification that $\ACAo+\QFAC^{0,1}$ can formalise the proof 
that the Fourier series of $f\in BV$ always converges to $\frac{f(x+)+f(x-)}{2}$.  
We refer to \cite{easypeasyfourieranalysie} for an elementary proof of this convergence result.
A more detailed discussion is in \cite{samRMR}*{\S3}, including various textbook proofs. 

\smallskip

Finally, we have used $\ACAo$ (plus extensions) as our base theory in the above; the following theorem implies that $(\exists^{2})$ can be expressed in terms 
of basic properties of regulated functions as well.  We use `usco' to abbreviate `upper semi-continuous'.
\begin{thm}[$\RCAo+\WKL$]
The following are equivalent to $(\exists^{2})$.
\begin{enumerate}
\renewcommand{\theenumi}{\roman{enumi}}
\item There exists Riemann integrable $f:[0,1]\di [0,1], g:[0,1]\di \R$ such that $g\circ f$ is not Riemann integrable.\label{F1}
\item  There exists a function that is not Riemann integrable. \label{F12}
\item There exists regulated $f:[0,1]\di [0,1], g:[0,1]\di \R$ such that $g\circ f$ is not regulated.\label{F2}
\item  There exists a function that is not regulated. \label{F3}
\item There exists $f:[0,1]\di [0,1], g:[0,1]\di \R$ in Baire 1 such that $g\circ f$ is not in Baire 1.\label{F4}
\item There exists a function $f:[0,1]\di \R$ that is not Baire 1.\label{F5}
\item There exists usco $f:[0,1]\di [0,1], g:[0,1]\di \R$ such that $g\circ f$ is not usco.\label{F6}
\item  There exists a function that is not everywhere usco. \label{F7}
\item There exists a function that is not everywhere quasi-continuous. \label{F10}
\item There exists a function that is not everywhere cliquish. \label{F11}
\end{enumerate}
We only need $\WKL$ for the first and second items.  
\end{thm}
\begin{proof}
First of all, assume $(\exists^{2})$ and define $f:[0,1]\di [0,1]$ as follows:
\be\label{thomae}
f(x):=
\begin{cases} 
0 & \textup{if } x\in \R\setminus\Q\\
\frac{1}{q} & \textup{if $x=\frac{p}{q}$ and $p, q$ are co-prime} 
\end{cases}.
\ee
Thomae introduces this function around 1875 in \cite{thomeke}*{p.\ 14, \S20}); one readily verifies that Thomae's function is Riemann integrable (with integral equal to zero) and regulated (with zero as left and right limits) on any interval.
Now define $g:[0,1]\di \R$ as $0$ in case $x=0$, and $1$ otherwise; this function is trivially Riemann integrable and regulated.  
However, $g\circ f$ is Dirichlet's function $\mathbb{1_{Q}}$, i.e.\ the characteristic function of the rationals, which is trivially shown to be \emph{not} Riemann integrable and \emph{not} regulated.
Thus, $(\exists^{2})$ implies items \eqref{F1}-\eqref{F3}. 

\smallskip

Secondly, assume item \eqref{F2} (similar for item \eqref{F3}) and note that $g\circ f$ must be discontinuous, as continuous functions are trivially regulated. 
However, the existence of a discontinuous function on $\R$ yields $(\exists^{2})$ by \cite{kohlenbach2}*{\S3}.
Similarly, for items \eqref{F1} and \eqref{F12}, $\WKL$ suffices to obtain an RM-code for a continuous functions on Cantor space (see \cite{kohlenbach4}*{\S4}); the same goes through \emph{mutatis mutandis} for functions on $[0,1]$.  
Hence, $\WKL$ suffices to show that a continuous function on $[0,1]$ is Riemann integrable by \cite{simpson2}*{IV.2.6}.
Thus, $g\circ f$ must be discontinuous, which yields $(\exists^{2})$ by \cite{kohlenbach2}*{\S3}.  Similarly, for items \eqref{F3} and \eqref{F4}, a function not in Baire 1 must be discontinuous, as continuous functions are trivially Baire 1; in this way, we obtain a discontinuous function and hence $(\exists^{2})$ by \cite{kohlenbach2}*{\S3}. The first five items now each imply $(\exists^{2})$, the first one using $\WKL$ as noted above.

\smallskip

Thirdly, assume $(\exists^{2})$ and note that Thomae's function is Baire 1.  In particular, finding a sequence of continuous function converging to $f$ as in \eqref{thomae} is straightforward (using $\exists^{2}$).
The same holds for define $g:[0,1]\di \R$ defined as $0$ in case $x=0$, and $1$ otherwise.  We now show that  $\mathbb{1_{Q}}=g\circ f$ is not Baire 1, establishing items \eqref{F4} and \eqref{F5}. 
To this end, suppose $(f_{n})_{n\in \N}$ is a sequence of continuous functions with pointwise limit $\mathbb{1_{Q}}$.  We first prove the following:
\begin{center}
For any non-empty $[a,b]\subset [0,1]$, there is an arbitrarily large $N$ and a non-empty $[c,d]\subset [a,b]$ such that  $f_{N}([c, d])=[\frac{1}{4}, \frac{3}{4}]$.
\end{center}
To establish this result, fix a non-empty interval $[a,b]\subset [0,1]$ and fix $x<y$ such that $x\in \Q\cap [a,b]$ and $y\in [a,b]\setminus \Q$.
Since $(f_{n})_{n\in \N}$ converges pointwise to $\mathbb{1_{Q}}$, there exists arbitrarily large $N$ such that $f_{N}(x)\geq \frac{3}{4}$ and $f_{N}(y)\leq\frac14$. 
By the intermediate value theorem (provable in $\RCA_{0}$ for RM-codes, and hence in $\ACAo$ for continuous functions), there exists an interval  $[c,d]\subseteq [x,y]\subset [a,b]$ such that $f_{N}([c,d])=[\frac14,\frac34]$.

\smallskip

By \cite{kohlenbach3}*{\S3}, $(\exists^{2})$ is equivalent to the existence of a functional witnessing the intermediate value theorem.  Hence, following the previous paragraph, $\exists^{2}$
readily yields a functional that returns the numbers $N\in \N$ and $c, d\in [0,1]$ as in the centred statement on input $[a,b]$ and $m\in \N$, where $N\geq m$.  
Using the latter functional, one readily obtains sequences $(c_{n})_{n\in \N}$, $(d_{n})_{n\in \N}$, and $g\in \N^{\N}$ such that $g(n)\geq n$, $f_{g(n)}([c_{n}, d_{n}])=[\frac14, \frac34]$, and $|c_{n}-d_{n}|<\frac{1}{2^{n}}$ for all $n\in \N$.
However, if $c=\lim_{n\di \infty }c_{n}$, then $\mathbb{1_{Q}}(c)=\lim_{n\di \infty }f_{g(n)}(c)\in [\frac14, \frac34]$, a contradiction.  Hence, we have proved item \eqref{F4} and \eqref{F5}.  
Since $\mathbb{1_{Q}}$ is not usco (or quasic-continuous or cliquish) by definition, the equivalence between $(\exists^{2})$ and items \eqref{F6}-\eqref{F11} follows in the same way. 
\end{proof}
The previous theorem yields the following strange result by contraposition: if all functions on $\R$ are Baire 1, then all functions on $\R$ are continuous.
In this light, Brouwer's theorem is not an isolated event, but rather the limit of a certain process.   One cannot push the previous equivalences much beyond Baire 1, as follows. 
\begin{thm}[$\ACAo+\IND_{0}$]
The principle $\NIN$ follows from the statement: \emph{there is a $[0,1]\di \R$ function that is not Baire 2}. 
\end{thm}
\begin{proof}
Fix $f:[0,1]\di \R$ and let $Y:[0,1]\di \N$ be injective.  Now define $f_{n}(x)$ as $f(x)$ in case $Y(x)\leq n$, and $0$ otherwise. 
Clearly, $f$ is the pointwise limit of the sequence $(f_{n})_{n\in \N}$.  Now fix some $n_{0}\in \N$ use $\IND_{0}$ to enumerate all $x\in [0,1]$ such that $Y(x)\leq n_{0}$.  
With this finite sequence, one readily defines a sequence of continuous functions converging to $f_{n_{0}}$, which shows that the latter is Baire 1.
\end{proof}
In conclusion, we note that the insights in this section (esp.\ regarding Definition~\ref{deadd}) came about after a recent FOM-discussion initiated by Friedman (\cite{fomo}).  

\subsubsection{On the choice of definitions}\label{crux}
In this section, we discuss our choice of definitions and provide some motivation.  

\smallskip

First of all, the following remark provides some motivation for the use of our definitions of finite and closed set as in Definitions \ref{openset} and \ref{deadd}. 
\begin{rem}\label{dichtbij}\rm
As discussed above, the sets $A_{n}$ from \eqref{lagel2} are finite and hence closed.  
In particular, working in $\ZF$ (or even $\Z_{2}^{\Omega}$ from Section \ref{lll}), the following objects can be constructed:
\begin{itemize}
\item for $n\in \N$, an injection $Y_{n}$ from $A_{n}$ to some $\{0, 1, \dots, k\}$ with $k\in \N$,
\item for $m\in \N$, an RM-code $C_{m}$ (see \cite{simpson2}*{II.5.6}) for the closed sets $A_{m}$.
\end{itemize}
However, it is shown in \cite{samwollic22post, samcsl23} that neither $Y_{n}$ nor $C_{n}$ are computable (in the sense of Kleene S1-S9) in terms of any $\SS_{m}^{2}$ and the other data.
Hence, it seems $\Z_{2}^{\omega}$ cannot prove the general existence of $Y_{n}$ and $C_{n}$ as in the previous items.
By contrast, the system $\ACAo$ (and even fragments) suffice to show that $A_{n}$ from \eqref{lagel2} is finite in the sense of Definition \ref{deadd}, and closed in the sense of Definition \ref{openset}.

\smallskip

In conclusion, the study of $BV$-functions readily yields finite (resp.\ closed) sets for which there is no reasonable injection to some fragment of $\N$ (resp.\ RM-code).  
This observation justifies our choice of definitions of closed and finite set as in Definitions \ref{openset} and \ref{deadd}
\end{rem}
Secondly, Remark \ref{dichtbij} has some ramifications for our choice of the definition of `countable set', as follows. 
Indeed, one could reformulate $\CUC_{\fin}+\cocode_{0}$ as:
\begin{center}
\emph{a \textbf{height countable} set in the unit interval can be enumerated}, 
\end{center}
where the boldface notion is defined as follows. 
\bdefi[Height countable]\label{ked}
A set $A\subset \R$ is \emph{height countable} if there is a \emph{height} $H:\R\di \N$ for $A$, i.e.\ for all $n\in \N$, $A_{n}:= \{ x\in A: H(x)<n\}$ is finite. 
\edefi
The notion of `height' is mentioned in e.g.\ \cite{vadsiger, royco, demol,komig,hux} in connection to countability.  
Now, as to the naturalness of Definition \ref{ked}, consider the set of discontinuities of a function $f\in BV$ (or even regular), definable in $\ACAo$:
\be\label{dink}
A:=\{x\in [0,1]:f(x+)\ne f(x-)\}.
\ee
The set $A$ is trivially height countable and central to many proofs in \cite{voordedorst}.  As discussed in \cite{samwollic22}, no $\SS_{m}^{2}$ suffices to compute an injection from $A$ to $\N$ in general.    

\smallskip

In conclusion, the textbook study of $BV$-functions yields height countable sets occuring `in the wild' but with no `reasonable' injection (or bijection) to $\N$.  
Hence, it seems we have a choice between using $\CUC_{\fin}$ or adopting Definition \ref{ked} as our definition of countable set.  
We choose the former option as e.g.\ Theorem \ref{bruha} is still quite elegant. 
By contrast, Definition \ref{ked} is used in \cites{samwollic22, samRMR}, as this seems to be the \textbf{only} way of obtaining elegant equivalences for the uncountability of $\R$.
To be absolutely clear, as documented in \cites{samwollic22, samRMR}, the statement \emph{the unit interval is not height countable} readily gives rise to many interesting equivalences while $\NIN$ does not (seem to), say working over $\ACAo+\QFAC^{0,1}$ or fragments.  

\smallskip

Thirdly, our notion of `finite set' as in Definition \ref{deadd} is different from the mainstream set theory definition (see Footnote~\ref{krukk}), for reasons discussed in Remark~\ref{diunk}.  
Nonetheless, the reader may desire an equivalence in Theorem \ref{bruha} involving a (more) mainstream definition of finite set.   To this end, let $\CUC_{\fin}^{'}$ be $\CUC_{\fin}$ formulated with the following finiteness notion.
\bdefi[Set theory finite]\label{stf}
A set $X\subset \R$ is \emph{set theory finite} if there are $k\in \N$ and $Y:[0,1]\di \N$ such that on $X$, $Y$ is bounded by $k$ and injective.
\edefi
\noindent
One readily shows that the following are equivalent, say over $\ACAo+\QFAC^{0,1}$.
\begin{itemize}
\item (Bolzano-Weierstrass) For $X\subset [0,1]$ which is not set theory finite, there is a limit point $y\in [0,1]$, i.e.\ $(\forall k\in \N)(\exists x\in X)(|x-y|<\frac{1}{2^{k}})$.
\item A finite set (in the sense of Definition \ref{deadd}) is set theory finite. 
\end{itemize}
Letting $\BW$ be the first item, we note that item \eqref{ha} from Theorem \ref{bruha} is equivalent to $\BW+\CUC_{\fin}'+\cocode_{0}$, and where the latter uses Definition \ref{stf} exclusively. 
Jordan mentions $\BW$ in e.g.\ \cite{jordel3}*{p.\ 23, \S27}.
We intend to explore the content of the previous remark in a future paper. 

\smallskip

Fourth, a regulated function has \emph{bounded Waterman variation} (\cite{voordedorst}*{Prop.\ 2.24}).  The latter notion amounts to replacing $|f(x_{i+1})-f(x_{i})|$ by $\lambda_{i}|f(x_{i+1})-f(x_{i})|$ in \eqref{tomb}, for a \emph{Waterman sequence}  $(\lambda_{k})_{k\in \N}$ as in \cite{voordedorst}*{Def.\ 2.15}.  Now, for $BV$-functions with variation bounded by $1$, \eqref{lagel} can have at most $2^{n}$ elements. Functions of bounded Waterman variation similarly have explicit upper bounds -defined in terms of $(\lambda_{k})_{k\in \N}$ and $\exists^{2}$- on the set \eqref{lagel}.  
In this way, the regulated function $g$ from \eqref{varivari} has bounded Waterman variation and this readily yields an upper bound function for $(X_{n})_{n\in \N}$ as in the uniform finite union theorem.  
Hence, we can avoid the use of the latter (and perhaps even $\QFAC^{0,1}$) if we have access to the information provided by the bounded Waterman variation of a regulated function. 


\subsection{Unordered sums}\label{unorder}
We develop the RM-study of \emph{unordered sums}, which are a device for bestowing meaning upon sums involving uncountable index sets.  
We first introduce the relevant definitions and prove the equivalence between $\cocode_{0}$ and basic properties of unordered sums in Theorem \ref{flunk2}.

\smallskip

%
First of all, \emph{unordered sums} are essentially `uncountable sums' $\sum_{x\in I}f(x)$ for \emph{any} index set $I$ and $f:I\di \R$.  
A central result is that if $\sum_{x\in I}f(x)$ somehow exists, it must be a `normal' series of the form $\sum_{i\in \N}f(y_{i})$, i.e.\ $f(x)=0$ for all but countably many $x\in [0,1]$; Tao mentions this theorem in \cite{taomes}*{p.~xii}. 

\smallskip

By way of motivation, there is considerable historical and conceptual interest in this topic: Kelley notes in \cite{ooskelly}*{p.\ 64} that E.H.\ Moore's study of unordered sums in \cite{moorelimit2} led to the concept of \emph{net} with his student H.L.\ Smith (\cite{moorsmidje}).
Unordered sums can be found in (self-proclaimed) basic or applied textbooks (\cites{hunterapp,sohrab}) and can be used to develop measure theory (\cite{ooskelly}*{p.\ 79}).  
Moreover, Tukey shows in \cite{tukey1} that topology can be developed using \emph{phalanxes}, which are nets with the same index sets as unordered sums.  

\smallskip

Now, unordered sums are just a special kind of \emph{net} and $a:[0,1]\di \R$ is therefore written $(a_{x})_{x\in [0,1]} $ in this context to suggest the connection to nets.  
The associated notation $\sum_{x\in [0,1]}a_{x}$ is purely symbolic.   
We only need the following notions in the below. 
Let $\fin(\R)$ be the set of all finite sequences of reals without repetitions.  
\bdefi\label{kaukie} Let $a:[0,1]\di \R$ be any mapping, also denoted $(a_{x})_{x\in [0,1]}$.
\begin{itemize}
\item We say that $\sum_{x\in [0,1]}a_{x} $ is \emph{bounded} if there is $N_{0}\in \N$ such that for any $J\in \fin(\R)$, $N_{0}>|\sum_{x\in J}a_{x}|$. 
\item We say that $(a_{x})_{x\in [0,1]} $ is \emph{convergent to $a\in \R$} if for all $k\in \N$, there is $I\in \fin({\R})$ such that for $J \in \fin({\R})$ with $I\subseteq J$, we have $|a-\sum_{x\in J}a_{x}|<\frac{1}{2^{k}}$.
\end{itemize}
\edefi
Note that in the first item, $\Phi$ is called a \emph{Cauchy modulus}.  
For simplicity, we focus on \emph{positive unordered sums}, i.e.\ $(a_{x})_{x\in [0,1]}$ such that $a_{x}\geq 0$ for $x\in [0,1]$.

\smallskip

Secondly, we establish equivalences basic properties of unordered sums. 
We note that $\QFAC^{0,1}$ is no longer needed in the base theory, while $\cocode_{0}$ is equivalent to the first item in Theorem \ref{flunk2} given $\CUC_{\fin}$.  
\begin{thm}[$\ACAo$]\label{flunk2}
The following are equivalent.
\begin{enumerate}
\renewcommand{\theenumi}{\roman{enumi}}
\item Let $(X_{n})_{n\in \N}$ and $g\in \N^{\N}$ be such that $g(n)$ is an upper bound on the size of $X_{n}$, for all $n\in \N$.  Then $\cup_{n\in \N}X_{n}$ can be enumerated. \label{bb3}
\item For a positive and bounded unordered sum $\sum_{x\in [0,1]}a_{x}$, there is a sequence $(y_{n})_{n\in \N}$ of reals such that $a_{y}=0$ for all $y$ not in this sequence.\label{bb7}
\end{enumerate}
Assuming $\QFAC^{0,1}$, the above are equivalent to:
\begin{enumerate}
\setcounter{enumi}{+2}
\renewcommand{\theenumi}{\roman{enumi}} 
\item A positive bounded unordered sum $\sum_{x\in [0,1]}a_{x}$ is convergent to some $a\in \R$.\label{bb8}
\end{enumerate}
\end{thm}
\begin{proof}
%
The equivalence between items \eqref{bb3} and \eqref{bb7} is as follows: assume the latter and let $(X_{n})_{n\in \N}$ and $g:\N\di \N$ be as in item \eqref{bb3}. 
Define $(a_{x})_{x\in [0,1]}$ as follows:
\[
a_{x}:=  
\begin{cases}
0  & x\not \in \cup_{n\in \N} X_{n} \\
\frac{1}{2^{n}}\frac{1}{g(n)+1} & x\in X_{n} \textup{ and $n$ is the least such natural} 
\end{cases}.
\]
Clearly, this unordered sum has upper bound $1$.  If $(y_{n})_{n\in \N}$ is as in item \eqref{bb7}, we obtain an enumeration of $ \cup_{n\in\N}X_{n}$.  
Now assume item \eqref{bb3} and let $(a_{x})_{x\in [0,1]}$ be an unordered sum that is Cauchy, and consider the following set:
\be\label{lort}
X_{n}:=\{x\in [0,1]: a_{x}>1/2^{n}\}.
\ee
Let $N_{0}\in \N$ be an upper bound for $\sum_{x\in K}a_{x}$ for any $K\in \fin{(\R)}$.
Hence, the finite set $X_{n}$ in \eqref{lort} has size at most $2^{n}N_{0}$.  
An enumeration of $\cup_{n\in \N}X_{n}$ immediately yields the sequence as in item~\eqref{bb7}. 

\smallskip

The implication \eqref{bb7} $\di $ \eqref{bb8} is straightforward: the former guarantees that an unordered sum is a `normal' series, which must 
converge by the monotone convergence theorem (provable in $\ACA_{0}$ by \cite{simpson2}*{III.2}).  Now assume item \eqref{bb8} and note
that convergence of an unordered sum to some $a\in \R$ implies
\be\textstyle\label{furlo}
(\forall k\in \N\)(\exists I\in \fin(\R)) \big(|a-\sum_{x\in I}a_{x}|<\frac{1}{2^{k}}\big).
\ee
Apply $\QFAC^{0,1}$ to \eqref{furlo} to obtain a sequence $(I_{n})_{n\in \N}$ of finite sequences of reals.
This sequence must contain all $y\in \R$ such that $a_{y}\ne0$, otherwise \eqref{furlo} would be false.  
Use Feferman's $\mu^{2}$ to remove all other reals, and we are done.   
\end{proof}
The following result is perhaps more surprising.  Note that the second item also follows from item \eqref{bb3} in Theorem \ref{flunk2}.
\begin{cor}[$\ACAo+\IND_{1}$]
The higher items imply the lower ones.
\begin{itemize}
\item $\QFAC^{0,1}$.
\item For a positive and \textbf{convergent} unordered sum $\sum_{x\in [0,1]}a_{x}$, there is a sequence $(y_{n})_{n\in \N}$ of reals such that $a_{y}=0$ for all $y$ not in this sequence.\label{bv7}
\item $\cocode_{1}$.
\end{itemize}
\end{cor}
\begin{proof}
The first downward implication is proved as in the proof of the theorem.  
For the second downward application, let $A\subset [0,1]$ and $Y:[0,1]\di \R$ be such that the latter is bijective on the former. 
Define $a_{x}:=\frac{1}{2^{Y(x)+1}}$ if $x\in A$, and $0$ otherwise.  One readily proves that $\sum_{x\in [0,1]}a_{x}$ is convergent to $1$, for which $\IND_{1}$ is needed.
The sequence from the second item now yields the enumeration of the set $A$ required by $\cocode_{1}$. 
\end{proof}
One can derive a version of $\QFAC^{0,1}$ involving an `at most finitely many' condition on the existential quantifier in the antecedent.  

\begin{ack}\rm
We thank Anil Nerode for his valuable advice.
We also thank the anonymous referee for the many detailed and helpful suggestions.    
Our research was supported by the \emph{Deutsche Forschungsgemeinschaft} via the DFG grant SA3418/1-1.
Initial results were obtained during the stimulating MFO workshop (ID 2046) on proof theory and constructive mathematics in Oberwolfach in early Nov.\ 2020.  
We express our gratitude towards the aforementioned institutions.    
\end{ack}

\appendix

\section{Reverse Mathematics: introduction and definitions}\label{RMA}
\subsection{Reverse Mathematics}\label{prelim1}
We discuss Reverse Mathematics (Section \ref{introrm}) and introduce -in full detail- Kohlenbach's base theory of \emph{higher-order} Reverse Mathematics (Section \ref{rmbt}).
Some essential axioms, functionals, and notations may be found in Sections \ref{kkk} and \ref{lll}.
\subsubsection{Introduction}\label{introrm}
Reverse Mathematics (RM hereafter) is a program in the foundations of mathematics initiated around 1975 by Friedman (\cites{fried,fried2}) and developed extensively by Simpson (\cite{simpson2}).  
The aim of RM is to identify the minimal axioms needed to prove theorems of ordinary, i.e.\ non-set theoretical, mathematics. 

\smallskip

We refer to \cite{stillebron} for a basic introduction to RM and to \cite{simpson2, simpson1} for an overview of RM.  We expect basic familiarity with RM, but do sketch some aspects of Kohlenbach's \emph{higher-order} RM (\cite{kohlenbach2}) essential to this paper, including the base theory $\RCAo$ (Definition \ref{kase}).  

\smallskip

First of all, in contrast to `classical' RM based on \emph{second-order arithmetic} $\Z_{2}$, higher-order RM uses $\L_{\omega}$, the richer language of \emph{higher-order arithmetic}.  
Indeed, while the former is restricted to natural numbers and sets of natural numbers, higher-order arithmetic can accommodate sets of sets of natural numbers, sets of sets of sets of natural numbers, et cetera.  
To formalise this idea, we introduce the collection of \emph{all finite types} $\mathbf{T}$, defined by the two clauses:
\begin{center}
(i) $0\in \mathbf{T}$   and   (ii)  If $\sigma, \tau\in \mathbf{T}$ then $( \sigma \di \tau) \in \mathbf{T}$,
\end{center}
where $0$ is the type of natural numbers, and $\sigma\di \tau$ is the type of mappings from objects of type $\sigma$ to objects of type $\tau$.
In this way, $1\equiv 0\di 0$ is the type of functions from numbers to numbers, and  $n+1\equiv n\di 0$.  Viewing sets as given by characteristic functions, we note that $\Z_{2}$ only includes objects of type $0$ and $1$.    

\smallskip

Secondly, the language $\L_{\omega}$ includes variables $x^{\rho}, y^{\rho}, z^{\rho},\dots$ of any finite type $\rho\in \mathbf{T}$.  Types may be omitted when they can be inferred from context.  
The constants of $\L_{\omega}$ include the type $0$ objects $0, 1$ and $ <_{0}, +_{0}, \times_{0},=_{0}$  which are intended to have their usual meaning as operations on $\N$.
Equality at higher types is defined in terms of `$=_{0}$' as follows: for any objects $x^{\tau}, y^{\tau}$, we have
\be\label{aparth}
[x=_{\tau}y] \equiv (\forall z_{1}^{\tau_{1}}\dots z_{k}^{\tau_{k}})[xz_{1}\dots z_{k}=_{0}yz_{1}\dots z_{k}],
\ee
if the type $\tau$ is composed as $\tau\equiv(\tau_{1}\di \dots\di \tau_{k}\di 0)$.  
Furthermore, $\L_{\omega}$ also includes the \emph{recursor constant} $\mathbf{R}_{\sigma}$ for any $\sigma\in \mathbf{T}$, which allows for iteration on type $\sigma$-objects as in the special case \eqref{special}.  Formulas and terms are defined as usual.  
One obtains the sub-language $\L_{n+2}$ by restricting the above type formation rule to produce only type $n+1$ objects (and related types of similar complexity).        

\subsubsection{The base theory of higher-order Reverse Mathematics}\label{rmbt}
We introduce Kohlenbach's base theory $\RCAo$, first introduced in \cite{kohlenbach2}*{\S2}.
\bdefi\label{kase} 
The base theory $\RCAo$ consists of the following axioms.
\begin{enumerate}
 \renewcommand{\theenumi}{\alph{enumi}}
\item  Basic axioms expressing that $0, 1, <_{0}, +_{0}, \times_{0}$ form an ordered semi-ring with equality $=_{0}$.
\item Basic axioms defining the well-known $\Pi$ and $\Sigma$ combinators (aka $K$ and $S$ in \cite{avi2}), which allow for the definition of \emph{$\lambda$-abstraction}. 
\item The defining axiom of the recursor constant $\mathbf{R}_{0}$: for $m^{0}$ and $f^{1}$: 
\be\label{special}
\mathbf{R}_{0}(f, m, 0):= m \textup{ and } \mathbf{R}_{0}(f, m, n+1):= f(n, \mathbf{R}_{0}(f, m, n)).
\ee
\item The \emph{axiom of extensionality}: for all $\rho, \tau\in \mathbf{T}$, we have:
\be\label{EXT}\tag{$\textsf{\textup{E}}_{\rho, \tau}$}  
(\forall  x^{\rho},y^{\rho}, \varphi^{\rho\di \tau}) \big[x=_{\rho} y \di \varphi(x)=_{\tau}\varphi(y)   \big].
\ee 
\item The induction axiom for quantifier-free formulas of $\L_{\omega}$.
\item $\QFAC^{1,0}$: the quantifier-free Axiom of Choice as in Definition \ref{QFAC}.
\end{enumerate}
\edefi
\noindent
Note that variables (of any finite type) are allowed in quantifier-free formulas of the language $\L_{\omega}$: only quantifiers are banned.
Recursion as in \eqref{special} is called \emph{primitive recursion}; the class of functionals obtained from $\mathbf{R}_{\rho}$ for all $\rho \in \mathbf{T}$ is called \emph{G\"odel's system $T$} of all (higher-order) primitive recursive functionals. 
\bdefi\label{QFAC} The axiom $\QFAC$ consists of the following for all $\sigma, \tau \in \textbf{T}$:
\be\tag{$\QFAC^{\sigma,\tau}$}
(\forall x^{\sigma})(\exists y^{\tau})A(x, y)\di (\exists Y^{\sigma\di \tau})(\forall x^{\sigma})A(x, Y(x)),
\ee
for any quantifier-free formula $A$ in the language of $\L_{\omega}$.
\edefi
As discussed in \cite{kohlenbach2}*{\S2}, $\RCAo$ and $\RCA_{0}$ prove the same sentences `up to language' as the latter is set-based and the former function-based.   
This conservation result is obtained via the so-called $\ECF$-interpretation discussed in Remark \ref{ECF}. 

%

\subsubsection{Notations and the like}\label{kkk}
We introduce the usual notations for common mathematical notions, like real numbers, as also introduced in \cite{kohlenbach2}.  
\begin{defi}[Real numbers and related notions in $\RCAo$]\label{keepintireal}\rm~
\begin{enumerate}
 \renewcommand{\theenumi}{\alph{enumi}}
\item Natural numbers correspond to type zero objects, and we use `$n^{0}$' and `$n\in \N$' interchangeably.  Rational numbers are defined as signed quotients of natural numbers, and `$q\in \Q$' and `$<_{\Q}$' have their usual meaning.    
\item Real numbers are coded by fast-converging Cauchy sequences $q_{(\cdot)}:\N\di \Q$, i.e.\  such that $(\forall n^{0}, i^{0})(|q_{n}-q_{n+i}|<_{\Q} \frac{1}{2^{n}})$.  
We use Kohlenbach's `hat function' from \cite{kohlenbach2}*{p.\ 289} to guarantee that every $q^{1}$ defines a real number.  
\item We write `$x\in \R$' to express that $x^{1}:=(q^{1}_{(\cdot)})$ represents a real as in the previous item and write $[x](k):=q_{k}$ for the $k$-th approximation of $x$.    
\item Two reals $x, y$ represented by $q_{(\cdot)}$ and $r_{(\cdot)}$ are \emph{equal}, denoted $x=_{\R}y$, if $(\forall n^{0})(|q_{n}-r_{n}|\leq {2^{-n+1}})$. Inequality `$<_{\R}$' is defined similarly.  
We sometimes omit the subscript `$\R$' if it is clear from context.           
\item Functions $F:\R\di \R$ are represented by $\Phi^{1\di 1}$ mapping equal reals to equal reals, i.e.\ extensionality as in $(\forall x , y\in \R)(x=_{\R}y\di \Phi(x)=_{\R}\Phi(y))$.\label{EXTEN}
\item The relation `$x\leq_{\tau}y$' is defined as in \eqref{aparth} but with `$\leq_{0}$' instead of `$=_{0}$'.  Binary sequences are denoted `$f^{1}, g^{1}\leq_{1}1$', but also `$f,g\in C$' or `$f, g\in 2^{\N}$'.  Elements of Baire space are given by $f^{1}, g^{1}$, but also denoted `$f, g\in \N^{\N}$'.
\item For a binary sequence $f^{1}$, the associated real in $[0,1]$ is $\r(f):=\sum_{n=0}^{\infty}\frac{f(n)}{2^{n+1}}$.\label{detrippe}
\item Sets of type $\rho$ objects $X^{\rho\di 0}, Y^{\rho\di 0}, \dots$ are given by their characteristic functions $F^{\rho\di 0}_{X}\leq_{\rho\di 0}1$, i.e.\ we write `$x\in X$' for $ F_{X}(x)=_{0}1$. \label{koer} 
\end{enumerate}
\end{defi}
For completeness, we list the following notational convention for finite sequences.  
\begin{nota}[Finite sequences]\label{skim}\rm
The type for `finite sequences of objects of type $\rho$' is denoted $\rho^{*}$, which we shall only use for $\rho=0,1$.  
Since the usual coding of pairs of numbers goes through in $\RCAo$, we shall not always distinguish between $0$ and $0^{*}$. 
Similarly, we assume a fixed coding for finite sequences of type $1$ and shall make use of the type `$1^{*}$'.  
In general, we do not always distinguish between `$s^{\rho}$' and `$\langle s^{\rho}\rangle$', where the former is `the object $s$ of type $\rho$', and the latter is `the sequence of type $\rho^{*}$ with only element $s^{\rho}$'.  The empty sequence for the type $\rho^{*}$ is denoted by `$\langle \rangle_{\rho}$', usually with the typing omitted.  

\smallskip

Furthermore, we denote by `$|s|=n$' the length of the finite sequence $s^{\rho^{*}}=\langle s_{0}^{\rho},s_{1}^{\rho},\dots,s_{n-1}^{\rho}\rangle$, where $|\langle\rangle|=0$, i.e.\ the empty sequence has length zero.
For sequences $s^{\rho^{*}}, t^{\rho^{*}}$, we denote by `$s*t$' the concatenation of $s$ and $t$, i.e.\ $(s*t)(i)=s(i)$ for $i<|s|$ and $(s*t)(j)=t(|s|-j)$ for $|s|\leq j< |s|+|t|$. For a sequence $s^{\rho^{*}}$, we define $\overline{s}N:=\langle s(0), s(1), \dots,  s(N-1)\rangle $ for $N^{0}<|s|$.  
For a sequence $\alpha^{0\di \rho}$, we also write $\overline{\alpha}N=\langle \alpha(0), \alpha(1),\dots, \alpha(N-1)\rangle$ for \emph{any} $N^{0}$.  By way of shorthand, 
$(\forall q^{\rho}\in Q^{\rho^{*}})A(q)$ abbreviates $(\forall i^{0}<|Q|)A(Q(i))$, which is (equivalent to) quantifier-free if $A$ is.   
\end{nota}
\subsubsection{Some comprehension functionals}\label{lll}
As noted in Section \ref{intro}, the logical hardness of a theorem is measured via what fragment of the comprehension axiom is needed for a proof.  
For this reason, we introduce some axioms and functionals related to \emph{higher-order comprehension} in this section.
We are mostly dealing with \emph{conventional} comprehension here, i.e.\ only parameters over $\N$ and $\N^{\N}$ are allowed in formula classes like $\Pi_{k}^{1}$ and $\Sigma_{k}^{1}$.

\smallskip

First of all, the following functional is clearly discontinuous at $f=11\dots$; in fact, $(\exists^{2})$ is equivalent to the existence of $F:\R\di\R$ such that $F(x)=1$ if $x>_{\R}0$, and $0$ otherwise (\cite{kohlenbach2}*{\S3}).  This fact shall be repeated often.  
\be\label{muk}\tag{$\exists^{2}$}
(\exists \varphi^{2}\leq_{2}1)(\forall f^{1})\big[(\exists n)(f(n)=0) \asa \varphi(f)=0    \big]. 
\ee
Related to $(\exists^{2})$, the functional $\mu^{2}$ in $(\mu^{2})$ is also called \emph{Feferman's $\mu$} (\cite{avi2}).
\begin{align}\label{mu}\tag{$\mu^{2}$}
(\exists \mu^{2})(\forall f^{1})\big[ (\exists n)(f(n)=0) \di [f(\mu(f))=0&\wedge (\forall i<\mu(f))(f(i)\ne 0) ]\\
& \wedge [ (\forall n)(f(n)\ne0)\di   \mu(f)=0]    \big], \notag
\end{align}
We have $(\exists^{2})\asa (\mu^{2})$ over $\RCAo$ and $\ACAo\equiv\RCAo+(\exists^{2})$ proves the same sentences as $\ACA_{0}$ by \cite{hunterphd}*{Theorem~2.5}. 

\smallskip

Secondly, the functional $\SS^{2}$ in $(\SS^{2})$ is called \emph{the Suslin functional} (\cite{kohlenbach2}).
\be\tag{$\SS^{2}$}
(\exists\SS^{2}\leq_{2}1)(\forall f^{1})\big[  (\exists g^{1})(\forall n^{0})(f(\overline{g}n)=0)\asa \SS(f)=0  \big], 
\ee
The system $\FIVE^{\omega}\equiv \RCAo+(\SS^{2})$ proves the same $\Pi_{3}^{1}$-sentences as $\FIVE$ by \cite{yamayamaharehare}*{Theorem 2.2}.   
By definition, the Suslin functional $\SS^{2}$ can decide whether a $\Sigma_{1}^{1}$-formula as in the left-hand side of $(\SS^{2})$ is true or false.   We similarly define the functional $\SS_{k}^{2}$ which decides the truth or falsity of $\Sigma_{k}^{1}$-formulas from $\L_{2}$; we also define 
the system $\SIXK$ as $\RCAo+(\SS_{k}^{2})$, where  $(\SS_{k}^{2})$ expresses that $\SS_{k}^{2}$ exists.  
We note that the operators $\nu_{n}$ from \cite{boekskeopendoen}*{p.\ 129} are essentially $\SS_{n}^{2}$ strengthened to return a witness (if existant) to the $\Sigma_{n}^{1}$-formula at hand.  

\smallskip

\noindent
Thirdly, full second-order arithmetic $\Z_{2}$ is readily derived from $\cup_{k}\SIXK$, or from:
\be\tag{$\exists^{3}$}
(\exists E^{3}\leq_{3}1)(\forall Y^{2})\big[  (\exists f^{1})(Y(f)=0)\asa E(Y)=0  \big], 
\ee
and we therefore define $\Z_{2}^{\Omega}\equiv \RCAo+(\exists^{3})$ and $\Z_{2}^\omega\equiv \cup_{k}\SIXK$, which are conservative over $\Z_{2}$ by \cite{hunterphd}*{Cor.\ 2.6}. 
Despite this close connection, $\Z_{2}^{\omega}$ and $\Z_{2}^{\Omega}$ can behave quite differently, as discussed in e.g.\ \cite{dagsamIII}*{\S2.2}.   
The functional from $(\exists^{3})$ is also called `$\exists^{3}$', and we use the same convention for other functionals.

\bye
\appendix

\section{Background}\label{prelimA}
We introduce \emph{Reverse Mathematics} in Section \ref{prelim1}, as well as Kohlenbach's generalisation to \emph{higher-order arithmetic}, and the associated base theory $\RCAo$.  
We introduce higher-order \emph{computability theory}, following Kleene's computation schemes S1-S9, in Section \ref{HCT}. 

\subsection{Reverse Mathematics}\label{prelim1}
We provide an introduction to Reverse Mathematics (Section \ref{introrm}) and introduce Kohlenbach's base theory of \emph{higher-order} Reverse Mathematics (Section \ref{rmbt}).
Some essential axioms, functionals, and notations may be found in Sections \ref{kkk} and \ref{lll}.
\subsubsection{Introduction}\label{introrm}
Reverse Mathematics (RM hereafter) is a program in the foundations of mathematics initiated around 1975 by Friedman (\cites{fried,fried2}) and developed extensively by Simpson (\cite{simpson2}).  
The aim of RM is to identify the minimal axioms needed to prove theorems of ordinary, i.e.\ non-set theoretical, mathematics. 

\smallskip

We refer to \cite{stillebron} for a basic introduction to RM and to \cite{simpson2, simpson1} for an overview of RM.  We expect basic familiarity with RM, but do sketch some aspects of Kohlenbach's \emph{higher-order} RM (\cite{kohlenbach2}) essential to this paper, including the base theory $\RCAo$ (Definition \ref{kase}).  

\smallskip

First of all, in contrast to `classical' RM based on \emph{second-order arithmetic} $\Z_{2}$, higher-order RM uses $\L_{\omega}$, the richer language of \emph{higher-order arithmetic}.  
Indeed, while the former is restricted to natural numbers and sets of natural numbers, higher-order arithmetic can accommodate sets of sets of natural numbers, sets of sets of sets of natural numbers, et cetera.  
To formalise this idea, we introduce the collection of \emph{all finite types} $\mathbf{T}$, defined by the two clauses:
\begin{center}
(i) $0\in \mathbf{T}$   and   (ii)  If $\sigma, \tau\in \mathbf{T}$ then $( \sigma \di \tau) \in \mathbf{T}$,
\end{center}
where $0$ is the type of natural numbers, and $\sigma\di \tau$ is the type of mappings from objects of type $\sigma$ to objects of type $\tau$.
In this way, $1\equiv 0\di 0$ is the type of functions from numbers to numbers, and  $n+1\equiv n\di 0$.  Viewing sets as given by characteristic functions, we note that $\Z_{2}$ only includes objects of type $0$ and $1$.    

\smallskip

Secondly, the language $\L_{\omega}$ includes variables $x^{\rho}, y^{\rho}, z^{\rho},\dots$ of any finite type $\rho\in \mathbf{T}$.  Types may be omitted when they can be inferred from context.  
The constants of $\L_{\omega}$ include the type $0$ objects $0, 1$ and $ <_{0}, +_{0}, \times_{0},=_{0}$  which are intended to have their usual meaning as operations on $\N$.
Equality at higher types is defined in terms of `$=_{0}$' as follows: for any objects $x^{\tau}, y^{\tau}$, we have
\be\label{aparth}
[x=_{\tau}y] \equiv (\forall z_{1}^{\tau_{1}}\dots z_{k}^{\tau_{k}})[xz_{1}\dots z_{k}=_{0}yz_{1}\dots z_{k}],
\ee
if the type $\tau$ is composed as $\tau\equiv(\tau_{1}\di \dots\di \tau_{k}\di 0)$.  
Furthermore, $\L_{\omega}$ also includes the \emph{recursor constant} $\mathbf{R}_{\sigma}$ for any $\sigma\in \mathbf{T}$, which allows for iteration on type $\sigma$-objects as in the special case \eqref{special}.  Formulas and terms are defined as usual.  
One obtains the sub-language $\L_{n+2}$ by restricting the above type formation rule to produce only type $n+1$ objects (and related types of similar complexity).        

\subsubsection{The base theory of higher-order Reverse Mathematics}\label{rmbt}
We introduce Kohlenbach's base theory $\RCAo$, first introduced in \cite{kohlenbach2}*{\S2}.
\bdefi\label{kase} 
The base theory $\RCAo$ consists of the following axioms.
\begin{enumerate}
 \renewcommand{\theenumi}{\alph{enumi}}
\item  Basic axioms expressing that $0, 1, <_{0}, +_{0}, \times_{0}$ form an ordered semi-ring with equality $=_{0}$.
\item Basic axioms defining the well-known $\Pi$ and $\Sigma$ combinators (aka $K$ and $S$ in \cite{avi2}), which allow for the definition of \emph{$\lambda$-abstraction}. 
\item The defining axiom of the recursor constant $\mathbf{R}_{0}$: for $m^{0}$ and $f^{1}$: 
\be\label{special}
\mathbf{R}_{0}(f, m, 0):= m \textup{ and } \mathbf{R}_{0}(f, m, n+1):= f(n, \mathbf{R}_{0}(f, m, n)).
\ee
\item The \emph{axiom of extensionality}: for all $\rho, \tau\in \mathbf{T}$, we have:
\be\label{EXT}\tag{$\textsf{\textup{E}}_{\rho, \tau}$}  
(\forall  x^{\rho},y^{\rho}, \varphi^{\rho\di \tau}) \big[x=_{\rho} y \di \varphi(x)=_{\tau}\varphi(y)   \big].
\ee 
\item The induction axiom for quantifier-free formulas of $\L_{\omega}$.
\item $\QFAC^{1,0}$: the quantifier-free Axiom of Choice as in Definition \ref{QFAC}.
\end{enumerate}
\edefi
\noindent
Note that variables (of any finite type) are allowed in quantifier-free formulas of the language $\L_{\omega}$: only quantifiers are banned.
Recursion as in \eqref{special} is called \emph{primitive recursion}; the class of functionals obtained from $\mathbf{R}_{\rho}$ for all $\rho \in \mathbf{T}$ is called \emph{G\"odel's system $T$} of all (higher-order) primitive recursive functionals. 
\bdefi\label{QFAC} The axiom $\QFAC$ consists of the following for all $\sigma, \tau \in \textbf{T}$:
\be\tag{$\QFAC^{\sigma,\tau}$}
(\forall x^{\sigma})(\exists y^{\tau})A(x, y)\di (\exists Y^{\sigma\di \tau})(\forall x^{\sigma})A(x, Y(x)),
\ee
for any quantifier-free formula $A$ in the language of $\L_{\omega}$.
\edefi
As discussed in \cite{kohlenbach2}*{\S2}, $\RCAo$ and $\RCA_{0}$ prove the same sentences `up to language' as the latter is set-based and the former function-based.   
This conservation results is obtained via the so-called $\ECF$-interpretation, which we now discuss. 
\begin{rem}[The $\ECF$-interpretation]\label{ECF}\rm
The (rather) technical definition of $\ECF$ may be found in \cite{troelstra1}*{p.\ 138, \S2.6}.
Intuitively, the $\ECF$-interpretation $[A]_{\ECF}$ of a formula $A\in \L_{\omega}$ is just $A$ with all variables 
of type two and higher replaced by type one variables ranging over so-called `associates' or `RM-codes' (see \cite{kohlenbach4}*{\S4}); the latter are (countable) representations of continuous functionals.  
The $\ECF$-interpretation connects $\RCAo$ and $\RCA_{0}$ (see \cite{kohlenbach2}*{Prop.\ 3.1}) in that if $\RCAo$ proves $A$, then $\RCA_{0}$ proves $[A]_{\ECF}$, again `up to language', as $\RCA_{0}$ is 
formulated using sets, and $[A]_{\ECF}$ is formulated using types, i.e.\ using type zero and one objects.  
\end{rem}
In light of the widespread use of codes in RM and the common practise of identifying codes with the objects being coded, it is no exaggeration to refer to $\ECF$ as the \emph{canonical} embedding of higher-order into second-order arithmetic. 

%

\subsubsection{Notations and the like}\label{kkk}
We introduce the usual notations for common mathematical notions, like real numbers, as also introduced in \cite{kohlenbach2}.  
\begin{defi}[Real numbers and related notions in $\RCAo$]\label{keepintireal}\rm~
\begin{enumerate}
 \renewcommand{\theenumi}{\alph{enumi}}
\item Natural numbers correspond to type zero objects, and we use `$n^{0}$' and `$n\in \N$' interchangeably.  Rational numbers are defined as signed quotients of natural numbers, and `$q\in \Q$' and `$<_{\Q}$' have their usual meaning.    
\item Real numbers are coded by fast-converging Cauchy sequences $q_{(\cdot)}:\N\di \Q$, i.e.\  such that $(\forall n^{0}, i^{0})(|q_{n}-q_{n+i}|<_{\Q} \frac{1}{2^{n}})$.  
We use Kohlenbach's `hat function' from \cite{kohlenbach2}*{p.\ 289} to guarantee that every $q^{1}$ defines a real number.  
\item We write `$x\in \R$' to express that $x^{1}:=(q^{1}_{(\cdot)})$ represents a real as in the previous item and write $[x](k):=q_{k}$ for the $k$-th approximation of $x$.    
\item Two reals $x, y$ represented by $q_{(\cdot)}$ and $r_{(\cdot)}$ are \emph{equal}, denoted $x=_{\R}y$, if $(\forall n^{0})(|q_{n}-r_{n}|\leq {2^{-n+1}})$. Inequality `$<_{\R}$' is defined similarly.  
We sometimes omit the subscript `$\R$' if it is clear from context.           
\item Functions $F:\R\di \R$ are represented by $\Phi^{1\di 1}$ mapping equal reals to equal reals, i.e.\ extensionality as in $(\forall x , y\in \R)(x=_{\R}y\di \Phi(x)=_{\R}\Phi(y))$.\label{EXTEN}
\item The relation `$x\leq_{\tau}y$' is defined as in \eqref{aparth} but with `$\leq_{0}$' instead of `$=_{0}$'.  Binary sequences are denoted `$f^{1}, g^{1}\leq_{1}1$', but also `$f,g\in C$' or `$f, g\in 2^{\N}$'.  Elements of Baire space are given by $f^{1}, g^{1}$, but also denoted `$f, g\in \N^{\N}$'.
\item For a binary sequence $f^{1}$, the associated real in $[0,1]$ is $\r(f):=\sum_{n=0}^{\infty}\frac{f(n)}{2^{n+1}}$.\label{detrippe}
\item Sets of type $\rho$ objects $X^{\rho\di 0}, Y^{\rho\di 0}, \dots$ are given by their characteristic functions $F^{\rho\di 0}_{X}\leq_{\rho\di 0}1$, i.e.\ we write `$x\in X$' for $ F_{X}(x)=_{0}1$. \label{koer} 
\end{enumerate}
\end{defi}
For completeness, we list the following notational convention for finite sequences.  
\begin{nota}[Finite sequences]\label{skim}\rm
The type for `finite sequences of objects of type $\rho$' is denoted $\rho^{*}$, which we shall only use for $\rho=0,1$.  
Since the usual coding of pairs of numbers goes through in $\RCAo$, we shall not always distinguish between $0$ and $0^{*}$. 
Similarly, we assume a fixed coding for finite sequences of type $1$ and shall make use of the type `$1^{*}$'.  
In general, we do not always distinguish between `$s^{\rho}$' and `$\langle s^{\rho}\rangle$', where the former is `the object $s$ of type $\rho$', and the latter is `the sequence of type $\rho^{*}$ with only element $s^{\rho}$'.  The empty sequence for the type $\rho^{*}$ is denoted by `$\langle \rangle_{\rho}$', usually with the typing omitted.  

\smallskip

Furthermore, we denote by `$|s|=n$' the length of the finite sequence $s^{\rho^{*}}=\langle s_{0}^{\rho},s_{1}^{\rho},\dots,s_{n-1}^{\rho}\rangle$, where $|\langle\rangle|=0$, i.e.\ the empty sequence has length zero.
For sequences $s^{\rho^{*}}, t^{\rho^{*}}$, we denote by `$s*t$' the concatenation of $s$ and $t$, i.e.\ $(s*t)(i)=s(i)$ for $i<|s|$ and $(s*t)(j)=t(|s|-j)$ for $|s|\leq j< |s|+|t|$. For a sequence $s^{\rho^{*}}$, we define $\overline{s}N:=\langle s(0), s(1), \dots,  s(N-1)\rangle $ for $N^{0}<|s|$.  
For a sequence $\alpha^{0\di \rho}$, we also write $\overline{\alpha}N=\langle \alpha(0), \alpha(1),\dots, \alpha(N-1)\rangle$ for \emph{any} $N^{0}$.  By way of shorthand, 
$(\forall q^{\rho}\in Q^{\rho^{*}})A(q)$ abbreviates $(\forall i^{0}<|Q|)A(Q(i))$, which is (equivalent to) quantifier-free if $A$ is.   
\end{nota}
\subsubsection{Some axioms and functionals}\label{lll}
As noted in Section \ref{intro}, the logical hardness of a theorem is measured via what fragment of the comprehension axiom is needed for a proof.  
For this reason, we introduce some axioms and functionals related to \emph{higher-order comprehension} in this section.

\smallskip

First of all, the following functional is clearly discontinuous at $f=11\dots$; in fact, $(\exists^{2})$ is equivalent to the existence of $F:\R\di\R$ such that $F(x)=1$ if $x>_{\R}0$, and $0$ otherwise (\cite{kohlenbach2}*{\S3}).  This fact shall be repeated often.  
\be\label{muk}\tag{$\exists^{2}$}
(\exists \varphi^{2}\leq_{2}1)(\forall f^{1})\big[(\exists n)(f(n)=0) \asa \varphi(f)=0    \big]. 
\ee
Related to $(\exists^{2})$, the functional $\mu^{2}$ in $(\mu^{2})$ is also called \emph{Feferman's $\mu$} (\cite{avi2}).
\begin{align}\label{mu}\tag{$\mu^{2}$}
(\exists \mu^{2})(\forall f^{1})\big[ (\exists n)(f(n)=0) \di [f(\mu(f))=0&\wedge (\forall i<\mu(f))(f(i)\ne 0) ]\\
& \wedge [ (\forall n)(f(n)\ne0)\di   \mu(f)=0]    \big], \notag
\end{align}
We have $(\exists^{2})\asa (\mu^{2})$ over $\RCAo$ and $\ACAo\equiv\RCAo+(\exists^{2})$ proves the same sentences as $\ACA_{0}$ by \cite{hunterphd}*{Theorem~2.5}. 

\smallskip

Secondly, the functional $\SS^{2}$ in $(\SS^{2})$ is called \emph{the Suslin functional} (\cite{kohlenbach2}).
\be\tag{$\SS^{2}$}
(\exists\SS^{2}\leq_{2}1)(\forall f^{1})\big[  (\exists g^{1})(\forall n^{0})(f(\overline{g}n)=0)\asa \SS(f)=0  \big], 
\ee
The system $\FIVE^{\omega}\equiv \RCAo+(\SS^{2})$ proves the same $\Pi_{3}^{1}$-sentences as $\FIVE$ by \cite{yamayamaharehare}*{Theorem 2.2}.   
By definition, the Suslin functional $\SS^{2}$ can decide whether a $\Sigma_{1}^{1}$-formula as in the left-hand side of $(\SS^{2})$ is true or false.   We similarly define the functional $\SS_{k}^{2}$ which decides the truth or falsity of $\Sigma_{k}^{1}$-formulas from $\L_{2}$; we also define 
the system $\SIXK$ as $\RCAo+(\SS_{k}^{2})$, where  $(\SS_{k}^{2})$ expresses that $\SS_{k}^{2}$ exists.  
We note that the operators $\nu_{n}$ from \cite{boekskeopendoen}*{p.\ 129} are essentially $\SS_{n}^{2}$ strengthened to return a witness (if existant) to the $\Sigma_{n}^{1}$-formula at hand.  

\smallskip

\noindent
Thirdly, full second-order arithmetic $\Z_{2}$ is readily derived from $\cup_{k}\SIXK$, or from:
\be\tag{$\exists^{3}$}
(\exists E^{3}\leq_{3}1)(\forall Y^{2})\big[  (\exists f^{1})(Y(f)=0)\asa E(Y)=0  \big], 
\ee
and we therefore define $\Z_{2}^{\Omega}\equiv \RCAo+(\exists^{3})$ and $\Z_{2}^\omega\equiv \cup_{k}\SIXK$, which are conservative over $\Z_{2}$ by \cite{hunterphd}*{Cor.\ 2.6}. 
Despite this close connection, $\Z_{2}^{\omega}$ and $\Z_{2}^{\Omega}$ can behave quite differently, as discussed in e.g.\ \cite{dagsamIII}*{\S2.2}.   
The functional from $(\exists^{3})$ is also called `$\exists^{3}$', and we use the same convention for other functionals.  

\subsection{Higher-order computability theory}\label{HCT}

\appendix

\subsection{Reverse Mathematics}\label{prelim1}
We discuss Reverse Mathematics (Section \ref{introrm}) and introduce -in full detail- Kohlenbach's base theory of \emph{higher-order} Reverse Mathematics (Section \ref{rmbt}).
Some essential axioms, functionals, and notations may be found in Sections \ref{kkk} and \ref{lll}.
\subsubsection{Introduction}\label{introrm}
Reverse Mathematics (RM hereafter) is a program in the foundations of mathematics initiated around 1975 by Friedman (\cites{fried,fried2}) and developed extensively by Simpson (\cite{simpson2}).  
The aim of RM is to identify the minimal axioms needed to prove theorems of ordinary, i.e.\ non-set theoretical, mathematics. 

\smallskip

We refer to \cite{stillebron} for a basic introduction to RM and to \cite{simpson2, simpson1} for an overview of RM.  We expect basic familiarity with RM, but do sketch some aspects of Kohlenbach's \emph{higher-order} RM (\cite{kohlenbach2}) essential to this paper, including the base theory $\RCAo$ (Definition \ref{kase}).  

\smallskip

First of all, in contrast to `classical' RM based on \emph{second-order arithmetic} $\Z_{2}$, higher-order RM uses $\L_{\omega}$, the richer language of \emph{higher-order arithmetic}.  
Indeed, while the former is restricted to natural numbers and sets of natural numbers, higher-order arithmetic can accommodate sets of sets of natural numbers, sets of sets of sets of natural numbers, et cetera.  
To formalise this idea, we introduce the collection of \emph{all finite types} $\mathbf{T}$, defined by the two clauses:
\begin{center}
(i) $0\in \mathbf{T}$   and   (ii)  If $\sigma, \tau\in \mathbf{T}$ then $( \sigma \di \tau) \in \mathbf{T}$,
\end{center}
where $0$ is the type of natural numbers, and $\sigma\di \tau$ is the type of mappings from objects of type $\sigma$ to objects of type $\tau$.
In this way, $1\equiv 0\di 0$ is the type of functions from numbers to numbers, and  $n+1\equiv n\di 0$.  Viewing sets as given by characteristic functions, we note that $\Z_{2}$ only includes objects of type $0$ and $1$.    

\smallskip

Secondly, the language $\L_{\omega}$ includes variables $x^{\rho}, y^{\rho}, z^{\rho},\dots$ of any finite type $\rho\in \mathbf{T}$.  Types may be omitted when they can be inferred from context.  
The constants of $\L_{\omega}$ include the type $0$ objects $0, 1$ and $ <_{0}, +_{0}, \times_{0},=_{0}$  which are intended to have their usual meaning as operations on $\N$.
Equality at higher types is defined in terms of `$=_{0}$' as follows: for any objects $x^{\tau}, y^{\tau}$, we have
\be\label{aparth}
[x=_{\tau}y] \equiv (\forall z_{1}^{\tau_{1}}\dots z_{k}^{\tau_{k}})[xz_{1}\dots z_{k}=_{0}yz_{1}\dots z_{k}],
\ee
if the type $\tau$ is composed as $\tau\equiv(\tau_{1}\di \dots\di \tau_{k}\di 0)$.  
Furthermore, $\L_{\omega}$ also includes the \emph{recursor constant} $\mathbf{R}_{\sigma}$ for any $\sigma\in \mathbf{T}$, which allows for iteration on type $\sigma$-objects as in the special case \eqref{special}.  Formulas and terms are defined as usual.  
One obtains the sub-language $\L_{n+2}$ by restricting the above type formation rule to produce only type $n+1$ objects (and related types of similar complexity).        

\subsubsection{The base theory of higher-order Reverse Mathematics}\label{rmbt}
We introduce Kohlenbach's base theory $\RCAo$, first introduced in \cite{kohlenbach2}*{\S2}.
\bdefi\label{kase} 
The base theory $\RCAo$ consists of the following axioms.
\begin{enumerate}
 \renewcommand{\theenumi}{\alph{enumi}}
\item  Basic axioms expressing that $0, 1, <_{0}, +_{0}, \times_{0}$ form an ordered semi-ring with equality $=_{0}$.
\item Basic axioms defining the well-known $\Pi$ and $\Sigma$ combinators (aka $K$ and $S$ in \cite{avi2}), which allow for the definition of \emph{$\lambda$-abstraction}. 
\item The defining axiom of the recursor constant $\mathbf{R}_{0}$: for $m^{0}$ and $f^{1}$: 
\be\label{special}
\mathbf{R}_{0}(f, m, 0):= m \textup{ and } \mathbf{R}_{0}(f, m, n+1):= f(n, \mathbf{R}_{0}(f, m, n)).
\ee
\item The \emph{axiom of extensionality}: for all $\rho, \tau\in \mathbf{T}$, we have:
\be\label{EXT}\tag{$\textsf{\textup{E}}_{\rho, \tau}$}  
(\forall  x^{\rho},y^{\rho}, \varphi^{\rho\di \tau}) \big[x=_{\rho} y \di \varphi(x)=_{\tau}\varphi(y)   \big].
\ee 
\item The induction axiom for quantifier-free formulas of $\L_{\omega}$.
\item $\QFAC^{1,0}$: the quantifier-free Axiom of Choice as in Definition \ref{QFAC}.
\end{enumerate}
\edefi
\noindent
Note that variables (of any finite type) are allowed in quantifier-free formulas of the language $\L_{\omega}$: only quantifiers are banned.
Recursion as in \eqref{special} is called \emph{primitive recursion}; the class of functionals obtained from $\mathbf{R}_{\rho}$ for all $\rho \in \mathbf{T}$ is called \emph{G\"odel's system $T$} of all (higher-order) primitive recursive functionals. 
\bdefi\label{QFAC} The axiom $\QFAC$ consists of the following for all $\sigma, \tau \in \textbf{T}$:
\be\tag{$\QFAC^{\sigma,\tau}$}
(\forall x^{\sigma})(\exists y^{\tau})A(x, y)\di (\exists Y^{\sigma\di \tau})(\forall x^{\sigma})A(x, Y(x)),
\ee
for any quantifier-free formula $A$ in the language of $\L_{\omega}$.
\edefi
As discussed in \cite{kohlenbach2}*{\S2}, $\RCAo$ and $\RCA_{0}$ prove the same sentences `up to language' as the latter is set-based and the former function-based.   
This conservation results is obtained via the so-called $\ECF$-interpretation, which we now discuss. 
\begin{rem}[The $\ECF$-interpretation]\label{ECF}\rm
The (rather) technical definition of $\ECF$ may be found in \cite{troelstra1}*{p.\ 138, \S2.6}.
Intuitively, the $\ECF$-interpretation $[A]_{\ECF}$ of a formula $A\in \L_{\omega}$ is just $A$ with all variables 
of type two and higher replaced by type one variables ranging over so-called `associates' or `RM-codes' (see \cite{kohlenbach4}*{\S4}); the latter are (countable) representations of continuous functionals.  
The $\ECF$-interpretation connects $\RCAo$ and $\RCA_{0}$ (see \cite{kohlenbach2}*{Prop.\ 3.1}) in that if $\RCAo$ proves $A$, then $\RCA_{0}$ proves $[A]_{\ECF}$, again `up to language', as $\RCA_{0}$ is 
formulated using sets, and $[A]_{\ECF}$ is formulated using types, i.e.\ using type zero and one objects.  
\end{rem}
In light of the widespread use of codes in RM and the common practise of identifying codes with the objects being coded, it is no exaggeration to refer to $\ECF$ as the \emph{canonical} embedding of higher-order into second-order arithmetic. 

\smallskip

Finally as noted above, Theorem \ref{cant2} is provable in the base theory. 
\begin{thm}[\cite{simpson2}*{II.4.9}] The following is provable in $\RCA_{0}$.
For any sequence of real numbers $(x_{n})_{n\in \N}$, there is a real $y$ different from $x_{n}$ for all $n\in \N$.
\end{thm}

\subsubsection{Notations and the like}\label{kkk}
We introduce the usual notations for common mathematical notions, like real numbers, as also introduced in \cite{kohlenbach2}.  
\begin{defi}[Real numbers and related notions in $\RCAo$]\label{keepintireal}\rm~
\begin{enumerate}
 \renewcommand{\theenumi}{\alph{enumi}}
\item Natural numbers correspond to type zero objects, and we use `$n^{0}$' and `$n\in \N$' interchangeably.  Rational numbers are defined as signed quotients of natural numbers, and `$q\in \Q$' and `$<_{\Q}$' have their usual meaning.    
\item Real numbers are coded by fast-converging Cauchy sequences $q_{(\cdot)}:\N\di \Q$, i.e.\  such that $(\forall n^{0}, i^{0})(|q_{n}-q_{n+i}|<_{\Q} \frac{1}{2^{n}})$.  
We use Kohlenbach's `hat function' from \cite{kohlenbach2}*{p.\ 289} to guarantee that every $q^{1}$ defines a real number.  
\item We write `$x\in \R$' to express that $x^{1}:=(q^{1}_{(\cdot)})$ represents a real as in the previous item and write $[x](k):=q_{k}$ for the $k$-th approximation of $x$.    
\item Two reals $x, y$ represented by $q_{(\cdot)}$ and $r_{(\cdot)}$ are \emph{equal}, denoted $x=_{\R}y$, if $(\forall n^{0})(|q_{n}-r_{n}|\leq {2^{-n+1}})$. Inequality `$<_{\R}$' is defined similarly.  
We sometimes omit the subscript `$\R$' if it is clear from context.           
\item Functions $F:\R\di \R$ are represented by $\Phi^{1\di 1}$ mapping equal reals to equal reals, i.e.\ extensionality as in $(\forall x , y\in \R)(x=_{\R}y\di \Phi(x)=_{\R}\Phi(y))$.\label{EXTEN}
\item The relation `$x\leq_{\tau}y$' is defined as in \eqref{aparth} but with `$\leq_{0}$' instead of `$=_{0}$'.  Binary sequences are denoted `$f^{1}, g^{1}\leq_{1}1$', but also `$f,g\in C$' or `$f, g\in 2^{\N}$'.  Elements of Baire space are given by $f^{1}, g^{1}$, but also denoted `$f, g\in \N^{\N}$'.
\item For a binary sequence $f^{1}$, the associated real in $[0,1]$ is $\r(f):=\sum_{n=0}^{\infty}\frac{f(n)}{2^{n+1}}$.\label{detrippe}
\item Sets of type $\rho$ objects $X^{\rho\di 0}, Y^{\rho\di 0}, \dots$ are given by their characteristic functions $F^{\rho\di 0}_{X}\leq_{\rho\di 0}1$, i.e.\ we write `$x\in X$' for $ F_{X}(x)=_{0}1$. \label{koer} 
\end{enumerate}
\end{defi}
For completeness, we list the following notational convention for finite sequences.  
\begin{nota}[Finite sequences]\label{skim}\rm
The type for `finite sequences of objects of type $\rho$' is denoted $\rho^{*}$, which we shall only use for $\rho=0,1$.  
Since the usual coding of pairs of numbers goes through in $\RCAo$, we shall not always distinguish between $0$ and $0^{*}$. 
Similarly, we assume a fixed coding for finite sequences of type $1$ and shall make use of the type `$1^{*}$'.  
In general, we do not always distinguish between `$s^{\rho}$' and `$\langle s^{\rho}\rangle$', where the former is `the object $s$ of type $\rho$', and the latter is `the sequence of type $\rho^{*}$ with only element $s^{\rho}$'.  The empty sequence for the type $\rho^{*}$ is denoted by `$\langle \rangle_{\rho}$', usually with the typing omitted.  

\smallskip

Furthermore, we denote by `$|s|=n$' the length of the finite sequence $s^{\rho^{*}}=\langle s_{0}^{\rho},s_{1}^{\rho},\dots,s_{n-1}^{\rho}\rangle$, where $|\langle\rangle|=0$, i.e.\ the empty sequence has length zero.
For sequences $s^{\rho^{*}}, t^{\rho^{*}}$, we denote by `$s*t$' the concatenation of $s$ and $t$, i.e.\ $(s*t)(i)=s(i)$ for $i<|s|$ and $(s*t)(j)=t(|s|-j)$ for $|s|\leq j< |s|+|t|$. For a sequence $s^{\rho^{*}}$, we define $\overline{s}N:=\langle s(0), s(1), \dots,  s(N-1)\rangle $ for $N^{0}<|s|$.  
For a sequence $\alpha^{0\di \rho}$, we also write $\overline{\alpha}N=\langle \alpha(0), \alpha(1),\dots, \alpha(N-1)\rangle$ for \emph{any} $N^{0}$.  By way of shorthand, 
$(\forall q^{\rho}\in Q^{\rho^{*}})A(q)$ abbreviates $(\forall i^{0}<|Q|)A(Q(i))$, which is (equivalent to) quantifier-free if $A$ is.   
\end{nota}
\subsubsection{Some axioms and functionals}\label{lll}
As noted in Section \ref{intro}, the logical hardness of a theorem is measured via what fragment of the comprehension axiom is needed for a proof.  
For this reason, we introduce some axioms and functionals related to \emph{higher-order comprehension} in this section.
We are mostly dealing with \emph{conventional} comprehension here, i.e.\ only parameters over $\N$ and $\N^{\N}$ are allowed in formula classes like $\Pi_{k}^{1}$ and $\Sigma_{k}^{1}$.  

\smallskip

First of all, the following functional is clearly discontinuous at $f=11\dots$; in fact, $(\exists^{2})$ is equivalent to the existence of $F:\R\di\R$ such that $F(x)=1$ if $x>_{\R}0$, and $0$ otherwise (\cite{kohlenbach2}*{\S3}).  This fact shall be repeated often.  
\be\label{muk}\tag{$\exists^{2}$}
(\exists \varphi^{2}\leq_{2}1)(\forall f^{1})\big[(\exists n)(f(n)=0) \asa \varphi(f)=0    \big]. 
\ee
Related to $(\exists^{2})$, the functional $\mu^{2}$ in $(\mu^{2})$ is also called \emph{Feferman's $\mu$} (\cite{avi2}).
\begin{align}\label{mu}\tag{$\mu^{2}$}
(\exists \mu^{2})(\forall f^{1})\big[ (\exists n)(f(n)=0) \di [f(\mu(f))=0&\wedge (\forall i<\mu(f))(f(i)\ne 0) ]\\
& \wedge [ (\forall n)(f(n)\ne0)\di   \mu(f)=0]    \big], \notag
\end{align}
We have $(\exists^{2})\asa (\mu^{2})$ over $\RCAo$ and $\ACAo\equiv\RCAo+(\exists^{2})$ proves the same sentences as $\ACA_{0}$ by \cite{hunterphd}*{Theorem~2.5}. 

\smallskip

Secondly, the functional $\SS^{2}$ in $(\SS^{2})$ is called \emph{the Suslin functional} (\cite{kohlenbach2}).
\be\tag{$\SS^{2}$}
(\exists\SS^{2}\leq_{2}1)(\forall f^{1})\big[  (\exists g^{1})(\forall n^{0})(f(\overline{g}n)=0)\asa \SS(f)=0  \big], 
\ee
The system $\FIVE^{\omega}\equiv \RCAo+(\SS^{2})$ proves the same $\Pi_{3}^{1}$-sentences as $\FIVE$ by \cite{yamayamaharehare}*{Theorem 2.2}.   
By definition, the Suslin functional $\SS^{2}$ can decide whether a $\Sigma_{1}^{1}$-formula as in the left-hand side of $(\SS^{2})$ is true or false.   We similarly define the functional $\SS_{k}^{2}$ which decides the truth or falsity of $\Sigma_{k}^{1}$-formulas from $\L_{2}$; we also define 
the system $\SIXK$ as $\RCAo+(\SS_{k}^{2})$, where  $(\SS_{k}^{2})$ expresses that $\SS_{k}^{2}$ exists.  
We note that the operators $\nu_{n}$ from \cite{boekskeopendoen}*{p.\ 129} are essentially $\SS_{n}^{2}$ strengthened to return a witness (if existant) to the $\Sigma_{n}^{1}$-formula at hand.  

\smallskip

\noindent
Thirdly, full second-order arithmetic $\Z_{2}$ is readily derived from $\cup_{k}\SIXK$, or from:
\be\tag{$\exists^{3}$}
(\exists E^{3}\leq_{3}1)(\forall Y^{2})\big[  (\exists f^{1})(Y(f)=0)\asa E(Y)=0  \big], 
\ee
and we therefore define $\Z_{2}^{\Omega}\equiv \RCAo+(\exists^{3})$ and $\Z_{2}^\omega\equiv \cup_{k}\SIXK$, which are conservative over $\Z_{2}$ by \cite{hunterphd}*{Cor.\ 2.6}. 
Despite this close connection, $\Z_{2}^{\omega}$ and $\Z_{2}^{\Omega}$ can behave quite differently, as discussed in e.g.\ \cite{dagsamIII}*{\S2.2}.   
The functional from $(\exists^{3})$ is also called `$\exists^{3}$', and we use the same convention for other functionals.  

\bye